\documentclass[a4paper,11pt]{amsart}

\usepackage[utf8]{inputenc}
\usepackage[T1]{fontenc}

\usepackage[english]{babel}

\usepackage{tikz}
\usetikzlibrary{calc}

\newcommand{\currentsidemargin}{%
	\ifodd\value{page}%
	\oddsidemargin%
	\else%
	\evensidemargin%
	\fi%
}

\newlength{\whatsleft}

\usepackage{comment}

\usepackage[margin=2cm]{geometry}

\usepackage{amsmath}
\usepackage{amsthm}
\usepackage{amssymb}
\usepackage[all]{xy} 
\usepackage{wasysym}
\usepackage{mathtools}
\usepackage{tensor}
\usepackage{amsfonts}
	
\usepackage{bbm}
\usepackage{upgreek} 

\usepackage{enumitem}

\usepackage{xcolor} 
\definecolor{InternalLinks}{rgb}{0.33, 0.29, 0.31}

\usepackage{hyperref} 

\theoremstyle{plain}

\newtheorem{theo}{Theorem}[section]
\newtheorem{lem}[theo]{Lemma}
\newtheorem{propo}[theo]{Proposition}
\newtheorem{coro}[theo]{Corollary}

\theoremstyle{definition}

\newtheorem{exe}[theo]{Example}

\newtheorem{defi}[theo]{Definition}

\theoremstyle{remark}

\newtheorem{rem}[theo]{Remark}

\numberwithin{equation}{section}

\newcommand{\C}{\mathbb{C}}
\newcommand{\N}{\mathbb{N}}

\newcommand{\Z}{\mathbb{Z}}

\newcommand{\id}{\mathrm{id}}
\newcommand{\spa}{\mathrm{span}}
\newcommand{\iso}{\cong}

\newcommand{\com}{\Delta}
\newcommand{\cou}{\varepsilon}
\newcommand{\G}{\mathbb{G}}
\newcommand{\HH}{\mathbb{H}}
\newcommand{\irr}{\mathrm{Irr}}
\newcommand{\mor}{\mathrm{Mor}}
\newcommand{\pol}{\mathcal{O}}
\newcommand{\rep}{\mathrm{Rep}}
\newcommand{\triv}{\mathbbm{1}}
\DeclareMathOperator{\xbox}{\otimes}
\newcommand{\blhd}{\blacktriangleleft}

\newcommand{\nes}{\psi}

\newcommand{\CC}{\mathcal{C}}
\newcommand{\LL}{\mathcal{L}}
\newcommand{\PP}{\mathcal{P}}
\DeclareMathOperator{\Proj}{Proj}
\newcommand{\rl}{\mathrm{rl}}

\newcommand{\B}{\mathrm{B}}
\newcommand{\BB}{\mathcal{B}}

\newcommand{\F}{\varphi}
\newcommand{\hil}{\mathrm{Hilb}}
\DeclareMathOperator{\Ind}{Ind}
\DeclareMathOperator{\Res}{Res}

\DeclareFontFamily{U}{MnSymbolC}{}
\DeclareFontShape{U}{MnSymbolC}{m}{n}{
	<-6>  MnSymbolC5
	<6-7>  MnSymbolC6
	<7-8>  MnSymbolC7
	<8-9>  MnSymbolC8
	<9-10> MnSymbolC9
	<10-12> MnSymbolC10
	<12->   MnSymbolC12}{}
\DeclareFontShape{U}{MnSymbolC}{b}{n}{
	<-6>  MnSymbolC-Bold5
	<6-7>  MnSymbolC-Bold6
	<7-8>  MnSymbolC-Bold7
	<8-9>  MnSymbolC-Bold8
	<9-10> MnSymbolC-Bold9
	<10-12> MnSymbolC-Bold10
	<12->   MnSymbolC-Bold12}{}

\DeclareSymbolFont{MnSyC}{U}{MnSymbolC}{m}{n}
\DeclareMathSymbol{\boxvert}{\mathbin}{MnSyC}{113}

\usepackage{calc}
\newcounter{PartitionDepth}
\newcounter{PartitionLength}

\newcommand{\partii}[3]{
	\begin{picture}(#3,#1)
	\setcounter{PartitionLength}{#3-#2}
	\setcounter{PartitionDepth}{-1-#1}
	\put(#2,\thePartitionDepth){\line(0,1){#1}}
	\put(#3,\thePartitionDepth){\line(0,1){#1}}
	\put(#2,\thePartitionDepth){\line(1,0){\thePartitionLength}}
	\end{picture}}

\newcommand{\uppartii}[3]{
	\begin{picture}(#3,#1)
	\setcounter{PartitionLength}{#3-#2}
	\setcounter{PartitionDepth}{#1}
	\put(#2,0){\line(0,1){#1}}
	\put(#3,0){\line(0,1){#1}}
	\put(#2,\thePartitionDepth){\line(1,0){\thePartitionLength}}
	\end{picture}}

\makeatletter
\newcommand{\addresseshere}{%
\enddoc@text\let\enddoc@text\relax
}
\makeatother
\makeatletter
\newcommand{\pushright}[1]{\ifmeasuring@#1\else\omit\hfill$\displaystyle#1$\fi\ignorespaces}
\newcommand{\pushleft}[1]{\ifmeasuring@#1\else\omit$\displaystyle#1$\hfill\fi\ignorespaces}
\makeatother

\title{Tannaka-Krein reconstruction and ergodic actions of easy quantum groups}
\date{}

\author{Amaury Freslon}
\address{Universit\'e Paris-Saclay, CNRS, Laboratoire de Math\'ematiques d'Orsay, 91405 Orsay, France}
\email{amaury.freslon@universite-paris-saclay.fr}

\author{Frank Taipe}
\address{Universit\'e Paris-Saclay, CNRS, Laboratoire de Math\'ematiques d'Orsay, 91405 Orsay, France}
\email{frank.taipe@universite-paris-saclay.fr}

\author{Simeng Wang}
\address{Institute for Advanced Study in Mathematics, Harbin Institute of Technology, Harbin 150001, China}
\email{simeng.wang@hit.edu.cn}


\begin{document}

\begin{abstract}
We give a new alternative version of the reconstruction procedure for ergodic actions of compact quantum groups and we refine it to include characterizations of (braided commutative) Yetter-Drinfeld C*-algebras. We then use this to construct families of ergodic actions of easy quantum groups out of combinatorial data involving partitions and study them. Eventually, we use this categorical point of view to show that the quantum permutation group cannot act ergodically on a classical connected compact space, thereby answering a question of D. Goswami and H. Huang.
\end{abstract}

\maketitle

\section{Introduction}

Compact quantum groups were introduced by S. L. Woronowicz in the 1980's \cite{W87}. Even though they were originally designed as an extension of compact groups in the setting of noncommutative geometry, it appeared right away that the definition allows for a great deal of different behaviors since for instance compact quantum groups also cover the theory of discrete groups. Moreover, it was soon established by S. L. Woronowicz in \cite{W88} that compact quantum groups can equivalently be described in categorical terms through a generalization of the Tannaka-Krein duality theorem. This proved useful in constructing and studying examples. In particular, T. Banica and R. Speicher used in \cite{banica2009liberation} the Tannaka-Krein approach to construct compact quantum groups out of combinatorial data called categories of partitions. This seminal work has led to many interesting connections between quantum groups, combinatorics and free probability.

Just like groups, quantum groups can act on ``quantum spaces'', that is to say on C*-algebras. Because of the importance of group actions in the modern theory of groups, it was natural to investigate actions of compact quantum groups, and this was done in the early days of the theory. In particular, F. Boca generalized in \cite{boca1995ergodic} the work of A. Wassermann on the spectral decomposition of actions \cite{wassermann1989ergodic}. But based on the previous paragraph, it is natural to wonder how to see actions through the Tannaka-Krein framework. An answer was given by C. Pinzari and J. E. Roberts in \cite{PR07} for actions which are \emph{ergodic} (see below for the definitions). Just like compact quantum groups arise from tensor functors to Hilbert spaces on specific monoidal categories, ergodic actions arise from ``not exactly tensor'' functors, called weak unitary tensor functors. This was later generalized to arbitrary actions by S. Neshveyev in \cite{N14}, where Hilbert spaces have to be replaced with Hilbert C*-modules. Moreover, it is also natural to consider more enriched dynamic structures such as Yetter-Drinfeld C*-algebras, which are closely related to the study of noncommutative Poisson boundaries \cite{NY14}, equivariant KK-theory for quantum groups \cite{NV09} and quantum transformation groupoids \cite{taipe21}. For example, a categorical duality for Yetter-Drinfeld C*-algebras was done by S. Neshveyev and M. Yamashita in \cite{NY14}, with the use of module C*-categories.

The purpose of this work is to extend the combinatorial approach of Banica and Speicher to the setting of actions. This means that in the same way as they used partitions to product tensor functors on C*-tensor categories, we will use them to produce functors that enable the reconstruction of an ergodic action or even a Yetter-Drinfeld structure. Two difficulties immediately arise. First, partitions only give information on some part of the C*-tensor category, namely that given by tensor powers of the fundamental representation. The strategy is therefore to build functors on that part and check that there is a unique way of extending them. We therefore start by proving a reconstruction theorem of that kind.

The second difficulty is that there is no canonical way of produce a weak unitary tensor functor from partitions. More precisely, we describe two general methods of doing this. In both cases, we give several families of examples which are of very different nature. Indeed, one can recover actions on finite spaces, on embedded homogeneous spaces and on quotient spaces from these. Moreover, many of the actions constructed can be naturally described from the aforementioned ones through an induction procedure that we detail.

The examples illustrate the power of the functorial approach to ergodic actions of easy quantum groups. As another illustration, we use categorical techniques to answer a question of D. Goswami and H. Huang on quantum rigidity of classical spaces (see \cite[Question 3.12]{huang2013faithful}). Namely, we show that the quantum permutation group $S_{N}^{+}$ cannot act ergodically on a compact connected space, unless the latter is a point. Our methods also work for many other free quantum groups such as the quantum orthogonal group $O_N^+$ and the quantum refection group $H_N^+$. 

We will now briefly outline the organization of this work. After some preliminaries in Section \ref{sec:preliminaries}, whose main purpose is to fix notations and recall some important facts which will be used afterwards, we state and prove in Section \ref{sec:duality} our Tannaka-Krein duality theorem for ergodic actions. For clarity, the proof is split into two parts, one concerning the construction of the functor and one concerning the explicit description of the C*-algebra obtained from that functor. We then give in a separate statement a characterization of the existence of a (braided commutative) Yetter-Drinfeld structure in terms of the weak unitary tensor functor. This relies on a corresponding extension of the reconstruction theorem of C. Pinzari and J. E. Roberts which we give in an appendix.

Then, Section \ref{sec:examples} deals with applications of our theorem to easy quantum groups. After recalling a few additional facts on the combinatorics of partitions, we produce several families of actions through the Tannaka-Krein reconstruction theorem. We identify some of them and postpone the description of some others because they require some results on induction. In Section \ref{sec:induction}, we describe the induction procedure for compact quantum groups. This is already known to experts, but such matters are usually treated in the general framework of locally compact quantum groups. We therefore believe that a self-contained treatment of the compact case can be of use, all the more that it is quite elementary. Eventually, in Section \ref{sec:rigidity} we use the reconstruction theorem to prove that the quantum permutation group $S_{N}^{+}$ cannot act ergodically on a non-trivial connected compact space. The conjecture that a genuine compact quantum group cannot act faithfully on a connected classical compact space was first formulated by D. Goswami, based upon strong evidence from the geometric setting. However, H. Huang showed in \cite{huang2013faithful} that there are faithful actions of the quantum permutation group on connected spaces. Here therefore refined the conjecture by requiring the action to be ergodic, and we prove that conjecture for $S_N^+$ as well as for other free easy quantum groups.

\subsection*{Conventions} The inner product in any Hilbert space will be always linear in the first component.

\subsection*{Acknowledgments}

A.F. and S.W. were partially funded by the ANR grant ``Noncommutative analysis on groups and quantum groups'' (ANR-19-CE40-0002) and the PHC Polonium ``Quantum structures and processes'', A.F. was also partially funded by the ANR grant ``Operator algebras and dynamics on groups'' (ANR-19-CE40-0008), the PHC Procope ``Quantum groups and quantum probability'' and the PHC Van Gogh ``Quantum groups, harmonic analysis and quantum probability''. S.W. was also partially supported by the Fundamental
Research Funds for the Central Universities No. FRFCUAUGA5710012222, the NSF of China No.
12031004 and a public grant as part of the Fondation Mathématique Jacques Hadamard.

\subsection*{Conflict of interests} The authors have no conflicts of interest to declare and there is no financial interest to report.
\subsection*{Data availability} Data sharing not applicable to this article as no datasets were generated or analyzed during the current study.

\section{Preliminaries}\label{sec:preliminaries}

In this section we gather the necessary material concerning the objects that we will study. We will not give any proofs, the aim being mainly to fix the notations and vocabulary.

\subsection{Compact quantum groups and their representations}

This work is concerned with compact quantum groups in the following sense,

\begin{defi}
A \emph{compact quantum group} is a pair $(C(\G), \com)$ where $C(\G)$ is a C*-algebra and
\begin{equation*}
\com : C(\G)\to C(\G)\otimes C(\G)
\end{equation*}
is a unital $*$-homomorphism satisfying the following two conditions:
\begin{enumerate}
\item $(\com\otimes \id)\circ\com = (\id\otimes \com)\circ\com$;
\item $\overline{\spa}\{\com(C(\G))(1\otimes C(\G))\} = C(\G)\otimes C(\G) = \overline{\spa}\{\com(C(\G))(C(\G)\otimes 1)\}$.
\end{enumerate}
We use the notation $\G = (C(\G), \com)$ to denote the compact quantum group.
\end{defi}

We refer the reader to the book \cite{NT13} for a detailed treatment of the theory of compact quantum groups, we will here simply give the definitions and results which will be needed in the sequel. We will freely use the usual \emph{leg numbering notations}: if $X$ is an element of a tensor product of two spaces, then $X_{(12)}$, $X_{(23)}$ and $X_{(13)}$ are the natural extensions to a triple tensor product spaces acting on the two tensors indicated as subscripts.

Let $\G = (C(\G),\com)$ be a compact quantum group. A {\em representation} of $\G$ on a finite-dimensional Hilbert space $H$ is an invertible element $u\in \B(H)\otimes C(\G)$ such that $(\id\otimes\com)(u) = u_{(12)}u_{(13)}$, and we may write $H = H_{u}$ when there is a risk of confusion. Note that the representations considered in this work will always be finite-dimensional unless specified otherwise. The simplest example is the \emph{trivial representation} $\triv = 1\otimes 1\in \B(\C)\otimes C(\G)$. If $(e_{i})_{i\in I}$ is an orthonormal basis of $H$, then the matrix coefficients of $u$ are
\begin{equation*}
u_{ij} = (\omega_{ij}\otimes \mathrm{id})u \in C(\G),
\end{equation*}
where $\omega_{ij} = \langle \;\cdot\;e_{j}, e_{i}\rangle$. We can then equivalently view $u$ as a coaction
\begin{equation*}
\delta_{u} :  
H  \to  H\otimes C(\mathbb G) ,\quad 
e_{j}  \mapsto   \displaystyle\sum_{i\in I} e_{i}\otimes u_{ij}.
\end{equation*}
The standard operations on representations are available in this quantum setting. In particular, the \emph{conjugate} of $u$ is defined as $\bar{u} = [u_{ij}^*]_{i,j \in I} \in \mathrm B(\bar{H})
\otimes C(\G)$ and is again a representation, and we will often be interested in the \emph{self-conjugate} case where $u_{ij} = u_{ij}^*$ for all $i,j \in I$ and where we write $u=\bar{u}$ by identifying $H$ and $\bar H$ in the canonical way; while the tensor product of $u$ and $v$ is defined as
\begin{equation*}
u\boxtimes v = u_{(13)}v_{(23)}\in\B(H_{u}\otimes H_{v})\otimes C(\G).
\end{equation*}
Because we will work in a categorical framework, the morphisms of representations will play a prominent role. They are defined as follows:
\begin{align*}
\mor_{\G}(u ,v) & := \left\{T\in \B(H_{u}, H_{v}) \mid v(T\otimes \id) = (T \otimes \id)u \right\} \\
& \phantom{:}= \left\{T\in \B(H_{u}, H_{v}) \mid \delta_{v}\circ T = (T\otimes \id)\circ\delta_{u}\right\}.
\end{align*} 
These are usually called the \emph{intertwiners} between $u$ and $v$. We will say that a unitary representation $u$ is {\em irreducible} if $\mor_{\G}(u,u) = \C\id_{H _{u}}$ and we will denote by $\rep(\G)$ the category of all finite-dimensional unitary representations of $\G$ with intertwiners as morphisms. We denote by $\irr(\G)$ the set of equivalence classes of finite-dimensional irreducible unitary representations of $\G$ and we fix, for each $x\in \irr(\G)$, a representative $u^{x}\in\B(H_{x})\otimes C(\G)$.

\subsection{Ergodic actions}

The central notion in the present work is that of an action of a compact quantum group on a C*-algebra.
 
\begin{defi}\label{def:action}
A {\em continuous action of a compact quantum group} $\G = (C(\G), \com)$ on a unital C*-algebra $B$ is a unital $*$-homomorphism $\alpha : B \to B \otimes C(\G)$ satisfying the following two conditions:
\begin{enumerate}
\item $(\alpha\otimes \id)\circ\alpha = (\id\otimes \com)\circ\alpha$;
\item $\overline{\spa}\{(1\otimes C(\G) )\alpha(B)\} = B \otimes C(\G)$.
\end{enumerate}
The corresponding {\em fixed point subalgebra} is the C*-subalgebra
\begin{equation*}
B^{\G} = \{b\in B \mid \alpha(b) = b\otimes 1\}
\end{equation*}
and the action $\alpha$ is called {\em ergodic} if $B^{\G} = \C 1$.
\end{defi} 

Because we are only interested in continuous actions, we will drop the word ``continuous'' from now on. The reader may refer to \cite{DC17} for a comprehensive treatment of actions of compact quantum groups and proofs of the results used hereafter. The notion of intertwiner extends to actions by setting, for a representation $u$ of $\G$,
\begin{equation*}
\mor_{\G}(u, \alpha) := \left\{T : H_{u}\to B \mid \alpha\circ T = (T\otimes\id)\circ\delta_{u}\right\}.
\end{equation*}
This space admits a natural Hilbert space structure if $\alpha$ is ergodic, given by
\begin{equation*}
\langle T, S\rangle = \sum_{i\in I} T(e_{i})S (e_{i})^{*}\in B^{\G} = \C 1
\end{equation*}
and this defines a suitable inner product. It is sometimes convenient to use an alternative description of the morphism spaces which we now describe. Observe that $\delta_{u}$ and $\alpha$ induce a coaction on $H_{u}\otimes B$ through the formula
\begin{equation*}
\delta_{u}\boxtimes \alpha : \xi\otimes b\mapsto \delta_{u}(\xi)_{(13)}\alpha(b)_{(23)}.
\end{equation*}
It is then easily checked that the map $T\mapsto \sum_{i\in I} e_{i}\otimes T(e_{i})$ yields an isomorphism between $\mor_{\G}(u, \alpha)$ and
\begin{equation*}
(H_{u}\otimes B)^{\G} := \{X\in H_{u}\otimes B \mid (\delta_{u}\boxtimes\alpha)(X) = X\otimes 1\}.
\end{equation*}
The \emph{spectral subspace} of $B$ associated with $u$ is
\begin{equation*}
B_{u} := \left\{T(\xi) \mid T\in \mor_{\G}(u, \alpha), \xi \in H_{u}\right\}.
\end{equation*}
For any $x\in \irr(\G)$, we write $B_{x}$ for $B_{u^{x}}$ and the map $T\otimes \xi\mapsto T(\xi)$ yields an isomorphism
\begin{equation*}
\mor_{\G}(u^{x},\alpha)\otimes H_{x}\cong B_{x}.
\end{equation*}
Given a basis $(T_{i})_{i\in I'}$ of $\mor_{\G}(u, \alpha)$ and writing $b_{ij} = T_{i}(e_{j})$, the restriction of $\alpha$ to $\mathcal{B}_{x}$ is simply given, for any $(i, j)\in I\times I'$, by
\begin{equation*}
\alpha(b_{ij}) = \sum_{k\in I} b_{ik}\otimes u^{x}_{kj}.
\end{equation*}
One crucial property of the spectral subspaces is that they enable to recover the whole structure of the action. This is because
\begin{equation*}
\BB := \bigoplus_{x \in \irr(\G)} \mathcal{B}_{x}
\end{equation*}
is a dense $*$-subalgebra of $B$, called the \emph{Podleś subalgebra} or \emph{algebraic core}, and $\alpha_{\mid\BB} : \BB \to \BB\otimes \pol(\G)$ is a coaction of the Hopf $*$-algebra $(\pol(\G), \com)$ of coefficients of finite-dimensional representations on $\G$. We can furthermore take completions of $\BB$ to recover a reduced and a universal version, involving the corresponding completions of $\pol(\G)$. In the present work,\emph{ we will always be considering the universal completions of $\BB$ and $\pol(\G)$, together with the corresponding action}.

The notion of quantum group action in Definition \ref{def:action} can be naturally extended to the general locally compact setting, covering in particular actions of the dual discrete quantum group $\widehat{\mathbb G}$ of $\mathbb G$. In this paper we will not define the object $\widehat{\mathbb G}$ since we do not need it. We will rather use the following well-known equivalent description of $\widehat{\mathbb G}$-actions in terms of $\pol(\G)$-actions:  

\begin{defi}\label{def:hopfaction}
Let $\G$ be a compact quantum group and let $\mathcal{B}$ be a unital $*$-algebra. We say that a linear map 
\[
\lhd: \mathcal{B} \otimes \pol(\G) \to \mathcal{B},\quad
b\otimes a \mapsto b\lhd a
\]
is an \emph{action of the Hopf $*$-algebra} $(\pol(\G), \com)$ on $\mathcal{B}$ (or say that $(\mathcal B , \lhd)$ is a \emph{$\pol(\G)$-module $*$-algebra}) if for all $b, b'\in\mathcal{B}$ and $a, a' \in \pol(\G)$, the following conditions hold:
\begin{enumerate}
\item $b\lhd (a  a') = (b \lhd a) \lhd a'$ and $b \lhd 1_{\pol(\G)} = b$;
\item $(b b') \lhd a = (b \lhd a_{(1)})(b' \lhd a_{(2)})$ and $1_{\mathcal{B}} \lhd a = \cou(a)1_{\mathcal{B}}$;
\item $(b \lhd a)^{*} = b^{*} \lhd S(a)^{*}$.
\end{enumerate}
\end{defi}
It is particularly important to consider the case where $\G$ and $\widehat{\G}$ act simultaneously on the same quantum space with several natural compatibility relations, such as the Yetter-Drinfeld structures. 
  
\begin{defi}[Yetter-Drinfeld C*-algebra]\label{def:bc_yd_algebra}
Let $\G$ be a compact quantum group and let $B$ be a C*-algebra endowed with a quantum group action $\alpha : B \to B \otimes C(\G)$ and a Hopf $*$-algebra action $\lhd : \mathcal{B}\otimes \pol(\G)\to \mathcal{B}$, where $\mathcal{B}$ denotes the algebraic core of $B$. The tuple $(B, \lhd, \alpha)$ is a {\em Yetter-Drinfeld $\G$-C*-algebra} if for all $a \in \pol(\G)$, $b \in \mathcal{B}$,
\begin{equation}\tag{YD}\label{eq:yetter_drinfeld_condition}
\alpha(b \lhd a) = (b_{[0]} \lhd a_{(2)}) \otimes S(a_{(1)})b_{[1]}a_{(3)},
\end{equation}
where we are using, for all $a \in \pol(\G)$, $b \in \mathcal{B}$, Sweedler's notations $\com(a) = a_{(1)} \otimes a_{(2)}$ and $\alpha(b) = b_{[0]} \otimes b_{[1]}$. A Yetter-Drinfeld $\G$-C*-algebra $(B,\lhd,\alpha)$ is said to be {\em braided commutative} if
\begin{equation}\tag{BC}\label{eq:braided_commutative_condition}
bb' = b'_{[0]}(b \lhd b'_{[1]})
\end{equation}
for all $b, b' \in \mathcal{B}$.
\end{defi}

\subsection{Easy quantum groups}

We will mainly focus on a class of compact quantum groups called \emph{easy quantum groups} which was introduced by T. Banica and R. Speicher in \cite{banica2009liberation}. This is a family of compact quantum groups built in a categorical way from combinatorial data expressed in terms of partitions. A \emph{partition} is given by two integers $k$ and $l$ and a partition $p$ of the set $\{1, \dots, k+l\}$ and we denote by $P$ the set of all partitions. It is convenient to picture such partitions as diagrams, in particular for computational purposes. A diagram consists in an upper row of $k$ points, a lower row of $l$ points, and some strings connecting these points if and only if they belong to the same subset of the partition. Let us consider for instance the partitions $p_{1} = \{\{1, 8\}, \{2, 6\}, \{3, 4\}, \{5, 7\}\}$ and $p_{2} = \{\{1, 4, 5, 6\}, \{2, 3\}\}$. If we see $p_{1}$ as an element of $P(3, 5)$ and $p_{2}$ as en element of $P(4, 2)$, then their diagram representations are
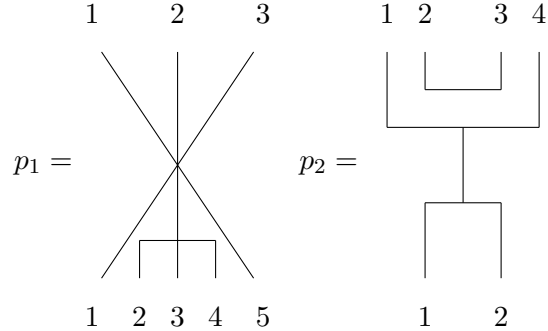
\begin{center}
\begin{tikzpicture}[scale=0.5]
\draw (0,-3) -- (0,3);
\draw (-1,-3) -- (-1,-2);
\draw (1,-3) -- (1,-2);
\draw (-1,-2) -- (1,-2);
\draw (-2,-3) -- (2,3);
\draw (2,-3) -- (-2,3);


\draw (-2.25,-3.5) node[below]{$1$};
\draw (-1,-3.5) node[below]{$2$};
\draw (0,-3.5) node[below]{$3$};
\draw (1,-3.5) node[below]{$4$};
\draw (2.25,-3.5) node[below]{$5$};


\draw (-2.25,3.5) node[above]{$1$};
\draw (0,3.5) node[above]{$2$};
\draw (2.25,3.5) node[above]{$3$};

\draw (-2.5,0) node[left]{$p_{1} = $};
\end{tikzpicture}
\begin{tikzpicture}[scale=0.5]
\draw (0,-1) -- (0,1);
\draw (-2,1) -- (2,1);
\draw (-2,1) -- (-2,3);
\draw (2,1) -- (2,3);
\draw (-1,2) -- (-1,3);
\draw (1,2) -- (1,3);
\draw (-1,2) -- (1,2);

\draw (-1,-1) -- (1,-1);
\draw (-1,-1) -- (-1,-3);
\draw (1,-1) -- (1,-3);


\draw (-1,-3.5) node[below]{$1$};
\draw (1,-3.5) node[below]{$2$};


\draw (-2,3.5) node[above]{$1$};
\draw (-1,3.5) node[above]{$2$};
\draw (1,3.5) node[above]{$3$};
\draw (2,3.5) node[above]{$4$};

\draw (-2.5,0) node[left]{$p_{2} = $};
\end{tikzpicture}
\end{center}

When manipulating partitions, the crucial notion is that of a block.

\begin{defi}
Let $p$ be a partition. A maximal set of points which are all connected (i.e.~one of the subsets defining the partition) is called a \emph{block} of $p$. If the block contains both upper and lower points (i.e.~the subset contains an element of $\{1, \dots, k\}$ and an element of $\{k+1, \dots, k+\ell\}$), then it is called a \emph{through-block}. Otherwise, it is called a \emph{non-through-block}. The total number of through-blocks of the partition $p$ is denoted by $t(p)$.
\end{defi}

In the following, we will pay particular attention to partitions which are non-crossing in the following sense.

\begin{defi}\label{de:noncrossing}
Let $p$ be a partition. A \emph{crossing} in $p$ is a tuple $k_{1} < k_{2} < k_{3} < k_{4}$ of integers such that
\begin{enumerate}
\item $k_{1}$ and $k_{3}$ are in the same block;
\item $k_{2}$ and $k_{4}$ are in the same block;
\item the four points are \emph{not} in the same block.
\end{enumerate}
If there is no crossing in $p$, then it is said to be  a \emph{non-crossing} partition. The set of non-crossing partitions will be denoted by $NC$.
\end{defi}

As an example, the partition $p_{1}$ above has a crossing, while the partition $p_{2}$ is non-crossing. There are several fundamental operations available on partitions called the \emph{category operations}:
\begin{enumerate}
\item If $p\in \CC(k, l)$ and $q\in \CC(m, n)$, then $p\odot q\in \CC (k+m, l+n)$ is their \emph{horizontal concatenation}, i.e.~the first $k$ of the $k+m$ upper points are connected by $p$ to the first $l$ of the $l+n$ lower points, whereas $q$ connects the remaining $m$ upper points with the remaining $n$ lower points. This is usually denoted by $p\otimes q$ but we will rather use the symbol $\odot$ to avoid confusion with tensor products.
\item If $p\in \CC(k, l)$ and $q\in \CC(l, m)$, then $qp\in \CC(k, m)$ is their \emph{vertical concatenation}, i.e.~$k$ upper points are connected by $p$ to $l$ middle points and the lines are then continued by $q$ to $m$ lower points. This process may produce loops in the partition. More precisely, consider the set $L$ of elements in $\{1, \dots, l\}$ which are not connected to an upper point of $p$ nor to a lower point of $q$. The lower row of $p$ and the upper row of $q$ both induce a partition of the set $L$. For $x, y\in L$, let us set $x\sim y$ if $x$ and $y$ belong either to the same block of the partition induced by $p$ or to the one induced by $q$. The transitive closure of $\sim$ is an equivalence relation on $L$ and the corresponding partition is called the \emph{loop partition} of $L$, its blocks are called \emph{loops} and their number is denoted by $\rl(q, p)$. To complete the operation, we remove all the loops.
\item If $p\in \CC(k, l)$, then $p^{*}\in \CC(l, k)$ is the partition obtained by reflecting $p$ with respect to an horizontal axis between the two rows.
\end{enumerate}

\begin{defi}
A \emph{category of partitions} $\CC$ is the data of a set of partitions $\CC(k, l)$ for all integers $k$ and $l$, which is stable under all the category operations and contains the identity partition $\vert\in P(1, 1)$.
\end{defi}

Such data gives rise to a compact quantum group through a natural way of associating linear maps with partitions.

\begin{defi}
Let $N$ be an integer and let $(e_{i})_{1\leqslant i\leqslant N}$ be the canonical basis of $\C^{N}$. For any partition $p\in P(k, l)$, we define a linear map
\begin{equation*}
T_{p}:(\C^{N})^{\otimes k} \mapsto (\C^{N})^{\otimes l}
\end{equation*}
by the following formula:
\begin{equation*}
T_{p}(e_{i_{1}} \otimes \dots \otimes e_{i_{k}}) = \sum_{j_{1}, \dots, j_{l} = 1}^{n} \delta_{p}(\mathbf{i},\mathbf{j})e_{j_{1}} \otimes \dots \otimes e_{j_{l}},
\end{equation*}
where $\delta_{p}(\mathbf{i},\mathbf{j}) = 1$ if all strings of the partition $p$ connect equal indices of the tuple $\mathbf{i}=(i_{1},\dots,i_{k})$ in the upper row with equal indices of the tuple $\mathbf{j}=(j_{1},\dots,j_{l})$ in the lower row, and $\delta_{p}(\mathbf{i},\mathbf{j}) = 0$ otherwise.
\end{defi}

The interplay between these maps and the category operations are given by the following rules proven in \cite[Prop 1.9]{banica2009liberation}:
\begin{enumerate}
\item $T_{p}^{*} = T_{p^{*}}$;
\item $T_{p}\otimes T_{q} = T_{p\otimes q}$;
\item $T_{p}\circ T_{q} = N^{\rl(p, q)}T_{pq}$.
\end{enumerate}
Here is now statement from \cite[Thm 3.9]{banica2009liberation} which enables to define easy quantum groups.

\begin{theo}\label{thm:tannakakrein}
Let $\CC$ be a category of partitions and let $N$ be an integer. Then, there exists a compact quantum group $\G$ together with a fundamental representation $u$ of dimension $N$ such that for all $k, l\in \N$,
\begin{equation*}
\mor_{\G}(u^{\boxtimes k}, u^{\boxtimes l}) = \spa\{T_{p} \mid p\in \CC(k, l)\}.
\end{equation*}
The compact quantum group $\G$ is called the \emph{easy quantum group} associated with $\CC$ and $N$ and is denoted by $\G_{N}(\CC)$.
\end{theo}

The smallest category of partitions is the set $NC_{2}$ of all non-crossing partitions with blocks of size two (these are usually called pair partitions) and the corresponding compact quantum group is the \emph{free orthogonal quantum group} $O_{N}^{+}$ originally introduced by S. Z. Wang in \cite{wang1995free}. Given a category of partitions $\CC$, the inclusion $NC_{2}\subset \CC$ translates into a surjective $*$-homomorphism $C(O_{N}^{+})\to C(\G_{N}(\CC))$ so that easy quantum groups all are quantum subgroups of $O_{N}^{+}$.

\section{Duality theorems for ergodic actions}\label{sec:duality}

In this section we will prove our reconstruction theorem for ergodic actions of compact quantum groups. The result will be split into two parts. First, we show how one can reconstruct an ergodic action out of a specific functor and then we investigate extra structures on the functor corresponding to the action having a Yetter-Drinfeld structure.

\subsection{The reconstruction theorem for ergodic actions}

We will work, for the sake of simplicity and because this is enough for our examples in the next section, with a compact quantum group $\G$ which is a quantum subgroup of $O_{N}^{+}$. This means that there is a specific unitary finite-dimensional representation $u$ of $\G$ acting on $H = \C^{N}$ which is orthogonal and whose coefficients generate $C(\G)$. This implies in particular that any irreducible representation is equivalent to a subrepresentation of $u^{\boxtimes k}$ for some $k$. Let us start with some notations. We will simply write $\mathrm{Mor}_{\mathbb{G}}(k,l)$ for $\mathrm{Mor}_{\mathbb{G}}(u^{\boxtimes k},u^{\boxtimes l})$. Moreover, $(e_{i})_{1\leqslant i\leqslant N}$ will denote the canonical basis of $\C^{N}$ and we will, for tuples $\mathbf{i} = (i_{1}, \dots, i_{n})$ and $\mathbf{j} = (j_{1}, \dots, j_{n})$,
\begin{equation*}
e_{\mathbf{i}} = e_{i_{1}}\otimes\cdots\otimes e_{i_{n}} \;\; \text{ and } \;\;  u_{\mathbf{i}\mathbf{j}} = u_{i_{1}j_{1}}\cdots u_{i_{n}j_{n}}.
\end{equation*}

The statement of our result is somewhat complicated, and to make it clearer we first describe it informally. Recall that given an ergodic action $\alpha : B \to B\otimes C(\G)$, the spaces $K_{v} = (H_{v}\otimes B)^{\G}$ are Hilbert spaces. Moreover, any intertwiner $T\in \mor_{\G}(u, v)$ yields a linear map $(T\otimes \id) : K_{v}\to K_{v'}$. In other words, the collection of spectral subspaces of $B$ can be seen as a functor from $\rep(\G)$ to the category of finite-dimensional Hilbert spaces (denoted in the sequel by $\hil_{f}$), and even a $*$-functor since this obviously preserves adjoints. Because $\rep(\G)$ has the extra structure of a monoidal category coming from the tensor product of representations, it is natural to investigate the behavior of this functor with respect to tensor products. It turns out (see \cite[Thm 7.3]{PR07}) that there are isometric inclusions $K_{v}\otimes K_{v'}\hookrightarrow K_{v\boxtimes v'}$, but these fail to be surjective in general (see for instance Section \ref{sec:examples} for examples). We therefore have the structure of a \emph{weak unitary tensor functor} as defined in \cite{N14} (originally called \emph{quasitensor $*$-functor} in \cite{PR07}).

\begin{defi}\label{de:weakunitarytensorfunctor}
A weak unitary tensor functor $\varphi:\rep(\G) \to \hil_{f}$ is given by a finite-dimensional Hilbert space $K_{v} $ for each finite-dimensional representation $v$ and a linear map $\varphi(T)\in B(K_{v}, K_{w})$ for each $T\in\mor_{\G}(v,w)$ and $v, w\in\rep(\G)$ such that for all $v, v', w, w'\in\rep(\G)$ and for all $T\in\mor_{\G}(v, v')$, $S\in\mor_{\G}(w, w')$, the following hold:
\begin{enumerate}[label=(F\arabic*)]
\item\label{cond:1} the map $T\mapsto \varphi(T)$ is linear and satisfies
\begin{itemize}
\item $\varphi(\id)=\id$,
\item $\varphi(T^{*}) = \varphi(T)^{*}$,
\item $\varphi(TS) = \varphi(T)\varphi(S)$ if $w' = v$;
\end{itemize}
\item\label{cond:2} there is an isometric embedding $\iota_{v, w} : K_{v}\otimes K_{w}\to K_{v\boxtimes w}$ such that 
\begin{equation*}
\iota_{v\boxtimes v', w}(\iota_{v, v'}\otimes\id_{K_{w}}) = \iota_{v, v'\boxtimes w}(\id_{K_{v}}\otimes\iota_{v', w})
\end{equation*}
and
\begin{equation*}
\varphi(T \xbox S)\iota_{v,w}=\iota_{v',w'}(\varphi(T)\otimes\varphi(S));
\end{equation*}
\item\label{cond:3} we have
\[
P_{v\boxtimes v',w} P_{v,v'\boxtimes w} \leq P_{v,v',w}
\]
where $P_{v\boxtimes v',w}$, $P_{v,v'\boxtimes w}$ and $P_{v,v',w}$ are the range projection of the maps $\iota_{v\boxtimes v',w}$, $\iota_{v,v'\boxtimes w}$ and
\[
\iota_{v,v',w}:=\iota_{v\boxtimes v',w}(\iota_{v,v'} \otimes \id_{K_{w}})=\iota_{v,v'\boxtimes w}(\id_{K_{v}} \otimes \iota_{v,w}).
\]
respectively;
\item we have $K_{v\oplus w} = K_{v}\oplus K_{w}$ and $\varphi (T\oplus S) = \varphi(T)\oplus \varphi(S)$.
\end{enumerate}
Following the usual categorical notation, we will also write $K_v = \varphi (v)$ if we want to emphasize the role of the functor.
\end{defi}

\begin{rem}\label{rem:cond}
	Condition \ref{cond:3} may implies other similar variants considered in \cite{N14}. For $u,v \in \rep(\G)$ and $X \in \varphi(u)$, we put
\[
S^{u,v}_{X}: \varphi(v) \to \varphi(u \boxtimes v),\quad Y  \mapsto  \iota_{u,v}(X \otimes Y) .
\]
Then for all $u,v,w \in \mathrm{Rep}(\mathbb{G})$, the equality $P_{u\boxtimes v,w} P_{u,v\boxtimes w} \leq P_{u,v,w}$ implies
\begin{equation*}
(S^{u,v\boxtimes w}_{X})^{*}\iota_{u \boxtimes v,w} = \iota_{v,w}((S^{u,v}_{X})^{*} \otimes \id).
\end{equation*}
This implication is well-known to experts and we will explain it a little bit in Appendix \ref{appendix:cond}.
\end{rem}
\medskip

As explained above, any ergodic action $(B, \alpha)$ yields such a functor, denoted by $\varphi_{\alpha}$. C. Pinzari and J. Robert proved in \cite{PR07} a Tannaka-Krein duality result for ergodic actions, which we may state briefly as follows:
\begin{theo}[{\cite[Thm 8.1]{PR07}}]\label{theo:tannakakreinpinzarirobert}
Any weak unitary tensor functor $\varphi:\rep(\G) \to \hil_{f}$ is naturally unitary monoidal isomorphic to $\F_{\alpha}$ for some ergodic action $(B,\alpha)$ of $\G$. 
\end{theo} 
Our next result is a variation on this where we show that the weak unitary tensor functor only needs to be defined on tensor powers of the fundamental representation $u$. Recall that we only focus on the orthogonal setting, and we will always be considering the universal versions of the quantum group and the corresponding action.

\begin{theo}\label{thm:tkergodic}
Let $(K_{n})_{n\in\N}$ be a family of finite-dimensional Hilbert spaces with $K_{0} = \C$, and let $\varphi : \mor_{\G}(k, l)\to \B(K_{k}, K_{l})$ be a map for all $k, l\in \N$. Assume that $\varphi$ and $(K_{n})_{n\in\N}$ satisfy \emph{\ref{cond:1}, \ref{cond:2}} and \emph{\ref{cond:3} }  of Definition \ref{de:weakunitarytensorfunctor} for $v = u^{\boxtimes k}$, $v' = u^{\boxtimes k'}$, $w = u^{\boxtimes l}$ and $w' = u^{\boxtimes l'}$. Then, there exists a unique weak unitary tensor $\varPhi: \rep(\G)\to \hil_{f}$ such that for all integers $n, k, l$,
\begin{equation*}
\varPhi (u^{\boxtimes n} ) = K_{n} \quad \text{and} \quad \varPhi | _{\mor_{\G}(k, l)} = \varphi.
\end{equation*}
Moreover, let $B$ be the universal unital $C^{*}$-algebra generated by the matrix coefficients of $X^{(n)}\in\B((\C^{N})^{\otimes n}, K_{n})\otimes B$ for $n\in\mathbb{N}$ such that for any $k, l\in \N$,
\begin{enumerate}[label=\textup{(\alph*)}]
\item\label{con:tensor} $X^{(k)}\boxtimes X^{(l)} = (\iota_{k, l}^{*}\otimes\id)X^{(k+l)}$, where $\iota_{k,l}$ is the isometric embedding $\iota_{u^{\boxtimes k},u^{\boxtimes l}}$;
\item\label{con:intertwiners} $X^{(l)}(T\otimes\id) = (\varphi(T)\otimes\id)X^{(k)}$ for any $T\in\mor_{\G}(k, l)$;
\item\label{cond:conjugate} $\overline{X^{(k)}}(j_{k}\otimes\id) = (J_{k}\otimes\id)X^{(k)}$, where $j_{k}: (\C^{N})^{\otimes k}\to \overline{(\C^{N})^{\otimes k}}$ is the complex conjugation, $J_{k}: K_{k}\to K_{k}$ is the anti-linear involutive map such that
\begin{equation*}
\langle J_{k}\xi, \eta\rangle = \langle\iota_{k, k}^{*}\varphi(R_{k})(1), \xi\otimes\eta\rangle
\end{equation*}
for any $\xi, \eta\in K_{k}$ and $R_{k}$ is the solution of conjugate equation for the representation $u^{\boxtimes k}$.
\end{enumerate}
Given a basis $(\xi_{p})_{p\in I}$ of $K_{n}$ and writing $x^{(n)}_{p\mathbf{i}} := (\omega_{p\mathbf{i}}\otimes\mathrm{id})(X^{(n)})$, define an ergodic action $\alpha$ of $\G$ on $B$ by
\begin{equation*}
\alpha(x^{(n)}_{p\mathbf{j}}) = \sum_{\mathbf{i}\in \{1, \dots, N\}^{n}} x^{(n)}_{p\mathbf{i}}\otimes u_{\mathbf{i}\mathbf{j}}.
\end{equation*}
Then, the weak unitary tensor functor is naturally unitarily monoidally isomorphic to $\varphi_{\alpha}$.
\end{theo}

\begin{rem}
Condition {\ref{cond:conjugate}} in the above theorem is equivalent to the relation 
\begin{equation*}
(\overline{X^{(k)}}\boxtimes X^{(k)})R_{k}=\iota_{k,k}^{*}\varphi(R_{k}).
\end{equation*}
\end{rem}

For the sake of clarity, we will split the proof of our theorem into two parts. The first part is the existence of the weak unitary tensor functor meaning that we have to extend the construction to all finite-dimensional representations of $\G$. Since any such representation splits into a direct sum of irreducible representations, it should be sufficient to be able to extend the construction to irreducible representations. The following result of S. Neshveyev \cite[Proposition 4.1]{N14} makes this precise.
  
\begin{lem}\label{prop:neshveyev}
Let $(E_{x})_{x\in\irr(\G)}$ be a family of finite-dimensional Hilbert spaces with $E_{\triv} = \C$ and let $\nes : \mor_{\G}(x\boxtimes y, z)\to B(E_{x}\otimes E_{y},E_{z})$ be a linear map for all $x, y, z\in\irr(\G)$ such that the following hold:
\begin{enumerate}[label=$(\arabic*)$]
\item\label{cond:isometric} if $I : E_{x}\otimes\C\to E_{x}$ is the canonical isomorphism, then $\nes(I) = \id_{E_{x}}$ and if $x\boxtimes y\cong\oplus_{i=1}^{n}z_{i}$ and $T_{i}\in\mathrm{Mor}_{\mathbb{G}}(x\boxtimes y, z_{i})$ are the associated intertwiners whose range projections add up to the identity, then the linear operator
\begin{equation*}
(\nes(T_{1}), \dots, \nes(T_{n})) : E_{x}\otimes E_{y}\to\bigoplus_{i=1}^{n}E_{z_{i}}
\end{equation*}
is isometric;
\item\label{cond:coherence} if a morphism in $\mor_{\G}(x\boxtimes y\boxtimes z, x')$ can be written as
\begin{equation*}
\sum_{i}S_{i}(T_{i}\otimes\id) = \sum_{j}S_{j}'(\id\otimes T_{j}'),
\end{equation*}
where $T_{i}\in\mor_{\G}(x\boxtimes y, x_{i})$, $S_{i}\in\mor_{\G}(x_{i}\boxtimes z, x')$, $T_{j}'\in\mor_{\G}(y\boxtimes z, y_{j})$, $S_{j}'\in\mor_{\G}(x\boxtimes y_{j}, x')$, then
\begin{equation*}
\sum_{i}\nes(S_{i})(\nes(T_{i})\otimes\id) = \sum_{j}\nes(S_{j}')(\id\otimes\nes(T_{j}'));
\end{equation*}
\item\label{cond:proj} for every vector $X \in E_{x}$ and every morphism $T \in\mor_{\G}(x\boxtimes y,z)$, denote by $[T]_{X}:E_{y}\to E_{z}$ the linear map sending $Y$ to $\nes(T)(X \otimes Y)$; if a morphism in $\mor_{\G}(y\boxtimes z,x\boxtimes x')$ can be written as
\begin{equation*}
\sum_{i}(\id\otimes S_{i})(T^{*}_{i}\otimes\id) = \sum_{j}P^{*}_{j}R_{j},
\end{equation*}
where $T_{i}\in\mor_{\G}(x\boxtimes y'_{i},y)$, $S_{i}\in\mor_{\G}(y'_{i}\boxtimes z, x')$, $R_{j}\in\mor_{\G}(y\boxtimes z, y''_{j})$, $P_{j}\in\mor_{\G}(x\boxtimes x', x''_{j})$, then
\begin{equation*}
\sum_{i}\nes(S_{i})([T_{i}]^{*}_{X}\otimes\id) = \sum_{j}[P_{j}]^{*}_{X}\nes(R_{j}).
\end{equation*}
\end{enumerate}
Then there exists a weak unitary tensor functor $\varPhi:\rep(\G) \to \hil_{f}$ such that $\varPhi(x) = E_x$ and $\varPhi(T)\iota_{x,y}=\nes(T)$ for all $T\in \mor_\G (x\boxtimes y , z)$ and all $x,y,z\in\irr(\G)$.
\end{lem}

We can now prove the first part of the theorem.

\begin{proof}[Proof of Theorem \ref{thm:tkergodic}, first part]
We first construct the family $(E_{x})_{x\in\irr(\G)}$. Let $u^{x}$ be a representative of $x\in \irr(\G)$ and let $k_{x}$ be an integer such that $u^{x}$ is a subrepresentation of $u^{\boxtimes k_{x}}$. Then, there exists an intertwiner $P_{x}\in\mor_{\G}(k_{x}, k_{x})$ such that $P_{x}\in \B(H^{\otimes k_{x}}_{u})$ is an orthogonal projection and the subrepresentation $(P_{x}\otimes\id)u^{\boxtimes k_{x}}$ is equivalent to $u^{x}$. Moreover $\varphi(P_{x})$ defines an orthogonal projection on the Hilbert space $K_{k_{x}}$ and we set
\begin{equation*}
E_{x} = \varphi(P_{x})(K_{k_{x}}).
\end{equation*}

Now let us construct the morphism $\nes(T)$ for a given $T\in\mathrm{Mor}_{\mathbb{G}}(x\boxtimes y, z)$. We fix for each $x\in \irr(\G)$ an isometric embedding $w_{x} : H_{x}\to H^{\otimes k_{x}}_{u}$ such that $w_{x}w_{x}^{*} = P_{x}$. Set
\begin{equation*}
\widetilde{T} = w_{z}T(w_{x}\otimes w_{y})^{*} \in\mathrm{Mor}_{\G}(k_{x}+k_{y}, k_{z})
\end{equation*}
and define
\begin{equation*}
\nes(T) = \varphi(P_{z})\varphi(\widetilde{T})\left(\iota_{k_{x},k_{y}}\right)_{\mid E_{x}\otimes E_{y}} : E_{x}\otimes E_{y}\to E_{z}.
\end{equation*}

To conclude the proof we must check that the conditions of Lemma \ref{prop:neshveyev} are satisfied. First, let $T_{i}\in\mor_{\G}(x\boxtimes y, z_{i})$ be intertwiners as in \ref{cond:isometric} of Lemma \ref{prop:neshveyev}. Then,
\begin{align*}
\sum_{i=1}^{n}\nes(T_{i})^{*}\nes(T_{i}) & = \sum_{i=1}^{n}\iota_{k_{x}, k_{y}}^{*}\varphi(\widetilde{T}_{i}P_{z_{i}}\widetilde{T}_{i})\iota_{k_{x}, k_{y}} \\
& = \iota_{k_{x}, k_{y}}^{*}\varphi\left((w_{x}\otimes w_{y})(\sum_{i}T_{i}^{*}w_{z_{i}}^{*}P_{z_{i}}w_{z_{i}}T_{i})(w_{x}\otimes w_{y})^{*}\right)\iota_{k_{x},k_{y}} \\
& = \id.
\end{align*}
Thus, $(\nes(T_{1}),\ldots,\nes(T_{n}))$ is isometric and the first condition \ref{cond:isometric} of Lemma \ref{prop:neshveyev} is verified. Second, let $T_{i}$, $S_{i}$, $T_{i}'$ and $S_{i}'$ be the intertwiners as in  \ref{cond:coherence} of Lemma \ref{prop:neshveyev}. We have
\begin{align*}
\sum_{i}\nes(S_{i})(\nes(T_{i})\otimes\id) & = \sum_{i}\varphi(P_{x'})\varphi(\widetilde{S}_{i})\iota_{k_{x_{i}}, k_{z}}\left(\varphi(P_{x_{i}})\varphi(\widetilde{T}_{i})\iota_{k_{x}, k_{y}}\otimes\id\right) \\
& = \sum_{i}\varphi(P_{x'})\varphi(\widetilde{S}_{i})\iota_{k_{x_{i}}, k_{z}}\left(\varphi(P_{x_{i}}\widetilde{T}_{i})\otimes\varphi(1)\right)(\iota_{k_{x}, k_{y}}\otimes\id) \\
& = \sum_{i}\varphi(P_{x'})\varphi(\widetilde{S}_{i})\varphi(P_{x_{i}}\widetilde{T}_{i}\otimes1)\iota_{k_{x}+k_{y}, k_{z}}(\iota_{k_{x}, k_{y}}\otimes\id) \\
& = \varphi\left(\sum_{i}P_{x'}\tilde{S}_{i}(P_{x_{i}}\widetilde{T}_{i}\otimes 1)\right)\iota_{k_{x}, k_{y}+k_{z}}(\id\otimes\iota_{k_{y}, k_{z}}).
\end{align*}
Note that
\begin{align*}
\sum_{i}P_{x'}\widetilde{S}_{i}(P_{x_{i}}\widetilde{T}_{i}\otimes\id) & = \sum_{i}w_{x'}S_{i}(w_{x_{i}}\otimes w_{z})^{*}\left(w_{x_{i}}T_{i}(w_{x}\otimes w_{y})^{*}\otimes\id\right) \\
& = w_{x'}\sum_{i}S_{i}(T_{i}\otimes\id)(\id\otimes w_{y}^{*}\otimes w_{z}^{*}) \\
& = w_{x'}\sum_{i}S_{j}'(\id\otimes T_{j}')(\id\otimes w_{y}^{*}\otimes w_{z}^{*}) \\
& = \sum_{i}P_{x'}\widetilde{S}_{j}'(\id\otimes P_{y_{j}}\widetilde{T}_{j}').
\end{align*}
Thus, repeating the above argument for $S_{j}'$ and $T_{j}'$ yields
\begin{align*}
\sum_{i}\nes(S_{i})(\nes(T_{i})\otimes\id) & = \varphi\left(\sum_{i}P_{x'}\widetilde{S}_{j}'(\id\otimes P_{y_{j}}\widetilde{T}_{j}')\right)\iota_{k_{x},k_{y}+k_{z}}(\id\otimes\iota_{k_{y},k_{z}})\\
& = \sum_{j}\varphi(S_{j}')(\id\otimes\varphi(T_{j}')),
\end{align*}
as desired. Finally, let us check \ref{cond:proj} of Lemma \ref{prop:neshveyev}. For $n,k\in\N$ and $X\in K_{n}$, denote
\[
\begin{array}{lccc}
S^{n,k}_{X} :& K_{k} & \to & K_{n+k} \\
& Y & \mapsto & \iota_{n,k}(X\otimes Y)
\end{array}.
\]
Note that by \ref{cond:2} we have
\begin{equation}\label{eq:sxphi}
	(S^{n,k}_{X})^{*}\circ\varphi(\id\otimes T) = \varphi(T)\circ (S^{n,k'}_{X})^{*}
\end{equation}  
for all $T\in\mor_{\G}(k',k)$. Moreover by Remark \ref{rem:cond}, for all $ l\in \mathbb N$ we have 
\begin{equation}\label{eq:condtionn}
(S^{n,k+l}_{X})^{*}\iota_{n+k,l}=\iota_{k,l}((S^{n,k}_{X})^{*}\otimes\id),
\end{equation}

Given $X\in E_{x}$ and $T \in \mor_{\G}(x\boxtimes y,z)$, by definition we have
\[
[T]_{X}(Y) = \nes(T)(X\otimes Y) = \varphi(P_{z})\varphi(\tilde{T})S^{k_{x},k_{y}}_{X}(Y) 
\]
and the adjoint map is given by
\[
[T]_{X}^{*} = \varphi(P_{y})(S^{k_{x},k_{y}}_{X})^{*}\varphi(\tilde{T}^{*})\varphi(P_{z}).
\]
Now take a morphism in $\mor_{\G}(y\boxtimes z,x\boxtimes x')$ that can be written as
\begin{equation*}
\sum_{i}(\id_{x}\otimes S_{i})(T^{*}_{i}\otimes\id_{z}) = \sum_{j}P^{*}_{j}R_{j},
\end{equation*}
where $T_{i}\in\mor_{\G}(x\boxtimes y'_{i},y)$, $S_{i}\in\mor_{\G}(y'_{i}\boxtimes z, x')$, $R_{j}\in\mor_{\G}(y\boxtimes z, x''_{j})$, $P_{j}\in\mor_{\G}(x\boxtimes x', x''_{j})$. We have
\begin{align*}
&\ \quad \sum_{i}\nes(S_{i})([T_{i}]^{*}_{X}\otimes\id) \\
& = \sum_{i}\varphi(P_{x'})\varphi(\tilde{S}_{i})\iota_{k_{y'_{i}},k_{z}}\left(\varphi(P_{y'_{i}})(S^{k_{x},k_{y'_{i}}}_{X})^{*}\varphi(\tilde{T}^{*}_{i})\varphi(P_{y}) \otimes \id\right) \\
& = \sum_{i}\varphi(P_{x'})\varphi(\omega_{x'}S_{i}(\omega_{y'_{i}}\otimes\omega_{z})^{*})\iota_{k_{y'_{i}},k_{z}}\left(\varphi(P_{y'_{i}})(S^{k_{x},k_{y'_{i}}}_{X})^{*}\varphi((\omega_{x}\otimes\omega_{y'_{i}})T^{*}_{i}\omega^{*}_{y})\varphi(P_{y}) \otimes \id\right) \\
& = \sum_{i}\varphi(\omega_{x'}S_{i}(\omega_{y'_{i}}\otimes\omega_{z})^{*})\iota_{k_{y'_{i}},k_{z}}\left(\varphi(P_{y'_{i}})(S^{k_{x},k_{y'_{i}}}_{X})^{*}\varphi((\omega_{x}\otimes\omega_{y'_{i}})T^{*}_{i}\omega^{*}_{y}) \otimes \id\right). \end{align*}
Together with \eqref{eq:sxphi} and \eqref{eq:condtionn} we have
\begin{align*}
\quad \sum_{i}\nes(S_{i})([T_{i}]^{*}_{X}\otimes\id)  
&   = \sum_{i}\varphi(\omega_{x'}S_{i}(\omega_{y'_{i}}\otimes\omega_{z})^{*})\iota_{k_{y'_{i}},k_{z}}((S^{k_{x},k_{y'_{i}}}_{X})^{*}\otimes\id)\left(\varphi((\omega_{x}\otimes P_{y'_{i}}\omega_{y'_{i}})T^{*}_{i}\omega^{*}_{y}) \otimes \id\right)
\\
& =  \sum_{i}\varphi(\omega_{x'}S_{i}(\omega_{y'_{i}}\otimes\omega_{z})^{*})(S^{k_{x},k_{y'_{i}}+k_{z}}_{X})^{*}\iota_{k_{x}+k_{y'_{i}},k_{z}}\left(\varphi((\omega_{x}\otimes P_{y'_{i}}\omega_{y'_{i}})T^{*}_{i}\omega^{*}_{y}) \otimes \id\right) \\
& = \sum_{i}(S^{k_{x},k_{x'}}_{X})^{*}\varphi(\id\otimes\omega_{x'}S_{i}(\omega_{y'_{i}}\otimes\omega_{z})^{*})\varphi((\omega_{x}\otimes P_{y'_{i}}\omega_{y'_{i}})T^{*}_{i}\omega^{*}_{y} \otimes \id)\iota_{k_{y},k_{z}} \\
& = \sum_{i}(S^{k_{x},k_{x'}}_{X})^{*}\varphi((\omega_{x}\otimes\omega_{x'})(\id\otimes S_{i})(T^{*}_{i}\otimes\id)(\omega_{y} \otimes \omega_{z})^{*})\iota_{k_{y},k_{z}}.
\end{align*}
On the other hand,
\begin{align*}
\sum_{j}[P_{j}]^{*}_{X}\nes(R_{j}) & = \sum_{j}\varphi(P_{x'})(S^{k_{x},k_{x'}}_{X})^{*}\varphi(\tilde{P}^{*}_{j})\varphi(P_{x''_{j}})\varphi(P_{x''_{j}})\varphi(\tilde{R}_{j})\iota_{k_{y},k_{z}} \\
& = \sum_{j}\varphi(P_{x'})(S^{k_{x},k_{x'}}_{X})^{*}\varphi(\tilde{P}^{*}_{j} P_{x''_{j}}\tilde{R}_{j})\iota_{k_{y},k_{z}} \\
& = \sum_{j}\varphi(P_{x'})(S^{k_{x},k_{x'}}_{X})^{*}\varphi((\omega_{x}\otimes\omega_{x'})P^{*}_{j}\omega^{*}_{x''_{j}}P_{x''_{j}}\omega_{x''_{j}}R_{j}(\omega_{y}\otimes\omega_{z})^{*})\iota_{k_{y},k_{z}} \\
& = \sum_{j}\varphi(P_{x'})(S^{k_{x},k_{x'}}_{X})^{*}\varphi((\omega_{x}\otimes\omega_{x'})P^{*}_{j}R_{j}(\omega_{y}\otimes\omega_{z})^{*})\iota_{k_{y},k_{z}} \\
& = \sum_{j}(S^{k_{x},k_{x'}}_{X})^{*}\varphi((\omega_{x}\otimes P_{x'}\omega_{x'})P^{*}_{j}R_{j}(\omega_{y}\otimes\omega_{z})^{*})\iota_{k_{y},k_{z}} \\
& = \sum_{j}(S^{k_{x},k_{x'}}_{X})^{*}\varphi((\omega_{x}\otimes\omega_{x'})P^{*}_{j}R_{j}(\omega_{y}\otimes\omega_{z})^{*})\iota_{k_{y},k_{z}}.
\end{align*} 
Combing the above equalities, we obtain 
\begin{equation*}
\sum_{i}\nes(S_{i})([T_{i}]^{*}_{X}\otimes\id_{E_{z}}) = \sum_{j}[P_{j}]^{*}_{X}\nes(R_{j}),
\end{equation*}
as desired.
\end{proof}

The second part of Theorem \ref{thm:tkergodic} is the description of the C*-algebra $B$ and its action. To prove it, let us first recall how this C*-algebra is constructed in \cite{PR07}. It is enough to describe the core $\BB$ of $B$, and this is done as follows:
\begin{enumerate}
\item As a linear space,
\begin{equation*}
\BB = \bigoplus_{x\in\irr(\G)} \bar{K}_{x}\otimes H_{x}
\end{equation*}
and the action is given on each summand by $\alpha_{\mid \bar{K}_{x}\otimes H_{x}} = \mathrm{id}_{\bar{K}_{x}}\otimes\delta_{u^{x}}$.
\item For each $v\in \rep(\G)$, there exists a matrix $Y^{v}\in \B(H_{v}, K_{v})\otimes \BB$ such that for all $T\in \mor_{\G}(v, w)$
\begin{equation*}
Y^{w}(T\otimes\id) = (\varphi(T)\otimes\id)Y^{v}
\end{equation*}
Moreover, for any $x\in\irr(\G)$, $(\omega_{\eta\xi }\otimes \mathrm{id})(Y^{x}) = \overline{\eta}\otimes \xi $ for all $\xi\in H_{x}$ and all $\eta\in K_{x}$ (see the discussions around \cite[Equality (8.1)]{PR07}).
\item For all representations $v$, $\overline{Y^{v}}(j_{v}\otimes\id) = (J_{v}\otimes\id)Y^{v}$, where $j_{v} : H_{v} \to  H_{\bar{v}}$ is the complex conjugation, $J_{v}: K_{v} \to K_{\bar{v}}$ is the anti-linear involution map such that
\begin{equation*}
\langle J_{v}(\xi),\eta\rangle = \langle\iota_{v,\overline{v}}^{*}\varphi(\bar{R}_{v})(1),\xi\otimes\eta\rangle,
\end{equation*}
for any $\xi \in K_{v}$, $\eta \in K_{\overline{v}}$, and where $\bar{R}_{v}$ is the solution of conjugate equation for $v$ (see the discussions around \cite[Equality (8.2)]{PR07}).
\end{enumerate} 

We can now complete the proof of Theorem \ref{thm:tkergodic}.

\begin{proof}[Proof of Theorem \ref{thm:tkergodic}, second part]
The statement of the theorem defines a universal $*$-algebra $\BB'$ of which $B$ is the universal completion. We will therefore prove that the algebraic core $\BB$ of the action described above is isomorphic to $\BB'$. For $k\in\mathbb{N}$, write $Y^{(k)} = Y^{u^{\boxtimes k}}$ and observe that they satisfy all the conditions of Theorem \ref{thm:tkergodic}, so that there is a surjective $*$-homomorphism $\pi : \BB' \to \BB$ sending $(\omega\otimes\id)X^{(k)}$ to $(\omega\otimes\id)Y^{(k)}$ for all $\omega\in \B(H^{\otimes k}, K_{k})^{*}$. Our task is therefore to prove the injectivity of $\pi$. For each $x\in\irr(\G)$ such that $u^{x}$ is a subrepresentation of $u^{\boxtimes k}$, we denote by $P^{(k)}_{x}\in\mor_{\G}(k, k)$ an intertwiner such that the map $P_{x}^{(k)}\in \B(H^{\otimes k}_{u})$ is an orthogonal projection and the subrepresentation $(P_{x}^{(k)}\otimes\id)u^{\boxtimes k}$ is equivalent to $u^{x}$. It follows from Condition \ref{con:intertwiners} in the statement of the theorem that
\begin{equation*}
X^{(k)} = \bigoplus_{u^{x}\prec u^{\boxtimes k}}\varphi(P_{x}^{(k)})X^{(k)}P_{x}^{(k)}.
\end{equation*}
If $i : E_{x} \hookrightarrow K_{k}$ denotes the inclusion and $\omega_{x} : H_{x} \to H^{\otimes k}$ denotes an isometry arising from the containment $u^{x} \prec u^{\boxtimes k}$, we set
\begin{equation*}
X^{x} := i^{*}\circ\varphi(P_{x}^{(k)})X^{(k)}P_{x}^{(k)}\circ\omega_{x}\in \B(H_{x}, E_{x})\otimes \BB,
\end{equation*}
Moreover, using Condition \ref{con:tensor} in the statement, we see that the definition of $X^{x}$ does not depend on the choice of $k$. As a consequence, $\BB'$ is spanned by $\{(\omega\otimes\id)X^{x} : \omega\in\B(H_{x}, E_{x})^{*}, x\in\irr(\G)\}$, a set which is itself spanned by
\begin{equation*}
\{(\omega_{\eta_{j}\xi_{i}}\otimes\id)X^{x} : (i, j)\in I\times J, x\in\irr(\G)\},
\end{equation*}
where $(\xi_{i})_{i\in I}$ and $(\eta_{j})_{j\in J}$ are orthonormal basis of $H_{x}$ and $E_{x}$ respectively. To conclude, observe that because $(\omega_{\xi_{i}\eta_{j}}\otimes\id)Y^{x} = \overline{\xi_{i}}\otimes\eta_{j}$, these coefficients are linearly independent. In other words, $\pi$ sends a generating family to a linearly independent family, hence it is injective and the proof is complete.
\end{proof}

\subsection{The reconstruction theorem for Yetter-Drinfeld structures}\label{sec:yd}

We will now characterize when the C*-algebra built in the previous subsection is a Yetter-Drinfeld C*-algebra. Following the Tannaka-Krein philosophy, we want to describe the properties a weak unitary tensor functor should satisfy for the corresponding C*-algebra to be Yetter-Drinfeld. We also refer to \cite{NY14} for related results in the setting of module C*-categories.

Let us first state the general reconstruction theorem for (braided commutative) Yetter-Drinfeld algebras. Because this is not the version which will be used hereafter, we postpone its proof to Appendix~\ref{app:yd} together with the definition of a (braided) compatible collection for a weak unitary tensor functor.

\begin{theo}\label{theo:yd_reconstruction_na}
Let $\varphi: \rep(\G) \to \hil_{f}$ be a weak unitary tensor functor and $(B,\alpha)$ be the action of $\G$ associated with $\varphi$ given by Theorem \ref{theo:tannakakreinpinzarirobert}. The following are equivalent:
\begin{enumerate}[label=\emph{(\roman*)}]
\item there exists a compatible collection $\{\mathcal{A}^{w}_{v}\}_{v,w \in \rep(\G)}$ for $\varphi$;
\item there is a Hopf $*$-algebra action of $ \pol(\G)$ on $\mathcal{B}$, denoted by $\lhd: \mathcal{B}  \otimes \pol(\G) \to \mathcal{B} $, such that the tuple $(B ,\lhd ,\alpha)$ yields a Yetter-Drinfeld $\G$-C*-algebra.
\end{enumerate}
In that case, if $\phi: \varphi \iso \varphi_{\alpha}$ denotes the natural unitary monoidal isomorphism between the weak unitary tensor functors $\varphi$ and $\varphi_{\alpha}$, then the diagram
\[
\xymatrix@C=5pc{
\varphi(v)\ar[r]^-{\mathcal{A}^{w}_{v}}\ar[d]_-{\phi_{v}}^-{\iso} & \varphi(\bar{w} \boxtimes v \boxtimes w)\ar[d]\ar[d]^-{\phi_{\bar{w} \boxtimes v \boxtimes w}}_-{\iso} \\
\mor_{\G}(v,\alpha)\ar[r]_{\tensor*[^\alpha]{\mathcal A}{_v^w}} & \mor_{\G}(\bar{w} \boxtimes v \boxtimes w,\alpha)
}
\]
commutes for every $v,w \in \rep(\G)$. Moreover, $\{\mathcal{A}^{w}_{v}\}_{v,w \in \rep(\G)}$ is braided compatible if and only if $(B ,\lhd ,\alpha)$ is braided commutative.
\end{theo}

Like in the previous subsection, our purpose is now to prove a version of the above characterization involving only tensor powers $(u^{\boxtimes n})_{n\in\mathbb N}$ of the fundamental representation $u$, such that it can be exploited in a combinatorial way for easy quantum groups. 
The following is the main result of this subsection. As we mentioned previously, we will focus on examples where the fundamental representation is self-conjugate, hence we will keep the assumption $u=\bar u$ in the following.

\begin{theo}\label{thm:tkergodic_yetter}
We keep the notations and assumptions of Theorem~\ref{thm:tkergodic}. Let $\{\mathrm{a}_{k}:K_{k} \to K_{k+2}\}_{k\in \N}$ be a family of linear maps and write for all $k,n \in \N$
\begin{equation}\label{eq:def_a}
\mathrm{a}^{0}_{k} = \id_{K_{k}},\quad 
\mathrm{a}^{n+1}_{k} = \mathrm{a}_{k+2n}\circ\mathrm{a}_{k+2(n-1)}\circ\cdots\circ\mathrm{a}_{k}.
\end{equation} 
Assume that $\{\mathrm{a}_{k}\}_{k\in \N}$ and $\{\mathrm{a}_{k}^n\}_{k,n\in \N}$ satisfy for all $k, k'\in \N$ the following conditions:
\begin{enumerate}[label=\textup{(a\arabic*)},start=0]
\item\label{cond:na2_new} For any $p, q$ orthogonal projections in $\mor_{\G}(n,n)$, we have
\[
\varphi(p \xbox \id \xbox q)\mathrm{a}^{n}_{k} = 0; 
\]
\item\label{cond:na1} $\mathrm{a}_{2k}\varphi(R_{k}) = \varphi(R_{k+1})$, where $R_{k}:\triv \to u^{\boxtimes 2k}$;
	
\item\label{cond:na2} For any $T \in \mor_{\G}(k,k')$, the diagram
\[
\xymatrix@R=1pc@C=5pc{
K_{k}\ar[r]^-{\varphi(T)}\ar[d]_-{\mathrm{a}_{k}} & K_{k'}\ar[d]^-{\mathrm{a}_{k'}} \\
K_{k+2}\ar[r]_-{\varphi(\id \xbox T \xbox \id)} & K_{k'+2}
}
\]
commutes;

\item\label{cond:na4} The diagram
\[
\xymatrix@R=1pc@C=7pc{
K_{k} \otimes K_{k'} \ar[rr]^-{\iota_{k,k'}}\ar[d]_-{\mathrm{a}_{k} \otimes \mathrm{a}_{k'}} && K_{k+k'}\ar[d]^-{\mathrm{a}_{k+k'}} \\
K_{k+2} \otimes K_{k'+2}\ar[r]_-{\iota_{k+2,k'+2}} & K_{k+k'+4}\ar[r]_-{\varphi(\id_{1,k} \xbox R^{*} \xbox \id_{k',1})} & K_{k+k'+2}
}
\]
commutes;

\item\label{cond:na5} The diagram
\[
\xymatrix@R=1pc@C=5pc{
K_{k}\ar[r]^-{J_{k}}\ar[d]_-{\mathrm{a}_{k}} & K_{k}\ar[d]^-{\mathrm{a}_{k}} \\
K_{k+2}\ar[r]_-{J_{k+2}} & K_{k+2}
}
\]
commutes.
\end{enumerate}
Then, there is a Hopf $*$-algebra action $\lhd: \mathcal{B} \otimes  \pol(\G) \to \mathcal{B}$ defined by
\[
x^{(k)}_{p\mathbf{i}} \lhd u_{jj'} = (\langle \;\cdot\;e_{j\mathbf{i}j'},\mathrm{a}_{k}(\xi_{p})\rangle \otimes \id)(X^{(k+2)})
\]
for all $k, n \in \N$, $e_{\mathbf{i}} \in (\C^{N})^{\otimes k}$ and $e_{j}, e_{j'} \in \C^{N}$, such that the tuple $(B,\lhd,\alpha)$ yields a Yetter-Drinfeld $\G$-C*-algebra.
\end{theo}
\begin{rem}
We will also frequently use the following reinterpretation of \ref{cond:na2_new}: for two mutually orthogonal nonzero subrepresentations $v,w \prec u^{\boxtimes n}$ with corresponding isometric intertwiners $\theta_v\in \mor_\G (v,u^{\boxtimes n})$, $\theta_w\in \mor_\G (w,u^{\boxtimes n})$, we may take $p=\theta_v \theta_v^*$, $q=\theta_w \theta_w^*$ and obtain
\begin{equation}\label{p:1_2}
\varphi (\theta_v^* \otimes \mathrm{id} \otimes \theta_w^*)\mathrm{a}_k^n =0.
\end{equation}	
\end{rem}
\smallskip

The fact that we only focus on the self-conjugate case has formally simplified some parts of the statement of the theorem. Indeed, in this setting the representation $\overline{u^{\boxtimes n}}$ on $\overline{(\mathbb C ^N)^{\otimes n}}$ is naturally identified with $u^{\boxtimes n}$ on $(\mathbb C ^N)^{\otimes n}$. To avoid confusions in the proof, we first clarify the associated conjugate notations and relations for functors in a separate easy lemma. In the sequel we will frequently use the canonical isomorphisms $c_{n} \in \mor_\G ( u^{\boxtimes n} , \overline{u^{\boxtimes n}} )$, $\sigma_{k,n}\in \mor_\G ( u^{\boxtimes k} \boxtimes u^{\boxtimes n} , \overline{u^{\boxtimes n}} \boxtimes \overline{u^{\boxtimes k}})$ and $\sigma_{n,k,n} \in\mor_\G (\overline{u^{\boxtimes n}} \boxtimes u^{\boxtimes k} \boxtimes u^{\boxtimes n},  \overline{u^{\boxtimes n}} \boxtimes \overline{u^{\boxtimes k}} \boxtimes u^{\boxtimes n} )$ for $n,k\in \mathbb N$, which are \emph{linear} maps determined by
\[
c_n : (\mathbb C ^N)^{\otimes n} \to \overline{(\mathbb C ^N)^{\otimes n}},\quad 
e_{i_1} \otimes \cdots \otimes e_{i_n} \mapsto \overline{e_{i_n} \otimes \cdots \otimes e_{i_1}}
\]
for the canonical basis $(e_{i})_{1\leqslant i\leqslant N}$ of $\mathbb C ^N$, and
\[
\sigma_{k,n} : \overline{(\mathbb C ^N)^{\otimes k} \otimes (\mathbb C ^N)^{\otimes n}} \to \overline{(\mathbb C ^N)^{\otimes n}} \otimes \overline{(\mathbb C ^N)^{\otimes k}},\quad \overline{\xi \otimes \eta} \mapsto \overline{\eta} \otimes \overline{\xi},
\]
\[
\sigma_{n,k,n}: \overline{\overline{(\mathbb C ^N)^{\otimes n}} \otimes (\mathbb C ^N)^{\otimes k} \otimes (\mathbb C ^N)^{\otimes n}} \to
\overline{(\mathbb C ^N)^{\otimes n}} \otimes
\overline{(\mathbb C ^N)^{\otimes k}} \otimes
{(\mathbb C ^N)^{\otimes n}} , \quad \overline{\overline{\xi} \otimes \eta \otimes \chi} \mapsto \overline{\chi} \otimes \overline{\eta} \otimes \xi.
\]
We will also consider the \emph{anti-linear} invertible maps $J_{n} : K_{n}\to K_{n}$ and $J_{v}: \varphi(v) \to \varphi(\bar{v})$ arising from the study of weak unitary tensor functors in the previous section, which are respectively determined by
\begin{equation*}
\langle J_{n}\xi, \eta \rangle = \langle \varphi(R_{n})(1), \iota_{n,n} ( \xi\otimes\eta) \rangle,\quad \forall \xi,\eta,\in K_n
\end{equation*}
and
\[
\langle J_{v}\xi,\eta \rangle = \langle \varphi(\bar{R}_{v})(1),\iota_{v,\bar{v}}(\xi \otimes \eta) \rangle,\quad \forall
\xi \in \varphi(v),\eta \in \varphi(\bar{v}),v\in \rep (\G).
\]

\begin{lem}\label{lem:conjugate}
The following properties hold:
\begin{enumerate}[label=\textup{(c\arabic*)}]
\item\label{cond:c1} for $n\in \N$, $\varphi(c_{n})J_{n} = J_{u^{\boxtimes n}}$;
\item\label{cond:c2} for $n,k\in \N$, $\varphi(c_{n} \otimes c_{k}) = \varphi(\sigma_{k,n})\varphi(c_{n+k}) $;
\item\label{cond:c3} for $n,k\in \N$, $\varphi(c_{n} \otimes c_{k} \otimes \id)J_{k+2n} = \varphi(\sigma_{n,k,n})J_{\overline{u^{\boxtimes n}} \boxtimes u^{\boxtimes k} \boxtimes u^{\boxtimes n}}\varphi(c_{n} \otimes \id)$;
\item\label{eq:inv_map}\label{lem:J_naturality} for $v,w \in \rep(\G)$ and $T \in \mor_{\G}(v,w)$, $
\varphi(\bar{T})J_{v} = J_{w}\varphi(T)$.
\end{enumerate}
\end{lem}
\begin{proof}
(c1)-(c3) are straightforward to check from the definitions and \ref{cond:1}-\ref{cond:2}. For (c4), we note that $\bar{R}^{*}_{v}(\xi \otimes \bar{\xi'}) = \langle \xi,\xi'\rangle$ for all $\xi,\xi' \in H_{v}$, so we have $\bar{R}^{*}_{v}(\id \otimes \bar{T}^{*}) = \bar{R}^{*}_{w}(T \otimes \id)$. Then
\begin{align*}
\langle \varphi(\bar{T})J_{v}(X),Y \rangle & = \langle J_{v}X,\varphi(\bar{T}^{*})(Y) \rangle \\
& = \langle \varphi(\bar{R}_{v})(1),\iota_{v,\bar{v}}(X \otimes \varphi(\bar{T}^{*})(Y)) \rangle \\
& = \langle \varphi(\bar{R}_{v})(1),\varphi(\id \otimes \bar{T}^{*})\iota_{v,\bar{w}}(X \otimes Y) \rangle \\
& = \langle \varphi(\bar{R}_{w})(1),\varphi(T \otimes \id)\iota_{v,\bar{w}}(X \otimes Y) \rangle \\
& = \langle \varphi(\bar{R}_{w})(1),\iota_{w,\bar{w}}(\varphi(T)(X) \otimes Y) \rangle \\
& = \langle J_{w}\varphi(T)(X),Y \rangle
\end{align*}
for all $X \in \varphi(v)$, $Y \in \varphi(\bar{w})$.
\end{proof}

The idea of the proof of Theorem \ref{thm:tkergodic_yetter} is to produce a compatible collection for the weak unitary tensor functor $\varphi:\rep(\G) \to \hil_{f}$ out of the family $\{\mathrm{a}_{k}\}_{k\in\N}$, which allows us to apply Theorem \ref{theo:yd_reconstruction_na}. We start by verifying the corresponding properties on the family $\{\mathrm{a}_{k}^n\}_{k,n\in \N}$.
\begin{lem}
We keep the notations and assumptions of Theorem~\ref{thm:tkergodic_yetter}. The family  $\{\mathrm{a}_{k}^n\}_{k,n\in \N}$ satisfies the following properties for all $n,k,k',l\in \mathbb N$:
\begin{enumerate}[label=\textup{(a\arabic*')}]
\item\label{cond:a1} $\mathrm{a}^{n}_{0} = \varphi(R_{n})$;		
\item\label{cond:a2} $\mathrm{a}^{n}_{k}\varphi(T) = \varphi(\id_{n} \xbox T \xbox \id_{n})\mathrm{a}^{n}_{k'}$ for all $T\in\mor_{\G}(k,k')$;
\item\label{cond:a3} $\mathrm{a}^{n}_{k+2l}\mathrm{a}^{l}_{k} = \mathrm{a}^{n+l}_{k}$;
\item\label{cond:a4} $\mathrm{a}^{n}_{k+k'}\iota_{k,k'} = \varphi(\id_{n,k} \xbox R^{*}_{n} \xbox \id_{k',n})\iota_{k+2n,k'+2n}(\mathrm{a}^{n}_{k} \otimes \mathrm{a}^{n}_{k'})$;
\item\label{cond:a5} $\mathrm{a}^{n}_{k}J_{k} = J_{k+2n}\mathrm{a}^{n}_{k}$.
\end{enumerate}
\end{lem}
\begin{proof}
Note that the definition of $\{\mathrm{a}_{k}^n\}_{k,n\in \N}$ in \eqref{eq:def_a} yields that $\mathrm{a}^{n+1}_{0} = \mathrm{a}_{2n}\mathrm{a}^{n}_{0}$. Then together with \ref{cond:na1}, \ref{cond:a1} follows immediately by induction. We note also that \eqref{eq:def_a} yields a recursion relation
\begin{equation}\label{eq:ind_a}
\mathrm{a}^{n+1}_{k} = \mathrm{a}^{n}_{k+2}\mathrm{a}_{k}.
\end{equation}
With this relation, \ref{cond:a2} and \ref{cond:a5} follow easily by induction on $n$ with \ref{cond:na2} and \ref{cond:na5} being the cases $n=1$. The proof of \ref{cond:a3} goes by induction on $l$ where \eqref{eq:ind_a} corresponds exactly to the case $l=1$. 
For the proof of \ref{cond:a4}, we note that by \eqref{eq:ind_a} and \ref{cond:na4}, we have
 \begin{equation*}
 \mathrm{a}^{n_{0}+1}_{k+k'}\iota_{k,k'}= \mathrm{a}^{n_{0}}_{k+k'+2}\mathrm{a}_{k+k'}\iota_{k,k'} = \mathrm{a}^{n_{0}}_{k+k'+2}\varphi(\id_{1,k} \xbox R^{*} \xbox \id_{k',1})\iota_{k+2,k'+2}(\mathrm{a}_{k} \otimes \mathrm{a}_{k'}),
 \end{equation*} 
 and hence together with \ref{cond:a2} we obtain
	\begin{equation*}
	\mathrm{a}^{n+1}_{k+k'}\iota_{k,k'}	= \varphi(\id_{n+1,k} \xbox R^{*} \xbox \id_{k',n+1})\mathrm{a}^{n}_{k+k'+4}\iota_{k+2,k'+2}(\mathrm{a}_{k} \otimes \mathrm{a}_{k'}) .
	\end{equation*}
Then an induction argument for the term $\mathrm{a}^{n}_{k+k'+4}\iota_{k+2,k'+2}$ together with \eqref{eq:ind_a} yields \ref{cond:a4}.	
\end{proof}

Recall that our aim is to construct a compatible collection $\{\mathcal{A}^{w}_{v}\}_{v,w \in \rep(\G)}$. We note that the family $\{\mathrm{a}_{k}^n\}_{k,n\in \N}$ plays the role of such a collection for tensor powers $v=u^{\boxtimes k}$ and $w=u^{\boxtimes n}$, so we can use these to build desired maps for arbitrary irreducible representations. Indeed, consider $x, y \in \irr(\G)$. There exist minimal integers $k_{x}, k_{y}$ such that $u^{x} \prec u^{\boxtimes k_{x}}$ and $u^{y} \prec u^{\boxtimes k_{y}}$. Denote by $\omega_{x} \in \mor_{\G}(u^{x},u^{\boxtimes k_{x}})$ and $\omega_{y} \in \mor_{\G}(u^{y},u^{\boxtimes k_{y}})$ the associated isometric intertwiners and set
\[
\mathcal{A}^{x}_{y} : \varphi(u^{y}) \to \varphi(\overline{u^{x}} \boxtimes u^{y} \boxtimes u^{x}),\quad
\mathcal{A}^{x}_{y} = \varphi(\bar{\omega}^{*}_{x}c_{k_{x}} \xbox \omega^{*}_{y} \xbox \omega^{*}_{x})\mathrm{a}^{k_{x}}_{k_{y}}\varphi(\omega_{y}).
\]
More generally, consider $w,v \in \rep(\G)$. If we write $w = \oplus w_{i}$ and $v = \oplus v_{j}$ for the decompositions into sums of irreducible representations with isometric intertwiners $\theta_{w_{i}}\in\mor_{\G}(w_{i},w)$, $\theta_{v_{j}}\in\mor_{\G}(v_{j},v)$, then we may define
\[
\mathcal{A}^{w}_{v}: \varphi(v) \to \varphi(\bar{w} \boxtimes v \boxtimes w),\quad
\mathcal{A}^{w}_{v} = \sum_{i,j}\varphi(\bar{\theta}_{w_{i}} \xbox \theta_{v_{j}} \xbox \theta_{w_{i}})\mathcal{A}^{w_{i}}_{v_{j}}\varphi(\theta^{*}_{v_{j}}).
\]
Note that in particular
\begin{equation}\label{eq:A0w}
\mathcal{A}^{w}_{v} = \sum_{i}\varphi(\bar{\theta}_{w_{i}} \xbox \mathrm{id} \xbox \theta_{w_{i}})\mathcal{A}^{w_{i}}_{v}
=\sum_j \varphi(\mathrm{id} \xbox \theta_{v_{j}} \xbox \mathrm{id})\mathcal{A}^{w}_{v_{j}}\varphi(\theta^{*}_{v_{j}})
.
\end{equation}
Before checking the compatibility of $\{\mathcal{A}^{w}_{v}\}_{v,w \in \rep(\G)}$, let us collect some useful auxiliary properties. 
\begin{lem}
	The family $\{\mathcal{A}^{w}_{v}\}_{v,w \in \rep(\G)}$ satisfies the following properties:
	\begin{enumerate}[label=\textup{(\arabic*)}]
		\item Let $w \prec w'$ be a subrepresentation with the corresponding isometric intertwiners $\theta\in\mor_\G (w, w')$. Then		
		\begin{equation}\label{eq:subrep}
		\mathcal{A}^{w}_{v} = \varphi(\bar{\theta}^{*} \xbox \id \xbox \theta^{*})\mathcal{A}^{w'}_{v}
		.
		\end{equation}
		More generally, for a direct sum of subrepresentations $w'=\oplus_i w_i$ with corresponding intertwiners $\theta_i \in \mor_\G (w_i ,w)$, we have
		\begin{equation}
		\label{eq:dec}
			\mathcal{A}^{w'}_{v} = \sum_i  \varphi(\bar{\theta} _i  \xbox \id \xbox \theta_i )\mathcal{A}^{w_i}_{v}.
		\end{equation}
		\item For $x,x' ,y,y' \in\irr(\G)$ we have \begin{equation}\label{p:4}
		\mathcal{A}^{x \boxtimes x'}_{y} = \varphi((\overline{\omega^{*}_{x} \xbox \omega^{*}_{x'}})(c_{k_{x}+k_{x'}}) \xbox \omega^{*}_{y} \xbox \omega^{*}_{x} \xbox \omega^{*}_{x'})\mathrm{a}^{k_{x}+k_{x'}}_{k_{y}}\varphi(\omega_{y}),
		\end{equation}
		\begin{equation}\label{p:5}
		\mathcal{A}^{x'}_{\bar{x} \boxtimes y \boxtimes x}  = \varphi(\bar{\omega}^{*}_{x'}c_{k_{x'}} \xbox \bar{\omega}^{*}_{x}c_{k_{x}} \xbox \omega^{*}_{y} \xbox \omega^{*}_{x} \xbox \omega^{*}_{x'})\mathrm{a}^{k_{x'}}_{k_{y}+2k_{x}}\varphi(\omega_{x}c^{*}_{x} \xbox \omega_{y} \xbox \omega_{x}),
		\end{equation}
		and
		\begin{equation}\label{p:6}
		\mathcal{A}^{x}_{y \boxtimes y'} = \varphi(\bar{\omega}^{*}_{x}c_{k_{x}} \xbox \omega^{*}_{y} \xbox \omega^{*}_{y'} \xbox \omega^{*}_{x})\mathrm{a}^{k_{x}}_{k_{y}+k_{y'}}\varphi(\omega_{y} \xbox \omega_{y'}).
		\end{equation}		
		\end{enumerate}
\end{lem}
\begin{proof}
For $w \prec w'$, we consider the decompositions $w = \oplus z$ and $w' \ominus w = \oplus z'$ into direct sums of irreducible representations with isometric intertwiners $\theta_{z} \in \mor_{\G}(z,w)$ and $\theta'_{z'} \in \mor_{\G}(z',w' \ominus w)$. By \eqref{eq:A0w} we get
\begin{align*}
\mathcal{A}^{w'}_{v} &= \sum_{z \prec w}\varphi(\bar{\theta}\bar{\theta}_{z} \xbox \id \xbox \theta\theta_{z})\mathcal{A}^{z}_{v} + \sum_{z' \prec w' \ominus w}\varphi(\bar{\theta'}_{z'} \xbox \id \xbox \theta'_{z'})\mathcal{A}^{z'}_{v}\\
&
=\varphi(\bar{\theta} \xbox \id \xbox \theta)\mathcal{A}^{w}_{v} + \sum_{z' \prec w' \ominus w}\varphi(\bar{\theta'}_{z'} \xbox \id \xbox \theta'_{z'})\mathcal{A}^{z'}_{v}.
\end{align*}
Then applying $\varphi(\bar{\theta}^{*} \xbox \id \xbox \theta^{*})$ on both sides of the above equality yields \eqref{eq:subrep}, and iterating the argument yields \eqref{eq:dec}. Moreover, taking $w=x\boxtimes x'$ and $w' = u^{\boxtimes ( k_x +k_{x'})}$, we see that \eqref{eq:subrep} implies \eqref{p:4}.

To prove \eqref{p:5}, we consider the decomposition $\bar{x} \boxtimes y \boxtimes x = \oplus z$ into direct sum of irreducible representations with associated isometric intertwiners $\theta_{z}\in\mor_{\G}(z,\bar{x} \boxtimes y \boxtimes x)$. By the definition of $\{\mathcal{A}^{w}_{v}\}_{v,w \in \rep(\G)}$ we have
\[
\mathcal{A}^{x'}_{\bar{x} \boxtimes y \boxtimes x}= \sum_{z } \varphi(\bar{\omega}^{*}_{x'}c_{k_{x'}} \xbox \theta_{z}\omega^{*}_{z} \xbox \omega^{*}_{x'})\mathrm{a}^{k_{x'}}_{k_{z}}\varphi(\omega_{z}\theta^{*}_{z}).
\]
Note further that \ref{cond:a2} yields
\begin{equation*}
\mathrm{a}^{k_{x'}}_{k_{z}}\varphi(\omega_{z}\theta^{*}_{z}) \varphi(c_{x}\omega^{*}_{x} \xbox \omega^{*}_{y} \xbox \omega^{*}_{x})= \varphi(\id \xbox \omega_{z}\theta^{*}_{z}(c_{x}\omega^{*}_{x} \xbox \omega^{*}_{y} \xbox \omega^{*}_{x})\xbox \id)\mathrm{a}^{k_{x'}}_{k_{y}+2k_{x}}.
\end{equation*}
Recall that $w_x$ and $w_y$ are isometries. Combing the above two equalities we obtain \eqref{p:5}, as desired. \eqref{p:6} can be dealt with similarly.
\end{proof}
Now we are ready for the proof of the theorem.
\begin{proof}[Proof of Theorem \ref{thm:tkergodic_yetter}]
It remains to check that $\{\mathcal{A}^{w}_{v}\}_{v,w \in \rep(\G)}$ satisfies the axioms \ref{cond:A0}-\ref{cond:A5} of a compatible collection in Definition~\ref{def:compatible_collection}. Note that \ref{cond:A0} is exactly given by \eqref{eq:A0w}. Let us prove \ref{cond:A1}-\ref{cond:A5}.

(A1) For any $x, y \in \irr(\G)$, we have
\begin{align*}
\mathcal{A}^{\triv}_{y} & \;=\; \varphi(\id_{\C} \xbox \omega^{*}_{y} \xbox \id_{\C})\mathrm{a}^{0}_{k_{y}}\varphi(\omega_{y}) \overset{\ref{cond:a1}}{=} \varphi(\omega^{*}_{y})\varphi(\omega_{y}) = \id_{\varphi(y)}, \\
\mathcal{A}^{x}_{\triv} & \;=\; \varphi(\bar{\omega}^{*}_{x}c_{k_{x}} \xbox \id_{\C} \xbox \omega^{*}_{x})\mathrm{a}^{k_{x}}_{0}\varphi(\id_{\C}) \overset{\ref{cond:a1}}{=} \varphi(\bar{\omega}^{*}_{x}c_{k_{x}} \xbox \omega^{*}_{x})\varphi(R_{k_{x}}) = \varphi(R_{u^{x}}).
\end{align*}
Then
\begin{align*}
\mathcal{A}^{\triv}_{v} & = \sum_{j}\varphi(\id_{\C} \xbox \theta_{v_{j}} \xbox \id_{\C})\mathcal{A}^{\triv}_{v_{j}}\varphi(\theta^{*}_{v_{j}})  = \sum_{j}\varphi(\theta_{v_{j}})\varphi(\theta^{*}_{v_{j}}) = \id_{\varphi(v)}
\end{align*}
and
\begin{align*}
\mathcal{A}^{w}_{\triv} & = \sum_{i}\varphi(\bar{\theta}_{w_{i}} \xbox \id_{\C} \xbox \theta_{w_{i}})\mathcal{A}^{w_{i}}_{\triv} = \sum_{i}\varphi(\bar{\theta}_{w_{i}} \xbox \theta_{w_{i}})\varphi(R_{w_{i}}) = \varphi(R_{w}),
\end{align*}
for any $v,w \in \rep(\G)$.

(A2) For any $x,y,y' \in \irr(\G)$ and $T \in \mor_{\G}(y,y')$, we have
\begin{flalign*}
 \mathcal{A}^{x}_{y'}\varphi(T) & = \varphi(\bar{\omega}^{*}_{x}c_{k_{x}} \xbox \omega^{*}_{y'} \xbox \omega^{*}_{x})\mathrm{a}^{k_{x}}_{k_{y'}}\varphi(\omega_{y'})\varphi(T) \\
 & = \varphi(\bar{\omega}^{*}_{x}c_{k_{x}} \xbox \omega^{*}_{y'} \xbox \omega^{*}_{x})\varphi(\id \xbox \omega_{y'}T\omega^{*}_{y} \xbox \id)\mathrm{a}^{k_{x}}_{k_{y}}\varphi(\omega_{y})  \hskip0.2\textwidth  \text{(by \ref{cond:a2})}\\
 & = \varphi(\id \xbox T \xbox \id)\varphi(\bar{\omega}^{*}_{x}c_{k_{x}} \xbox \omega^{*}_{y} \xbox \omega^{*}_{x})\mathrm{a}^{k_{x}}_{k_{y}}\varphi(\omega_{y}) = \varphi(\id \xbox T \xbox \id)\mathcal{A}^{x}_{y}. 
\end{flalign*}
Now take $v,v',w \in \rep(\G)$ and $T \in \mor_{\G}(v,v')$. Consider the decomposition $w = \oplus_{i} w_{i}$, $v = \oplus_{j} v_{j}$, and $v' = \oplus_{k} v'_{k}$ into direct sums of irreducible representations with associated isometric intertwiners ${\theta_{w_{i}}\in\mor_{\G}(w_{i},w)}$, $\theta_{v_{j}}\in\mor_{\G}(v_{j},v)$ and $\theta_{v'_{k}}\in\mor_{\G}(v'_{k},v')$. Using the decompositions in \eqref{eq:A0w}, and replacing $x,y,y',T$ by $w,v_j,v_k ',\theta^{*}_{v'_{k}}T\theta_{v_{j}}$ respectively in the above equality, we get 
\[
\mathcal{A}^{w}_{v'_{k}}\varphi(\theta^{*}_{v'_{k}}T\theta_{v_{j}}) = \varphi(\id \xbox \theta^{*}_{v'_{k}}T\theta_{v_{j}} \xbox \id)\mathcal{A}^{w}_{v_{j}}
\]
and
\begin{align*}
\mathcal{A}^{w}_{v'}\varphi(T) & = \sum_{k}\varphi(\id \xbox \theta_{v'_{k}} \xbox \id)\mathcal{A}^{w}_{v'_{k}}\varphi(\theta^{*}_{v'_{k}})\varphi(T) \sum_{j}\varphi(\theta_{v_{j}})\varphi(\theta^{*}_{v_{j}}) \\
& = \sum_{j,k}\varphi(\id \xbox \theta_{v'_{k}} \xbox \id)\varphi(\id \xbox \theta^{*}_{v'_{k}} T\theta_{v_{j}} \xbox \id)\mathcal{A}^{w}_{v_{j}}\varphi(\theta^{*}_{v_{j}}) \\
& = \sum_{j}\varphi(\id \xbox \sum_{k}\theta_{v'_{k}}\theta^{*}_{v'_{k}} \xbox \id)\varphi(\id \xbox T\theta_{v_{j}} \xbox \id)\mathcal{A}^{w}_{v_{j}}\varphi(\theta^{*}_{v_{j}}) \\
& = \varphi(\id \xbox T \xbox \id)\mathcal{A}^{w}_{v}.
\end{align*}

(A3) For any $x,x',y \in \irr(\G)$, we have
\begin{align*}
&\ \quad\mathcal{A}^{x'}_{\bar{x} \boxtimes y \boxtimes x}\mathcal{A}^{x}_{y} \\
& = \varphi(\bar{\omega}^{*}_{x'}c_{k_{x'}} \xbox c_{x}\omega^{*}_{x} \xbox \omega^{*}_{y} \xbox \omega^{*}_{x} \xbox \omega^{*}_{x'})\mathrm{a}^{k_{x'}}_{k_{y}+2k_{x}}\varphi(\omega_{x}\omega^{*}_{x} \xbox \omega_{y}\omega^{*}_{y} \xbox \omega_{x}\omega^{*}_{x})\mathrm{a}^{k_{x}}_{k_{y}}\varphi(\omega_{y}) 
& \text{(by \eqref{p:5})}\\
& = \varphi(\bar{\omega}^{*}_{x'}c_{k_{x'}} \xbox c_{x}\omega^{*}_{x} \xbox \omega^{*}_{y} \xbox \omega^{*}_{x} \xbox \omega^{*}_{x'})\mathrm{a}^{k_{x'}}_{k_{y}+2k_{x}}\mathrm{a}^{k_{x}}_{k_{y}}\varphi(\omega_{y}) 
& \text{(by \ref{cond:a2})}\\
& = \varphi(\bar{\omega}^{*}_{x'}c_{k_{x'}} \xbox \bar{\omega}^{*}_{x}c_{k_{x}} \xbox \omega^{*}_{y} \xbox \omega^{*}_{x} \xbox \omega^{*}_{x'})\mathrm{a}^{k_{x'}+k_{x}}_{k_{y}}\varphi(\omega_{y})   & \text{(by \ref{cond:a3})} \\
& = \varphi((\bar{\omega}^{*}_{x'} \xbox \bar{\omega}^{*}_{x})\sigma_{k_{x},k_{x'}}(c_{k_{x} + k_{x'}}) \xbox \omega^{*}_{y} \xbox \omega^{*}_{x} \xbox \omega^{*}_{x'})\mathrm{a}^{k_{x'}+k_{x}}_{k_{y}}\varphi(\omega_{y}) 
&  \text{(by \ref{cond:c2})} \\
& = \varphi(\sigma_{x,x'} \xbox \id_{y,x,x'})\mathcal{A}^{x \boxtimes x'}_{y}. & \text{(by \eqref{p:4})}
\end{align*}
Then using \eqref{eq:dec} and \ref{cond:A2}, we may immediately pass this equality to direct sums of irreducible representations $x,x',y$, which yields the desired condition \ref{cond:A3}.

(A4) Fix $x,y,y' \in \irr(\G)$. Combing \eqref{p:6}, \ref{cond:2} and \ref{cond:a4} we may write
\begin{equation*}
\mathcal{A}^{x}_{y \boxtimes y'}\iota_{y,y'} = \varphi(\bar{\omega}^{*}_{x}c_{k_{x}} \xbox \omega^{*}_{y} \xbox \omega^{*}_{y'} \xbox \omega^{*}_{x})\varphi(\id \xbox R^{*}_{k_{x}}  \xbox \id)\iota_{k_{y}+2k_{x},k_{y'}+2k_{x}}(\mathrm{a}^{k_{x}}_{k_{y}} \otimes \mathrm{a}^{k_{x}}_{k_{y'}})(\varphi(\omega_{y}) \otimes \varphi(\omega_{y'})) .
\end{equation*} 
Consider the decomposition $u^{\boxtimes k_{x}} = \oplus z$ into direct sum of irreducible representations such that some $z$ in the decomposition equals $x$. We denote by $\theta_{z}\in\mor_{\G}(u^{z},u^{\boxtimes k_{x}})$ the corresponding intertwiners and in particular $\theta_x = \omega_x$. Note that $
\bar{R}_{u^{\boxtimes k_{x}}} = \sum_{z } (\theta_{z} \xbox \bar{\theta}_{z})\bar{R}_{u^{z}}$ and $ \bar{R}_{u^{\boxtimes k_{x}}} = (\id \xbox c_{k_{x}})R_{k_{x}}$. So we may write
\begin{equation}
\label{eq:R decomposition}
R^{*}_{k_{x}} = \sum_{z } \bar{R}^{*}_{u^z}(\theta^{*}_{z} \otimes \bar{\theta}^{*}_{z}c_{k_{x}}).
\end{equation}
Putting this into the previous equality we obtain
\begin{align*}
& \ \quad \mathcal{A}^{x}_{y \boxtimes y'}\iota_{y,y'} \\
&= \sum_{z}\varphi(\id \xbox \bar{R}^{*}_{u^{z}} \xbox \id)\varphi(\bar{\omega}^{*}_{x}c_{k_{x}} \xbox \omega^{*}_{y} \xbox \theta^{*}_{z} \xbox \bar{\theta}^{*}_{z}c_{k_{x}} \xbox \omega^{*}_{y'} \xbox \omega^{*}_{x})\iota_{k_{y}+2k_{x},k_{y'}+2k_{x}}(\mathrm{a}^{k_{x}}_{k_{y}}\varphi(\omega_{y}) \otimes \mathrm{a}^{k_{x}}_{k_{y'}}\varphi(\omega_{y'})) \\
& = \sum_{z}\varphi(\id \xbox \bar{R}^{*}_{u^{z}} \xbox \id)\iota_{\bar{x} \boxtimes y \boxtimes z,\bar{z} \boxtimes y' \boxtimes x}(\varphi(c_{x}\theta^{*}_{x} \xbox \omega^{*}_{y} \xbox \theta^{*}_{z})\mathrm{a}^{k_{x}}_{k_{y}}\varphi(\omega_{y}) \otimes \varphi(\bar{\theta}^{*}_{z}c_{k_{x}} \xbox \omega^{*}_{y'} \xbox \omega^{*}_{x})\mathrm{a}^{k_{x}}_{k_{y'}}\varphi(\omega_{y'})) \\
&= \varphi(\id \xbox \bar{R}^{*}_{u^{x}} \xbox \id)\iota_{\bar{x} \boxtimes y \boxtimes x,\bar{x} \boxtimes y' \boxtimes x}(\mathcal{A}^{x}_{y} \otimes \mathcal{A}^{x}_{y'}). \hskip0.4\textwidth \text{(by \eqref{p:1_2})}  
\end{align*}
Now consider direct sums of irreducible representations $w = \oplus_{i} w_{i}$, $v = \oplus_{j} v_{j}$, and $v' = \oplus_{k} v'_{k}$ with corresponding intertwiners $\theta_{w_{i}}\in\mor_{\G}(w_{i},w$), $\theta_{v_{j}}\in\mor_{\G}(v_{j},v)$ and $\theta_{v'_{k}}\in\mor_{\G}(v'_{k},v')$. Then together with \eqref{eq:dec}, we may pass the above equality to direct sums and obtain
\[\mathcal{A}^{w}_{v \boxtimes v'}\iota_{v,v'}=
\sum_{j,k,i}\varphi(\bar{\theta}_{w_{i}} \xbox \theta_{v_{j}} \xbox \theta_{v'_{k}} \xbox \theta_{w_{i}})\varphi(\id \xbox \bar{R}^{*}_{w_{i}} \xbox \id)\iota_{\bar{w}_{i}\boxtimes v_{j}\boxtimes w_{i},\bar{w}_{i}\boxtimes v'_{k}\boxtimes w_{i}}(\mathcal{A}^{w_{i}}_{v_{j}}\varphi(\theta^{*}_{v_{j}} ) \otimes \mathcal{A}^{w_{i}}_{v'_{k}} \varphi(\theta^{*}_{v'_{k}})).
\]
Note that $
\bar{R}_{w} = \sum_{i}(\theta_{w_{i}} \xbox \bar{\theta}_{w_{i}})\bar{R}_{w_{i}},
$ and hence $
\bar{R}^{*}_{w}(\theta_{w_{i}} \otimes \bar{\theta}_{w_{i'}}) = \delta_{i,i'}\bar{R}^{*}_{w_{i}}$
for all $i,i'$. Thus we have
\begin{align*}
& \mathcal{A}^{w}_{v \boxtimes v'}\iota_{v,v'} = \varphi(\id \xbox \bar{R}^{*}_{w} \xbox \id) \,\cdot \\
&\qquad \sum_{j,k,i,i'}\varphi(\bar{\theta}_{w_{i}} \xbox \theta_{v_{j}} \xbox \theta_{w_{i}} \xbox \bar{\theta}_{w_{i'}} \xbox \theta_{v'_{k}} \xbox \theta_{w_{i'}})\iota_{\bar{w}_{i}\boxtimes v_{j}\boxtimes w_{i},\bar{w}_{i'}\boxtimes v'_{k}\boxtimes w_{i'}}(\mathcal{A}^{w_{i}}_{v_{j}}\varphi(\theta^{*}_{v_{j}} ) \otimes \mathcal{A}^{w_{i}}_{v'_{k}} \varphi(\theta^{*}_{v'_{k}})) \\
& = \varphi(\id \xbox \bar{R}^{*}_{w} \xbox \id)\iota_{\bar{w}\boxtimes v\boxtimes w,\bar{w}\boxtimes v'\boxtimes w}\,\cdot\\
&\qquad\sum_{j,k,i,i'}(\varphi(\bar{\theta}_{w_{i}} \xbox \theta_{v_{j}} \xbox \theta_{w_{i}}) \otimes \varphi(\bar{\theta}_{w_{i'}} \xbox \theta_{v'_{k}} \xbox \theta_{w_{i'}}))(\mathcal{A}^{w_{i}}_{v_{j}}\varphi(\theta^{*}_{v_{j}} ) \otimes \mathcal{A}^{w_{i}}_{v'_{k}} \varphi(\theta^{*}_{v'_{k}})) \\
& = \varphi(\id \xbox \bar{R}^{*}_{w} \xbox \id)\iota_{\bar{w}\boxtimes v\boxtimes w,\bar{w}\boxtimes v'\boxtimes w}(\mathcal{A}^{w}_{v} \otimes \mathcal{A}^{w}_{v'}).
\end{align*}

(A5) Let $x,y \in \irr(\G)$. Consider the intertwiner $\bar{\omega}^{*}_{x} \xbox \omega^{*}_{y} \xbox \omega^{*}_{x}\in\mor_{\G}( \overline{u^{\boxtimes k_{x}}} \boxtimes u^{\boxtimes k_{y}} \boxtimes u^{\boxtimes k_{x}},\bar{x} \boxtimes y \boxtimes x)$. We see that
\begin{align*}
&\ \quad \varphi(\sigma_{\bar{x},y,x})J_{\bar{x} \boxtimes y \boxtimes x}\mathcal{A}^{x}_{y} \\
& = \varphi(\bar{\omega}^{*}_{x} \xbox \bar{\omega}^{*}_{y} \xbox \omega^{*}_{x})\varphi(\sigma_{k_{x},k_{y},k_{x}})J_{\overline{u^{\boxtimes k_{x}}} \boxtimes u^{\boxtimes k_{y}} \boxtimes u^{\boxtimes k_{x}}}\varphi(c_{k_{x}} \xbox \id)\mathrm{a}^{k_{x}}_{k_{y}}\varphi(\omega_{y})
& \text{(by Lemma \ref{lem:conjugate} \ref{lem:J_naturality})} \\
& = \varphi(\bar{\omega}^{*}_{x} \xbox \bar{\omega}^{*}_{y} \xbox \omega^{*}_{x})\varphi(c_{k_{x}} \xbox c_{k_{y}} \xbox \id)J_{k_{y}+2k_{x}}\mathrm{a}^{k_{x}}_{k_{y}}\varphi(\omega_{y}) 
& \text{(by \ref{cond:c3})}\\
& = \varphi(\bar{\omega}^{*}_{x}c_{k_{x}} \xbox \bar{\omega}^{*}_{y}c_{k_{y}} \xbox \omega^{*}_{x})\mathrm{a}^{k_{x}}_{k_{y}}J_{k_{y}}\varphi(\omega_{y}) 
& \text{(by \ref{cond:a5})}\\
& = \varphi(\bar{\omega}^{*}_{x}c_{k_{x}} \xbox \bar{\omega}^{*}_{y}c_{k_{y}} \xbox \omega^{*}_{x})\mathrm{a}^{k_{x}}_{k_{y}}\varphi(c^{*}_{k_{y}}\bar{\omega}_{y})J_{y} 
& \text{(by \ref{cond:c1})}\\
& = \mathcal{A}^{x}_{\bar{y}}J_{y}.
\end{align*}
Now consider $v,w \in \rep(\G)$ and the decompositions $w = \oplus_{i} w_{i}$ and $v = \oplus_{j} v_{j}$ into direct sums of irreducible representations with intertwiners $\omega_{w_{i}}\in\mor_{\G}(w_{i},w)$ and $\theta_{v_{j}}\in\mor_{\G}(v_{j},v)$. Then the above equality passes to direct sums and implies that 
\begin{align*}
\varphi(\sigma_{\bar{w},v,w})J_{\bar{w} \boxtimes v \boxtimes w}\mathcal{A}^{w}_{v} 
& = \sum_{i,j}\varphi(\sigma_{\bar{w},v,w})\varphi(\overline{\bar{\theta}_{w_{i}} \xbox \theta_{v_{j}} \xbox \theta_{w_{i}}})J_{\bar{w}_{i} \boxtimes v_{j} \boxtimes w_{i}}\mathcal{A}^{w_{i}}_{v_{j}}\varphi(\theta^{*}_{v_{j}}) 
&& \text{(by \ref{eq:inv_map})} \\
& = \sum_{i,j}\varphi(\bar{\theta}_{w_{i}} \xbox \theta_{\bar{v}_{j}} \xbox \theta_{w_{i}})\varphi(\sigma_{\bar{w}_{i},v_{j},w_{i}})J_{\bar{w}_{i} \boxtimes v_{j} \boxtimes w_{i}}\mathcal{A}^{w_{i}}_{v_{j}}\varphi(\theta^{*}_{v_{j}}) \\
& = \sum_{i,j}\varphi(\bar{\theta}_{w_{i}} \xbox \theta_{\bar{v}_{j}} \xbox \theta_{w_{i}})\mathcal{A}^{w_{i}}_{\bar{v}_{j}}J_{v_{j}}\varphi(\theta^{*}_{v_{j}}) \\
& = \mathcal{A}^{w}_{\bar{v}}J_{v}.
&& \text{(by \ref{eq:inv_map})}
\end{align*}

Now the proof of the theorem is complete.
\end{proof}


Based on the previous theorem, we may further characterize the braided commutativity of the Yetter-Drinfeld C*-algebra.

\begin{coro}\label{cor:bc}
We keep the assumptions and notations of Theorem~\ref{thm:tkergodic_yetter}, and assume furthermore that for all $k,n\in\N$, we have 
\begin{equation}\label{cond:na6}
\varphi(R^{*}_{n+1} \xbox \id_{k,n+1})\iota_{n+1,k+2(n+2)}(\id \otimes \mathrm{a}^{n+1}_{k}) = \iota_{k,n+1}\Sigma_{n+1,k}. \tag{a5}
\end{equation}
Then the Yetter-Drinfeld $\G$-C*-algebra $(B, \lhd, \alpha)$ is braided commutative.
\end{coro}

\begin{proof}
By Theorem~\ref{thm:tkergodic_yetter}, we know that there is a compatible collection $\{\mathcal{A}^{w}_{v}\}_{w,v \in \rep(\G)}$ for the functor $\varphi$. It is enough to prove that \eqref{cond:na6} implies the braidedness \ref{cond:A6} of $\{\mathcal{A}^{w}_{v}\}_{w,v \in \rep(\G)}$ in  Definition~\ref{def:compatible_collection}. Consider first $x,y \in \irr(\G)$ and decompositions $u^{\boxtimes k_{x}} = \oplus z$ into direct sums of irreducible representations such that some $z$ in the decomposition equals $x$. We denote by $\theta_{z}\in\mor_{\G}(u^{z},u^{\boxtimes k_{x}})$ the associated isometric intertwiners. Then $\theta_z^* w_x = \delta_{z,x} \mathrm{id}$ and we may formally write
\begin{align*}
&\ \quad \varphi(\bar{R}^{*}_{u^{x}} \xbox \id_{y,x})\iota_{x,\bar{x} \boxtimes y \boxtimes x}(\id \otimes \mathcal{A}^{x}_{y}) \\
& = \varphi(\bar{R}^{*}_{u^{x}} \xbox \id_{y,x})\iota_{x,\bar{x} \boxtimes y \boxtimes x}(\id \otimes \varphi(\bar{\omega}^{*}_{x}c_{k_{x}} \xbox \omega^{*}_{y} \xbox \omega^{*}_{x})\mathrm{a}^{k_{x}}_{k_{y}}\varphi(\omega_{y})) \\
& = \sum_{z}\varphi(\bar{R}^{*}_{u^{z}} \xbox \id_{y,x})\varphi(\theta^{*}_{z} \xbox \bar{\theta}^{*}_{z}c_{k_{x}} \xbox \omega^{*}_{y} \xbox \omega^{*}_{x})\iota_{k_{x},k_{y}+2k_{x}}(\varphi(\omega_{x}) \otimes \mathrm{a}^{k_{x}}_{k_{y}}\varphi(\omega_{y})) 
\end{align*}
As explained in \eqref{eq:R decomposition}, we have $
R^{*}_{k_{x}} = \sum_{z } \bar{R}^{*}_{u^{z}}(\theta^{*}_{z} \xbox \bar{\theta}^{*}_{z}c_{k_{x}})
$. Together with \eqref{cond:na6} we may rewrite the above equality to obtain
\begin{align*}
\varphi(\bar{R}^{*}_{u^{x}} \xbox \id_{y,x})\iota_{x,\bar{x} \boxtimes y \boxtimes x}(\id \otimes \mathcal{A}^{x}_{y}) & = \varphi(\omega^{*}_{y} \xbox \omega^{*}_{x})\varphi(R^{*}_{k_{x}} \xbox \id_{k_{y},k_{x}})\iota_{k_{x},k_{y}+2k_{x}}(\id \otimes \mathrm{a}^{k_{x}}_{k_{y}})(\varphi(\omega_{x}) \otimes \varphi(\omega_{y})) \\
&= \varphi(\omega^{*}_{y} \xbox \omega^{*}_{x})i_{k_{y},k_{x}}\Sigma(\varphi(\omega_{x}) \otimes \varphi(\omega_{y})) \\
&= \iota_{y,x}\Sigma.
\end{align*}
Now consider $v,w\in \rep(\G)$ and the decompositions $w = \oplus_{i} w_{i}$ and $v = \oplus_{j} v_{j}$ into direct sums of irreducible representations with intertwiners $\omega_{w_{i}}\in\mor_{\G}(w_{i},w)$ and $\theta_{v_{j}}\in\mor_{\G}(v_{j},v)$. We may write $\mathrm{id}=\sum_{i'} \varphi(\theta_{w_{i'}}\theta^{*}_{w_{i'}})$ and further
\begin{align*}
\quad \iota_{w,\bar{w} \boxtimes v \boxtimes w}(\id \otimes \mathcal{A}^{w}_{v}) &= \iota_{w,\bar{w} \boxtimes v \boxtimes w}\Bigg(\id \otimes \Bigg(\sum_{i,j}\varphi(\bar{\theta}_{w_{i}} \xbox \theta_{v_{j}} \xbox \theta_{w_{i}})\mathcal{A}^{w_{i}}_{v_{j}}\varphi(\theta^{*}_{v_{j}})\Bigg)\Bigg) \\
&= \sum_{i,j,i'}\varphi(\theta_{w_{i'}} \xbox \bar{\theta}_{w_{i}} \xbox \theta_{v_{j}} \xbox \theta_{w_{i}})\iota_{w_{i'},\bar{w}_{i} \boxtimes v_{j} \boxtimes w_{i}}(\id \otimes \mathcal{A}^{w_{i}}_{v_{j}})(\varphi(\theta^{*}_{w_{i'}}) \otimes \varphi(\theta^{*}_{v_{j}}))
\end{align*}
Note also that
$
\bar{R}_{w} = \sum_{i}(\theta_{w_{i}} \xbox \bar{\theta}_{w_{i}})\bar{R}_{w_{i}},
$ and hence
$
\bar{R}^{*}_{w} = \sum_{i}\bar{R}^{*}_{w_{i}}(\theta^{*}_{w_{i}} \xbox \bar{\theta}^{*}_{w_{i}}) $.
So together with the previous equality for irreducible representations $x,y\in\irr(\G)$, we obtain
\begin{align*}
\quad \varphi(\bar{R}^{*}_{w} \xbox \id)\iota_{w,\bar{w} \boxtimes v \boxtimes w}(\id \otimes \mathcal{A}^{w}_{v}) &= \sum_{i,j}\varphi(\bar{R}^{*}_{w_{i}} \xbox \theta_{v_{j}} \xbox \theta_{w_{i}})\iota_{w_{i},\bar{w}_{i} \boxtimes v_{j} \boxtimes w_{i}}(\id \otimes \mathcal{A}^{w_{i}}_{v_{j}})(\varphi(\theta^{*}_{w_{i}}) \otimes \varphi(\theta^{*}_{v_{j}})) \\
& = \sum_{i,j}\varphi(\theta_{v_{j}} \xbox \theta_{w_{i}})\iota_{v_{j},w_{i}}\Sigma(\varphi(\theta^{*}_{w_{i}}) \otimes \varphi(\theta^{*}_{v_{j}})) \\
& = \iota_{v,w}\Sigma. \qedhere
\end{align*}
\end{proof}

The reader might keep in mind the following basic example illustrating Theorem \ref{thm:tkergodic_yetter} and Corollary \ref{cor:bc}.
\begin{exe}\label{ex:c(G)}
Set $K_{n} = (\C^{N})^{\otimes n}$ for all $n\in \N$ and let $\varphi : \mor_{\G}(k, l)\to \B(K_{k}, K_{l})$ be the identity map. Consider the collection of maps
\begin{equation*}
\mathrm{a}_{k}:K_{k} \mapsto K_{k+2},\quad  e_{\mathbf{i}} \mapsto \sum_{j =1}^N e_{j} \otimes e_{\mathbf{i}} \otimes e_{j},\quad k \in \N.
\end{equation*}
It is straightforward to check that those maps satisfy the conditions \ref{cond:na2_new} to \eqref{cond:na6}, yielding a braided commutative Yetter-Drinfeld $\G$-C*-algebra structure on $B = C(\G)$ with action $\alpha = \com_{\G}$. Moreover, the Hopf $*$-algebra action $\lhd: \pol(\G) \otimes \pol(\G) \to \pol(\G)$ is the adjoint action $a \lhd b = S(a_{(1)})ba_{(2)}$. 
Indeed,
\begin{align*}
u_{\mathbf{i}\mathbf{i}'} \lhd u_{jj'} & = (\langle \;\cdot\;(e_{j} \otimes e_{\mathbf{i}'} \otimes e_{j'}),\mathrm{a}_{k}(e_{\mathbf{i}})\rangle \otimes \id)(u^{\boxtimes (k+2)}) \\
& = \sum_{k \in \{1,\dots,N\}} (\langle \;\cdot\;(e_{j} \otimes e_{\mathbf{i}'} \otimes e_{j'}),e_{k} \otimes e_{\mathbf{i}} \otimes e_{k}\rangle \otimes \id)(u^{\boxtimes (k+2)}) \\
& = \sum_{k \in \{1,\dots,N\}} u_{kj}u_{\mathbf{i}\mathbf{i}'}u_{kj'} = S(u_{jj'_{(1)}})u_{\mathbf{i}\mathbf{i}'}u_{jj'_{(2)}},
\end{align*}
for all $\mathbf{i},\mathbf{i}' \in \{1,\dots,N\}^{k}$, $j,j' \in \{1,\dots,N\}$.
\end{exe}

\section{Applications to easy quantum groups}\label{sec:examples}

In this section, we will show how Theorem \ref{thm:tkergodic} can be applied in the setting of easy quantum groups to produce new ergodic actions out of combinatorial data. The idea is to build a Hilbert space from partition vectors or linear maps and then to use the vertical concatenation of partitions to produce an action. We will describe in the sequel two ways of doing this. In both cases, we are not able to describe the corresponding action in full generality, but some particular cases are worked out. Besides, many of these constructions can be understood through the induction procedure, which will be explained in Section \ref{sec:induction}.

\subsection{Projective partitions}\label{subsec:projectivepartitions}

Let us fix a category of partitions $\CC$ and an integer $N$ once and for all and we write $H = \C^{N}$. Our first construction relies on specific partitions which we now define.

\begin{defi}
A partition $p\in \CC$ is said to be \emph{projective} if $p^{2} = p = p^{*}$. The set of projective partitions in $\CC(k,k)$ (projectivity forces the number of upper and lower points to coincide) will be denoted by $\Proj_{\CC}(k)$.
\end{defi}

If $p$ is projective, then the corresponding linear map $T_{p}$ is \emph{not} necessarily a projection, because $T_{p}\circ T_{p} = N^{\rl(p, p)}T_{p^{2}}$. However, it can be normalized to become a projection. The correct normalization is the following one, as explained in \cite[Def 2.15 and Prop 2.18]{FW16}:
\begin{equation*}
Q_{p} = N^{-\beta(p)/2}T_{p},
\end{equation*}
where $\beta(p)$ is the number of \emph{non-through-blocks} of $p$. Via that normalization, projective partitions can be though of as projections, which leads to the definition of a specific order relation: we say that $q$ is \emph{dominated by $p$}, written $q\preceq p$, if $pq = q$ (in which case $qp = q$ also by applying the reflection operation).

\subsubsection{The construction}

We will now construct a family of vector spaces out of projective partitions in the following way. Given a set $\PP_{n}\subset P(n, n)$ of projective partitions and $p\in \PP_{n}$, we set
\begin{equation*}
\eta_{p} = T_{p}(e_{1}^{\otimes n})
\end{equation*}
and 
\begin{equation*}
K_{n}^{\PP} = \spa\{\eta_{p} \mid p\in \PP_{n}\}\subset H^{\otimes n}.
\end{equation*}
The Hilbert space structure on $K_{n}^{\PP}$ can be read directly on partitions through the following formula:
\begin{equation*}
\langle \eta_{p}, \eta_{q}\rangle = N^{\rl(q, p)}.
\end{equation*}
If we have a set $\PP = \cup_{n}\PP_{n}$, then the formula $\eta_{p}\otimes\eta_{q} = \eta_{p\odot q}$ provides isometries
\begin{equation*}
\iota_{k, l} : K_{k}^{\PP}\otimes K_{l}^{\PP}\to K_{k+l}^{\PP}
\end{equation*}
for all $k$ and $l$ provided $\PP_{k}\odot\PP_{l}\subset \PP_{k+l}$. To produce an ergodic action from this using our machinery, we will need the following result from \cite[Lem 2.11]{FW16}:

\begin{propo}
A partition $p$ is projective if and only if it is of the form $r^{*}r$ for some partition $r$.
\end{propo}

Let us consider now the easy quantum group $\G_{N}(\CC)$. The spaces $\mor_{\G_{N}(\CC)}(k, l)$ are linearly spanned by the partition operators $T_{r}$, hence to define a functor we only have to define a linear map for each partition $r\in \CC(k, l)$. We will do this thanks to the following formula: for $p\in \PP_{n}$,
\begin{equation*}
\varphi(r)(\eta_{p}) = N^{\rl(r, p)}\eta_{rpr^{*}}.
\end{equation*}
That this map is well-defined is not clear if the vectors $\eta_{p}$ are not linearly independent. To prove it, we will need yet another result on partitions.

\begin{lem}\label{lem:removedloopsprojective}
Let $p\in P(k, k)$ be a projective partition. Then, for any $q\in P(l, l)$ and $r\in P(k, l)$,
\begin{equation*}
\rl(q, rp) = \rl(q, rpr^{*}).
\end{equation*}
As a consequence, we have the equality
\begin{equation*}
\eta_{rp} = \eta_{rpr^{*}}.
\end{equation*}
\end{lem}

\begin{proof}
We will show that the lower non-through-blocks of $rp$ are the same as those of $rpr^{*}$, from which the result follows. Observe that by projectivity of $p$, $rpr^{*} = (rp)(rp)^{*}$. Thus, if a point $i$ in the lower row of $rp$ is connected to an upper point $j$, then the lower point corresponding to $j$ is connected in $(rp)^{*}$ to the upper point corresponding to $i$. Thus, any point of a through-block of $rp$ belongs to a through-block of $rpr^{*}$. The converse being clear, the proof is complete.

Note that this result can also be written as the equality
\begin{equation*}
\langle \eta_{rp}, \eta_{q}\rangle = \langle \eta_{rpr^{*}}, \eta_{q}\rangle.
\end{equation*}
From this and the fact that all the inner products involved are real-valued, hence symmetric, we get
\begin{equation*}
\langle \eta_{rpr^{*}}, \eta_{rpr^{*}}\rangle = \langle \eta_{rpr^{*}}, \eta_{rp}\rangle = \langle \eta_{rp}, \eta_{rp}\rangle.
\end{equation*}
Developping $\|\eta_{rp} - \eta_{rpr^{*}}\|^{2}$ and using the previous equalities then yields $\eta_{rp} = \eta_{rpr^{*}}$.
\end{proof}

We can now turn to the definition of $\varphi(r)$.

\begin{propo}\label{prop:welldefined}
The maps $\varphi(r)$ are well-defined.
\end{propo}

\begin{proof}
We have to prove that if $x = \sum \lambda_{p}\eta_{p} = 0$, then $\widetilde{x} = \sum \lambda_{p}N^{\rl(r, p)}\eta_{rpr^{*}}$ also vanishes. This follows from a computation using Lemma \ref{lem:removedloopsprojective},
\begin{align*}
\left\|\sum_{p}\lambda_{p}N^{\rl(r, p)}\eta_{rpr^{*}}\right\|^{2} & = \sum_{p, q}\lambda_{p}\overline{\lambda}_{q}N^{\rl(r, p)}N^{\rl(r, q)}\langle \eta_{rpr^{*}}, \eta_{rqr^{*}}\rangle \\
& = \sum_{p, q}\lambda_{p}\overline{\lambda}_{q}N^{\rl(r, p)}N^{\rl(r, q)}\langle \eta_{rp}, \eta_{rqr^{*}}\rangle \\
& = \sum_{p, q}\lambda_{p}\overline{\lambda}_{q}N^{\rl(r, p)}N^{\rl(r, q)}\langle \eta_{rp}, \eta_{rq}\rangle \\
& = \left\|\sum_{p}\lambda_{p}N^{\rl(r, p)}\eta_{rp}\right\|^{2} \\
& = \left\|T_{r}\left(\sum_{p}\lambda_{p}\eta_{p}\right)\right\|^{2} \\
& = 0 \qedhere
\end{align*}
\end{proof}

The only issue with the maps $\varphi(r)$ is that it is not clear that their range is in $K_{l}^{\PP}$. Ensuring this requires an additional assumption on $\PP$, and it is time to gather all the needed properties into a definition. This however requires a few notations: let us denote by $D_{m}\in NC_{2}(0, 2m)$ the pair partition with blocks $(i, 2m-i+1)$ and let $\overline{p}$ the partition obtained by rotating $p$ upside-down, that is to say
\begin{equation*}
\overline{p} = \left(D_{m}^{*}\odot \vert^{\odot m}\right)\left(\vert^{\odot m}\odot p \odot \vert^{\odot m}\right)\left(\vert^{\odot m}\odot D_{m}\right).
\end{equation*}

\begin{defi}
Let $\CC$ be a category of partitions. A \emph{module of projective partitions} over $\CC$ is a set $\PP$ of projective partitions such that, writing $\PP_{k} = \PP\cap P(k, k)$,
\begin{enumerate}
\item $\PP_{k}\odot \PP_{l}\subset \PP_{k+l}$ for all $k, l\in \N$;
\item For any $r\in \CC(k, l)$ and $p\in \PP_{k}$, $rpr^{*}\in \PP_{l}$;
\item For any $k\in \N$ and $p\in \PP_{k}$, $\overline{p}\in \PP_{k}$.
\end{enumerate}
\end{defi}

With this in hand, we can prove the main result of this section.

\begin{theo}\label{thm:projectivepartitionaction}
Let $\CC$ be a category of partitions, let $\PP$ be a module of projective partitions over $\CC$ and let $N$ be an integer. Then, the spaces $(K_{n}^{\PP})_{n\in \N}$ and the map $\varphi$ satisfy the assumptions of Theorem \ref{thm:tkergodic}. The corresponding weak unitary tensor functor will be denoted by $\F_{N}^{\CC, \PP}$.
\end{theo}

\begin{proof}
The first step is to show that there is for each $k, l\in \N$ a well-defined linear map
\begin{equation*}
\varphi : \spa\{T_{r} \mid r\in \CC(k, l)\} \to \B(K_{k}^{\PP}, K_{l}^{\PP})
\end{equation*}
which sends $T_{r}$ to $\varphi(r)$. In other words, we have to show that if $\sum_{r}\lambda_{r}T_{r} = 0$, then so does $\sum_{r}\lambda_{r}\varphi(r)$. And indeed, for any $p\in \PP_{k}$,
\begin{align*}
\left\|\left(\sum_{r}\lambda_{r}\varphi(r)\right)(\eta_{p})\right\|^{2} & = \sum_{r, s}\lambda_{r}\overline{\lambda}_{s}N^{\rl(r, p) + \rl(s, p)}\langle \eta_{rpr^{*}}, \eta_{sps^{*}}\rangle \\
& = \sum_{r, s}\lambda_{r}\overline{\lambda}_{s}N^{\rl(r, p) + \rl(s, p)}\langle \eta_{rp}, \eta_{sp}\rangle \\
& = \sum_{r, s}\lambda_{r}\overline{\lambda}_{s}\langle T_{r}(\eta_{p}), T_{s}(\eta_{p})\rangle \\
& = \left\|\left(\sum_{r}\lambda_{r}T_{r})\right)(\eta_{p})\right\|^{2} \\
& = 0
\end{align*}

We now have to check that all the assumptions of Theorem \ref{thm:tkergodic} are satisfied by the maps $\varphi(r)$. First, by Lemma \ref{lem:removedloopsprojective}, $\rl(r, sps^{*}) = \rl(r, sp)$ so that
\begin{align*}
\varphi(r)\circ\varphi(s)(\eta_{p}) & = N^{\rl(s, p) + \rl(r, sps^{*})}\eta_{rsps^{*}r^{*}} = N^{\rl(s, p) + \rl(r, sp)}\eta_{rsps^{*}r^{*}} \\
& = N^{\rl(r, s) + \rl(rs, p)}\eta_{(rs)p(rs)^{*}} = N^{\rl(r, s)}\varphi(rs)(\eta_{p})
\end{align*}
and since $T_{rs} = N^{\rl(r, s)}T_{r}\circ T_{s}$, this shows the compatibility with composition of morphisms. Second, using again Lemma \ref{lem:removedloopsprojective} we have
\begin{align*}
\langle \varphi(r)(\eta_{p}), \eta_{q}\rangle & = N^{\rl(r, p) + \rl(q^{*}, rpr^{*})} = N^{\rl(r, p) + \rl(q^{*}, rp)} \\
& = N^{\rl(q^{*}, r) + (q^{*}r, p)} = N^{\rl(r^{*}, q) + (r^{*}q^{*}r, p)} \\
& = \langle \eta_{p}, \varphi(r^{*})(\eta_{q})\rangle
\end{align*}
and since $T_{r^{*}} = T_{r}^{*}$, this shows the compatibility with taking adjoints. Eventually,
\begin{align*}
\varphi(r\odot s)\circ\iota_{k, l}(\eta_{p}\otimes \eta_{q}) & = N^{\rl(r\odot s, p\otimes q)}\eta_{(r\odot s)^{*}(p\odot q)(r\odot s)} \\
& = N^{\rl(r, p) + \rl(s, q)}(\eta_{(rpr^{*})\odot (sqs^{*})}) \\
& = \iota_{k',l'}\circ(\varphi(r)\otimes \varphi(s))(\eta_{p}\otimes \eta_{q}) 
\end{align*}
and since $T_{r\odot s} = T_{r}\otimes T_{s}$, this shows the compatibility with tensor products.

As for Condition \ref{cond:3}, first observe that the range of $\iota_{k, l}$ is nothing but the inclusion
\begin{equation*}
K_{k}^{\PP}\otimes K_{l}^{\PP}\subset H^{\otimes k}\otimes H^{\otimes l} = H^{\otimes (k+l)}.
\end{equation*}
As a consequence, the corresponding range projection $P_{k, l}$ is given by
\begin{equation*}
P_{k, l} = P_{K^{\PP}_{k}}\otimes P_{K^{\PP}_{l}}
\end{equation*}
and the range projection of $\iota_{k, l, m}$ is simply $P_{K^{\PP}_{k}}\otimes P_{K^{\PP}_{l}}\otimes P_{K^{\PP}_{m}}$. To conclude, it is enough to show that for any $l, m\in \N$,
\begin{equation}\label{eq:conditionF3projective}
(\id_{H^{\otimes l}}\otimes P_{K_{m}^{\PP}})P_{K^{\PP}_{l+m}} = P_{K_{l}^{\PP}}\otimes P_{K_{m}^{\PP}}.
\end{equation}
Indeed, we will then have, for any $k, l, m\in \N$,
\begin{align*}
P_{k+l, m}\circ P_{k, l+m} & = \left(P_{K^{\PP}_{k+l}}\otimes P_{K^{\PP}_{m}}\right)\circ\left(P_{K^{\PP}_{k}}\otimes P_{K^{\PP}_{l+m}}\right) \\
& = \left(P_{K^{\PP}_{k+l}}\otimes \id_{H^{\otimes m}}\right)\circ\left(\id_{H^{\otimes k}}\otimes\id_{H^{\otimes l}}\otimes P_{K^{\PP}_{m}}\right)\circ\left(P_{K^{\PP}_{k}}\otimes P_{K^{\PP}_{l+m}}\right) \\
& = \left(P_{K^{\PP}_{k+l}}\otimes \id_{H^{\otimes m}}\right)\circ\left(P_{K^{\PP}_{k}}\otimes\id_{H^{\otimes l}}\otimes \id_{H^{\otimes m}}\right)\circ\left(\id_{H^{\otimes k}}\otimes \left(\id_{H^{\otimes l}}\otimes P_{K^{\PP}_{m}}\right)\circ P_{K^{\PP}_{l+m}}\right) \\
& = \left(P_{K^{\PP}_{k+l}}\otimes \id_{H^{\otimes m}}\right)\circ\left(P_{K^{\PP}_{k}}\otimes P_{K_{l}^{\PP}}\otimes P_{K_{m}^{\PP}}\right) \\
& = P_{K^{\PP}_{k}}\otimes P_{K_{l}^{\PP}}\otimes P_{K_{m}^{\PP}}.
\end{align*}
To prove Equation \eqref{eq:conditionF3projective}, let us fix $p\in \PP_{m}$ and $q\in \PP_{l+m}$ and prove that $(\id_{H^{\otimes l}}\otimes \eta_{p}^{*})(\eta_{q})\in K_{l}^{\PP}$, which is enough to conclude. Let us now set
\begin{equation*}
p\vdash q  = \left(\vert^{\odot l}\odot D_{m}^{*}\right)\left(q\odot \overline{p}\right)\left(\vert^{\odot l}\odot D_{m}\right) \in \PP_{l}.
\end{equation*}
Then,
\begin{align*}
(\id_{H^{\otimes l}}\otimes \eta_{p}^{*})(\eta_{q}) & = \sum_{i_{1}, \dots, i_{l+m}} \delta_{q}(\mathbf{1},\mathbf{i})\langle \eta_{p}, e_{i_{l+1}}\otimes\cdots\otimes e_{i_{l+m}}\rangle e_{i_{1}}\otimes \cdots \otimes e_{i_{l}} \\
& = \sum_{i_{1}, \dots, i_{l+m}}\sum_{j_{1}, \dots, j_{m}}\delta_{q}(\mathbf{1},\mathbf{i})\delta_{p}(\mathbf{1},\mathbf{j})\langle e_{j_{1}}\otimes \cdots\otimes e_{j_{m}}, e_{i_{l+1}}\otimes\cdots\otimes e_{i_{l+m}}\rangle e_{i_{1}}\otimes \cdots \otimes e_{i_{l}} \\
& = \sum_{i_{1}, \dots, i_{l+m}}\delta_{q}(\mathbf{1},\mathbf{i})\delta_{p}(\mathbf{1}, i_{l+1}\cdots i_{l+m}) e_{i_{1}}\otimes \cdots \otimes e_{i_{l}} \\
& = N^{\rl(\vert^{\odot l}\odot D_{m}^{*}, q\odot \overline{p})}\eta_{\left(\vert^{\odot l}\odot D_{m}^{*}\right)\left(q\odot \overline{p}\right)} \\
& = N^{\rl(\vert^{\odot l}\odot D_{m}^{*}, q\odot \overline{p})}\eta_{p\vdash q},
\end{align*}
where we used Lemma \ref{lem:removedloopsprojective} in the last line. This is an element of $K_{l}^{\PP}$, hence the proof is complete.
\end{proof}

As a consequence of Theorem \ref{thm:projectivepartitionaction}, there is a C*-algebra $B_{N}(\CC, \PP)$ with an ergodic action $\alpha_{N}^{\CC, \PP}$ of $\G_{N}(\CC)$ whose spectral functor is isomorphic to $\F_{N}^{\CC, \PP}$. It is difficult to obtain a manageable description of the C*-algebra $B_{N}(\CC, \PP)$ in general, even though we will show how to do this for a specific choice of $\PP$ when $\CC$ is assumed to be non-crossing, but we would like to mention one case where we have a general result.

\begin{propo}\label{prop:projectivequotient}
Let $\CC$ be a category of partitions, let $N$ be an integer and set
\begin{equation*}
\Proj_{\CC'}^{0} = \{p\in \Proj_{\CC'} \mid t(p) = 0\}
\end{equation*}
for some category of partitions $\CC\subset \CC'$. Then, $\Proj_{\CC'}^{0}$ is a module of projective partitions over $\CC$ and there is a $\G_{N}(\CC)$-equivariant $*$-isomorphism
\begin{equation*}
B_{N}(\CC, \Proj_{\CC'}^{0}) \cong C(\G_{N}(\CC')\backslash\G_{N}(\CC)),
\end{equation*}
where the action on the right-hand side is the translation action.
\end{propo}

The proof relies on induction and will therefore be postponed to Section \ref{sec:induction}. This nevertheless shows that our construction can produce many interesting actions with just one specific choice of $\PP$.

\subsubsection{The noncrossing case}

The simplest choice of a module of projective partitions is certainly $\PP = \Proj_{\CC}$. In that case, we can describe the C*-algebra and its action in terms of the ``first column space'' of $\G_{N}(\CC)$. More precisely, let $A_{N}(\CC)$ be the C*-subalgebra of $C(\G_{N}(\CC))$ generated by the elements $(u_{i1})_{1\leqslant i\leqslant N}$. Equivalently, $A_{N}(\CC)$ is the linear space spanned by the elements $a_{\mathbf{i}} = u_{\mathbf{1}\mathbf{i}}$ for all $\mathbf{i}\in \{1, \dots, k\}^{N}$. The restriction of the coproduct yields an action of $\G_{N}(\CC)$ which is easily seen to be ergodic. We will now show that the spectral functor of this action is isomorphic to $\F_{N}^{\CC, \Proj_{\CC}}$, at least when the partitions are non-crossing.

\begin{theo}\label{thm:firstcolumn}
Assume that $\CC$ is noncrossing and that $N\geqslant 4$. Then, the functor $\F_{\com_{\mid A_{N}(\CC)}}$ is isomorphic to the functor $\F_{N}^{\CC, \Proj_{\CC}}$.
\end{theo}

\begin{proof}
For simplicity, we will write $A$ for $A_{N}(\CC)$. Let us define, for any integer $k$ and any $p\in \Proj_{\CC}(k)$,
\begin{equation*}
\xi_{p} = (P_{p}\otimes \id)u^{\boxtimes k}(e_{1}^{\otimes k}) = \sum_{\mathbf{i}\in \{1, \dots, k\}^{N}} P_{p}(e_{1}^{\otimes k})\otimes u_{\mathbf{1}\mathbf{i}} \quad \text{ and } \quad \xi = u^{\boxtimes k}(e_{1}^{\otimes k}).
\end{equation*}
We claim that these elements are invariant and span the invariant subspaces. To see this, first recall that as was shown in \cite{FW16}, to any projective partition $p\in \Proj_{\CC}(k)$ corresponds a minimal projection $P_{p}\in \mathrm{Mor}(u^{\boxtimes k}, u^{\boxtimes k})$ and therefore also an irreducible subrepresentation $u^{(p)}$. By \cite[Prop 3.7]{F14}, the supremum of these projections is the identity, hence $A$ is spanned by the elements
\begin{align*}
a^{(p)}_{\mathbf{i}} & = (\id\otimes e_{\mathbf{i}}^{*}\circ P_{p})u^{\boxtimes k}(e_{1}^{\otimes k}) \\
& = (\id\otimes e_{\mathbf{i}}^{*})u^{\boxtimes k}(P_{p}(e_{1}^{\otimes k})) = (\id\otimes e_{\mathbf{i}}^{*})(\xi_{p}).
\end{align*}
Let now $\Phi_{k} : H^{\otimes k}\to A$ be the unique linear map sending $e_{\mathbf{i}}$ to $a_{\mathbf{i}}$. This is an equivariant surjection and moreover, setting $H_{p} = P_{p}(H^{\otimes k})$, the restriction $\Phi_{p} = \Phi_{\mid H_{p}}$ maps $P_{p}(e_{\mathbf{i}})$ to $a^{p}_{\mathbf{i}}$. As a consequence,
\begin{equation*}
A_{p} = \spa\{a^{(p)}_{\mathbf{i}} \mid \mathbf{i}\in \{1, \dots, N\}^{k}\}
\end{equation*}
is a spectral subspace corresponding to the representation $u^{(p)}$ and of dimension at most $\dim(u^{(p)})$. In other words, $A_{p}$ is either $0$ or isomorphic to $H_{p}$. Summing up, the spectral decomposition of $A$ is
\begin{equation*}
A = \bigoplus_{k\in \N}\;\bigoplus_{p\in \Proj_{\CC}(k)}A_{p},
\end{equation*}
where each summand is either $0$ or has multiplicity one. For each $p$, the corresponding invariant vector is
\begin{equation*}
\sum_{\mathbf{i}}e_{\mathbf{i}}\otimes a^{(p)}_{\mathbf{i}} = \xi_{p}
\end{equation*}
so that in the end
\begin{equation*}
(H^{\otimes k}\otimes A)^{\G_{N}(\CC)} = \spa\{\xi_{p} \mid p\in \Proj_{\CC}(k)\}
\end{equation*}
and our claim is proved.

Now, trying to send $\xi_{p}$ to $\eta_{p}$ would be a mistake, because many of the vectors $\xi_{p}$ are pairwise orthogonal while the same is not true for the vectors $\eta_{p}$. Instead, we should reverse the construction as follows: for $p\in \Proj_{\CC}$, set
\begin{equation*}
\widetilde{\xi}_{p} = \sum_{q\preceq p}\xi_{q} = \sum_{q\preceq p} (P_{q}\otimes \id)(\xi)= Q_{p}(\xi).
\end{equation*}
It is straightforward to check by induction on the number of through-blocks that
\begin{equation*}
\spa\{\widetilde{\xi}_{p} \mid p\in \Proj_{\CC}(k)\} = \spa\{\xi_{p} \mid p\in \Proj_{\CC}(k)\}
\end{equation*}
so that we still have a generating family of the invariant subspace. Moreover, because $Q_{p}$ is an intertwiner, we have $\widetilde{\xi}_{p} = u^{\boxtimes k}Q_{p}(e_{1}^{\otimes k})$ so that by unitarity of $u$,
\begin{align*}
\langle \widetilde{\xi}_{p}, \widetilde{\xi}_{q}\rangle & = \left\langle u^{\boxtimes k}Q_{p}(e_{1}^{\otimes k}), u^{\boxtimes k}Q_{q}(e_{1}^{\otimes k})\right\rangle \\
& = \left\langle Q_{q}^{*}Q_{p}(e_{1}^{\otimes k}), e_{1}^{\otimes k}\right\rangle \\
& = N^{\mathrm{rl}(q^{*}, p) - \beta(p)/2 - \beta(q)/2}\left\langle T_{q^{*}p}(e_{1}^{\otimes k}), e_{1}^{\otimes k}\right\rangle \\
& = N^{\mathrm{rl}(q^{*}, p) - \beta(p)/2 - \beta(q)/2} 
\end{align*}
Using this, we see that there is a unique isometric linear map $\Phi_{k}: \widetilde{\xi}_{p}\mapsto N^{-\beta(p)/2}\eta_{p}$ between $(H^{\otimes k}\otimes B)^{\G_{N}(\CC)}$ and $K_{k}$. Since $\Phi_{k}$ is also surjective, it is an isomorphism of Hilbert spaces. If now $r\in \CC(k, l)$, then using Lemma \ref{lem:removedloopsprojective} we see that
\begin{align*}
\langle T_{r}(\widetilde{\xi}_{p}), \widetilde{\xi}_{q}\rangle & = N^{-\beta(p)/2 - \beta(q)/2}\langle T_{q^{*}}\circ T_{r}\circ T_{p}(e_{1}^{\otimes k}), e_{1}^{\otimes k}\rangle \\
& = N^{-\beta(p)/2 - \beta(q)/2}N^{\rl(q^{*}, rp) + \rl(r, p)}\langle e_{1}^{\otimes k}, T_{p^{*}rq}(e_{1}^{\otimes k})\rangle \\
& = N^{-\beta(p)/2 - \beta(q)/2}N^{\rl(r, p) + \rl(q^{*}, rp)} \\
& = N^{-\beta(p)/2 - \beta(q)/2}N^{\rl(r, p) + \rl(q^{*}, rpr^{*})} \\
& = N^{-\beta(p)/2 - \beta(q)/2}N^{\rl(r, p)}\langle \eta_{rpr^{*}}, \eta_{q}\rangle.
\end{align*}
In other words, $\Phi_{l}\circ T_{r}\circ\Phi_{k}^{-1} = \varphi(r)$ and the proof is complete.
\end{proof}

\begin{rem}
A general theory of the ``first column space'' for easy quantum groups was first developed by T. Banica, A. Skalski and P. So\l{}tan in \cite{banica2012noncommutative} and studied using the combinatorics of partitions. However, the connection with our approach is unclear. Let us nevertheless mention that they also consider ``first $d$ columns spaces'' and that it would be interesting to describe them through our construction. For another approach to the same objects, one can also see the work of S. Jung and M. Weber \cite{jung2020partition}.
\end{rem}

Before turning to the second construction, let us study a few examples in more detail.

\begin{exe}\label{ex:orthogonal}
As a first example, let us consider the free orthogonal quantum group $O_{N}^{+} = \G_{N}(NC_{2})$. Then, $A_{N}(\CC)$ is known to be a completion of the algebra of functions on the \emph{free orthogonal quantum sphere} $S_{+}^{N-1}$ (see \cite{banica2010quantum}). It was proven in \cite[Sec 7]{DFW19} that in the spectral decomposition, there is exactly one copy of each equivalence class of irreducible representations. As a consequence, $(H^{\otimes k}\otimes A_{N}(\CC))^{O_{N}^{+}}$ has dimension equal to the number of irreducible subrepresentations of $u^{\boxtimes k}$, which is also the number of projective partitions in $NC_{2}(k, k)$. This proves that the family of vectors $(\eta_{p})_{p\in \Proj_{NC_{2}}(k)}$ is linearly independent in that case.
\end{exe}

\begin{exe}
The situation changes drastically if we consider the free hyperoctahedral quantum group $H_{N}^{+} = \G_{N}(NC_{\text{even}})$, where $NC_{\text{even}}$ is the category of all non-crossing partitions with all blocks of even size. Using the defining relations of $C(H_{N}^{+})$ (see for instance \cite[Def 1.3]{banica2009fusion}), it is easy to see that its first column space is isomorphic to $\C^{2N}$. Therefore, we get an action on a commutative and finite-dimensional C*-algebra. The spectral decomposition is easily seen to be, in the notations of \cite{banica2009fusion}, $u^{0}\oplus u^{1}$. This means that the maps $(\eta_{p})_{p\in \Proj_{NC_{\text{even}}}(k)}$ are not linearly independent for $k\geqslant 2$. Indeed, if $p_{4}\in NC(2, 2)$ denotes the partition with only one block, then the definition implies that $\eta_{\vert\vert} = \eta_{p_{4}}$.
\end{exe}

\begin{exe}\label{ex:permutation}
If we consider now the quantum permutation group $S_{N}^{+} = \G_{N}(NC)$ of S. Z. Wang \cite{wang1998quantum}, the coefficients of $u$ are projections summing up to one on each row and each column. As a consequence, $A_{N}(\CC) = \C^{N}$ with the standard permutation action of $S_{N}^{+}$. The corresponding spectral decomposition is then simply $u$ and once again, the vectors $\eta_{p}$ are not linearly independent.
\end{exe}

The proof of Theorem \ref{thm:firstcolumn} heavily relies on results of \cite{FW16} which are only valid for non-crossing partitions. There may nevertheless be a more fundamental reason why $B_{N}(\CC, \Proj_{\CC})$ is the first column space which applies to arbitrary categories of partitions $\CC$. At least, it does not seem unreasonable to conjecture that $\F_{N}^{\CC, \Proj_{\CC}}$ always recovers the first column space. As already mentioned, \cite{jung2020partition} may be a starting point to investigate that question.

\subsection{Shifted line partitions}

\subsubsection{The construction}

We will now give another, different construction involving partitions. The idea this time is to use partitions lying on one line. More precisely, for a subset $\LL_{n}\subset P(0, n)$ we can define a Hilbert space
\begin{equation*}
K_{n}^{\LL} = \spa\{T_{p}\mid p\in \LL_{n}\}\subset H^{\otimes n}
\end{equation*}
with the induced inner product which reads $\langle T_{p}, T_{q}\rangle = N^{\rl(p, q^{*})}$. Once again, the inclusion for tensor products is given by horizontal concatenation: $T_{p}\otimes T_{q}\to T_{p\odot q}$. Moreover, there is a natural way of letting $\mor_{\G_{N}(\CC)}(k, l)$ act on these spaces by simply using the composition of partitions:
\begin{equation*}
\varphi(r): T_{p}\mapsto T_{rp}.
\end{equation*}
Once again, this will only make sense if we add an assumption on $\LL$.

\begin{defi}
Let $\CC$ be a category of partitions. A \emph{module of line partitions} over $\CC$ is a set $\LL$ of partitions lying on the lower row such that, writing $\LL_{n} = \LL\cap P(0, n)$,
\begin{enumerate}
\item $\LL_{k}\odot\LL_{l}\subset \LL_{k+l}$;
\item For any $r\in \CC(k, l)$ and $p\in \LL_{k}$, $rp\in \LL_{l}$.
\end{enumerate}
\end{defi}

It is not very difficult to check that this is enough to produce a weak unitary tensor functor.

\begin{propo}
Let $\CC$ be a category of partitions, let $\LL$ be a module of line partitions over $\CC$ and let $N$ be an integer. Then, the spaces $(K_{n}^{\LL})_{n\in \N}$ and the map $\varphi$ satisfy the assumptions of Theorem \ref{thm:tkergodic}. The corresponding weak unitary tensor functor will be denoted by $\F_{N}^{\CC, \LL}$.
\end{propo}

\begin{proof}
This times, the fact that the maps $\varphi(r)$ are well-defined and that there exists a linear map $\mor_{\G_{N}(\CC)}(k, l)\to \B(K_{k}^{\LL}, K_{l}^{\LL})$ sending $T_{r}$ to $\varphi(t)$ follow from the definitions. These maps behave well with respect to the inner product,
\begin{align*}
\langle \varphi(r)(T_{p}), T_{q}\rangle & = N^{\rl(q^{*}, rp) + \rl(r, p)} \\
& = N^{\rl(q^{*}r, p) + \rl(q^{*}, r)} \\
& = N^{\rl((r^{*}q)^{*}, p) + \rl(r^{*}, q)} \\
& = \langle T_{p}, \varphi(r^{*})T_{q}\rangle.
\end{align*}
Compatibility with composition follows from a similar computation,
\begin{align*}
\varphi(r)(\varphi(r')(T_{p})) & = N^{\rl(r, r'p) + \rl(r', p)}T_{r\circ r'\circ p} \\
& = N^{\rl(rr', p) + \rl(r, r')}T_{(r\circ r')\circ p} \\
& = N^{\rl(r, r')}\varphi(r\circ r')T_{p}. \qedhere
\end{align*}
As for Condition \ref{cond:3}, the proof follows the same line as in Theorem \ref{thm:projectivepartitionaction}, with this time the formula
\begin{equation*}
(\id_{H^{\otimes l}}\otimes T_{p})T_{q} = N^{\rl(\vert^{\odot l}\odot p^{*}, q)}T_{(\vert^{\odot l}\odot p^{*})q}.
\end{equation*}
\end{proof}

The most natural example of a module of line partitions over $\CC$ is given by $\LL_{n} = \CC'(0, n)$ for some category of partitions $\CC\subset \CC'$. It turns out that this example can be completely elucidated.

\begin{propo}
Let $\CC\subset \CC'$ be categories of partitions. Then, the action corresponding to $\F_{N}^{\CC, \CC'(0, \bullet)}$ is the translation action on the quotient space $C(\G_{N}(\CC')\backslash\G_{N}(\CC))$.
\end{propo}

\begin{proof}
Observe that the Hilbert spaces $K_{n}^{\LL}$ are exactly the spaces fixed vectors in $H^{\otimes n}$ for the fundamental representation of the quantum subgroup $\G_{N}(\CC')$. The result therefore directly follows from \cite[Thm 11.1]{PR07}.
\end{proof}

\begin{rem}
This result can also be recovered using induction, see Section \ref{sec:induction}.
\end{rem}

These are the only examples that we were able to identify in this setting. However, as we will see hereafter, it is possible to modify the definition of $\F_{N}^{\CC, \LL}$ to encompass examples of different types. But before turning to this, we would like to make a connection with projective partitions.

\begin{propo}
Let $\CC$ be a category of non-crossing partitions. Then, the functor $\F_{N}^{\CC, \CC'(0, \bullet)}$ is isomorphic to the functor $\F_{N}^{\CC, \Proj_{\CC'}^{0}}$.
\end{propo}

\begin{proof}
Consider the linear map $\Phi : K_{N}^{\CC, \CC'(0, \bullet)}\to K_{N}^{\CC, \Proj_{\CC}^{0}}$ sending $T_{p}$ to $\eta_{pp^{*}}$. This is well-defined because we assume $\CC$ to be non-crossing, hence the vectors $T_{p}$ are linearly independent. Moreover, $\Phi$ respects tensor product. Eventually, for $r\in \CC(k, l)$ and $p\in \CC'(0, k)$,
\begin{equation*}
\Phi\circ\varphi(r)(T_{p})) = N^{\rl(r, p)}\Phi(T_{rp}) = N^{\rl(r, p)}\eta_{rpp^{*}r^{*}} = \varphi(r)\circ\Phi(T_{p})
\end{equation*}
and the proof is complete.
\end{proof}

\subsubsection{Shift by one}

We now want to investigate the possibility of extending the previous construction to line partitions which are shifted. By this, we mean that we fix an integer $k_{0}$ and consider sets of line partitions $\LL_{n}\subset P(0, n+k_{0})$. Given a partitions $r\in \CC(k, l)$, we can still define $\varphi(r) : T_{p} \mapsto T_{r\odot\vert^{\odot k_{0}}}\circ T_{p}$ and the compatibility with composition and adjoints still holds. However, there is a problem concerning tensor products. Indeed, if $p\in P(0, k+k_{0})$ and $q\in P(0, l+k_{0})$, how can we build from them a partition in $P(0, k+l+k_{0})$? We must cancel $k_{0}$ points, but there are many ways of doing that, which need not yield a weak unitary tensor functor.

We will give here one construction for $k_{0} = 1$ which works only for specific choices of $\LL$. It relies on the following ``shifted horizontal concatenation operation''. If $p\in P(0, k+1)$ and $q\in P(0, k+1)$, we produce a partition $p\odot_{1}q\in P(0, k+l+1)$ by applying the following procedure
\begin{itemize}
\item Rotate the right-most point of $p$ to its left;
\item Make the horizontal concatenation with $q$;
\item Rotate back the now left-most point to the right;
\item Merge the last two points of the obtained partition into one point.
\end{itemize}
Here is a pictorial example,
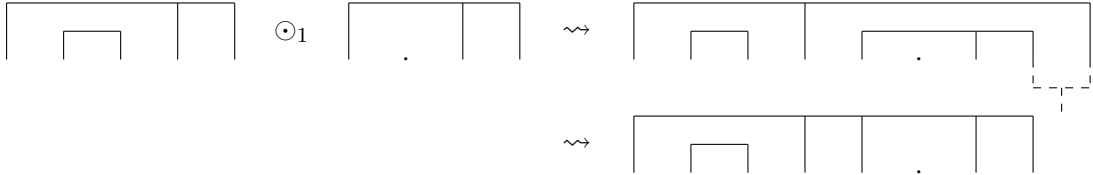
\begin{center}
\begin{tikzpicture}[scale=0.75]
\draw (-5, 0)--(-5,1);
\draw (-4, 0)--(-4,0.5);
\draw (-3, 0)--(-3,0.5);
\draw (-2, 0)--(-2,1);
\draw (-1, 0)--(-1,1);

\draw (-4,0.5)--(-3,0.5);
\draw (-5,1)--(-1, 1);

\draw (0, 0.5)node{$\odot_{1}$};

\draw (4, 0)--(4,1);
\draw (3, 0)--(3,1);
\draw (2, 0)node{$\cdot$};
\draw (1, 0)--(1,1);

\draw (4, 1)--(1, 1);


\draw (5, 0.5)node{$\rightsquigarrow$};

\draw (6, 0)--(6, 1);
\draw (7, 0)--(7, 0.5);
\draw (8, 0)--(8, 0.5);
\draw (9, 0)--(9, 1);
\draw (10, 0)--(10, 0.5);
\draw (11, 0)node{$\cdot$};
\draw (12, 0)--(12, 0.5);
\draw (13, 0)--(13, 0.5);
\draw (14, 0)--(14, 1);

\draw (6, 1)--(14, 1);
\draw (7, 0.5)--(8, 0.5);
\draw (10, 0.5)--(13, 0.5);

\draw [dashed] (13, 0)--(13, -0.5);
\draw [dashed] (14, 0)--(14, -0.5);
\draw [dashed] (13, -0.5)--(14, -0.5);
\draw [dashed] (13.5, -0.5)--(13.5,-1);


\draw (5, -1.5)node{$\rightsquigarrow$};

\draw (6, -2)--(6, -1);
\draw (7, -2)--(7, -1.5);
\draw (8, -2)--(8, -1.5);
\draw (9, -2)--(9, -1);
\draw (10, -2)--(10, -1);
\draw (11, -2)node{$\cdot$};
\draw (12, -2)--(12, -1);
\draw (13, -2)--(13, -1);

\draw (6, -1)--(13, -1);
\draw (7, -1.5)--(8, -1.5);

\end{tikzpicture}
\end{center}
It is easy to see that $T_{p}\otimes T_{q}\mapsto T_{p\odot_{1}q}$ is isometric. However, it is not clear that the Hilbert spaces $(K_{n}^{\LL})_{n\in \N}$ are stable under it.

\begin{propo}\label{prop:functorshiftbyone}
Let $\CC\subset \CC'$ be categories of partitions such that $\CC'$ contains $NC$. Then, the spaces $(K_{n}^{\CC'(0, \bullet + 1)})_{n\in \N}$ and the map $\varphi$, together with the previous weak tensor structure, satisfy the assumptions of Theorem \ref{thm:tkergodic}.
\end{propo}

\begin{proof}
The inclusion $\CC\subset \CC'$ means that the corresponding functor can be written as the composition of an inclusion functor $\rep(\G_{N}(\CC))\subset \rep(\G_{N}(\CC'))$ and a functor on $\rep(\G_{N}(\CC'))$. To prove that the second one is a weak tensor functor, we will describe the corresponding ergodic action.

The assumption $NC\subset \CC'$ means that $\G_{N}(\CC')$ is a quantum subgroup of $S_{N}^{+} = \G_{N}(NC)$. As a consequence, it has a natural action on $H = \C^{N}$ whose spectral decomposition reduces to the fundamental representation $u$. Thus, for any integer $n$,
\begin{equation*}
\left(H^{\otimes n}\otimes \C^{N}\right)^{\G_{N}(\CC')} = \left(H^{\otimes (n+1)}\right)^{\G_{N}(\CC')} = \spa\{T_{p} \mid p\in \CC'(0, n+1)\}.
\end{equation*}
To prove that the tensor product structure is the correct one, let us recall how it is define in general for the fixed points of a general ergodic action $\alpha$ on a C*-algebra $B$. We can write arbitrary elements of $H^{\otimes n}\otimes B$ as
\begin{equation*}
\sum_{i_{1}, \dots, i_{n}}e_{i_{1}}\otimes \cdots \otimes e_{i_{n}}\otimes b_{i_{1}\cdots i_{n}} = \sum_{\mathbf{i}} e_{\mathbf{i}}\otimes b_{\mathbf{i}}.
\end{equation*}
Such an element is invariant if and only if $\alpha(b_{\mathbf{i}}) = \sum_{\mathbf{j}}b_{\mathbf{j}}\otimes u_{\mathbf{j}\mathbf{i}}$. Given two such invariant vectors $x = \sum_{\mathbf{i}}e_{\mathbf{i}}\otimes b_{\mathbf{i}}$ and $y = \sum_{\mathbf{j}}e_{\mathbf{j}}\otimes b_{\mathbf{j}}$, we set
\begin{equation*}
\iota_{k, l}(x\otimes y) = \sum_{\mathbf{i}\mathbf{j}}e_{\mathbf{i}}\otimes e_{\mathbf{j}}\otimes b_{\mathbf{i}}b_{\mathbf{j}}.
\end{equation*}
In our setting, we have
\begin{align*}
T_{p} & = \sum_{\mathbf{i}}\delta_{p}(\mathbf{i})e_{i_{1}}\otimes \cdots \otimes e_{i_{k+1}} \\
& = \sum_{i_{1}, \dots, i_{k}}\left(e_{i_{1}}\otimes\cdots\otimes e_{i_{k}}\right)\otimes\left(\sum_{i_{k+1}}\delta_{p}(\mathbf{i})e_{i_{k+1}}\right)
\end{align*}
so that
\begin{equation*}
b_{i_{1}\cdots i_{k}} = \sum_{i_{k+1}=1}^{N} \delta_{p}(i_{1}, \dots, i_{k}, i_{k+1})e_{i_{k+1}}.
\end{equation*}
The formula for the tensor product then yields
\begin{align*}
\iota_{k, l}(T_{p}\otimes T_{q}) & = \sum_{i_{1}, \dots, i_{k}, j_{1}, \dots, j_{l}}e_{i_{1}}\otimes \cdots\otimes e_{i_{k}}\otimes e_{j_{1}}\otimes \cdots\otimes e_{j_{l}}\otimes\left(\sum_{i_{k+1}, j_{l+1}}\delta_{p}(\mathbf{i})\delta_{q}(\mathbf{j})e_{i_{k+1}}e_{j_{l+1}}\right) \\
& = \sum_{\mathbf{i}\mathbf{j}}\delta_{p\odot q}(\mathbf{i}\cdot\mathbf{j})\delta_{e_{i_{k+1}}, e_{j_{l+1}}}e_{i_{1}}\otimes \cdots\otimes e_{i_{k}}\otimes e_{j_{1}}\otimes \cdots\otimes e_{j_{l+1}} \\
& = T_{p\odot_{1}q}. \qedhere
\end{align*}
We have proven that the functor comes from an ergodic action, hence Condition \ref{cond:3} is also satisfied and the proof is complete.
\end{proof}

In the special case $\CC = \CC'$, the corresponding action is simply the permutation action of $\G_{N}(\CC')$ on $\C^{N}$. In general, the action is induced from that one, see Section \ref{sec:induction}.

\subsubsection{Shift by two}

Our last example will be a variation of the previous one where we shift the partition by $k_{0} = 2$ instead of $1$. This times, we can give a construction which works in full generality. Let us first introduce the tensor product structure. If $p\in P(0, k+2)$ and $s\in P(0, q+2)$, we produce a partition $p\odot_{2}q\in P(0, k+l+2)$ by applying the following procedure
\begin{itemize}
\item Rotate the two right-most points of $p$ to its left;
\item Make the horizontal concatenation with $q$;
\item Rotate back the now two left-most points to the right;
\item Cap the the second and third points from the right with $\sqcap$.
\end{itemize}
Here is a pictorial example,
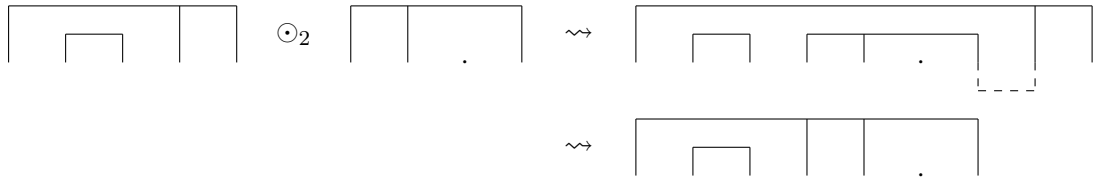
\begin{center}
\begin{tikzpicture}[scale=0.75]
\draw (-5, 0)--(-5,1);
\draw (-4, 0)--(-4,0.5);
\draw (-3, 0)--(-3,0.5);
\draw (-2, 0)--(-2,1);
\draw (-1, 0)--(-1,1);

\draw (-4,0.5)--(-3,0.5);
\draw (-5,1)--(-1, 1);

\draw (0, 0.5)node{$\odot_{2}$};

\draw (4, 0)--(4,1);
\draw (3, 0)node{$\cdot$};
\draw (2, 0)--(2,1);
\draw (1, 0)--(1,1);

\draw (4, 1)--(1, 1);


\draw (5, 0.5)node{$\rightsquigarrow$};

\draw (6, 0)--(6, 1);
\draw (7, 0)--(7, 0.5);
\draw (8, 0)--(8, 0.5);
\draw (9, 0)--(9, 0.5);
\draw (10, 0)--(10, 0.5);
\draw (11, 0)node{$\cdot$};
\draw (12, 0)--(12, 0.5);
\draw (13, 0)--(13, 1);
\draw (14, 0)--(14, 1);

\draw (6, 1)--(14, 1);
\draw (7, 0.5)--(8, 0.5);
\draw (9, 0.5)--(12, 0.5);

\draw [dashed] (12, 0)--(12, -0.5);
\draw [dashed] (13, 0)--(13, -0.5);
\draw [dashed] (12, -0.5)--(13, -0.5);


\draw (5, -1.5)node{$\rightsquigarrow$};

\draw (6, -2)--(6, -1);
\draw (7, -2)--(7, -1.5);
\draw (8, -2)--(8, -1.5);
\draw (9, -2)--(9, -1);
\draw (10, -2)--(10, -1);
\draw (11, -2)node{$\cdot$};
\draw (12, -2)--(12, -1);

\draw (6, -1)--(12, -1);
\draw (7, -1.5)--(8, -1.5);

\end{tikzpicture}
\end{center}

\begin{propo}
Let $\CC\subset \CC'$ be categories of partitions. Then, the spaces $(K_{n}^{\CC'(0, \bullet + 2)})_{n\in \N}$ and the map $\varphi$, together with the previous weak tensor structure, satisfy the assumptions of Theorem \ref{thm:tkergodic}.
\end{propo}

\begin{proof}
The proof is very similar to that of Proposition \ref{prop:functorshiftbyone}. Once again we split the functor as the composition of an inclusion and one coming from an ergodic action of $\G_{N}(\CC')$ and we simply have to describe the latter and check that it yields the correct functor.

The compact quantum group $\G_{N}(\CC')$ has an action on $M_{N}(\C)$ given by the formula
\begin{equation*}
M\in M_{N}(\C)\mapsto u^{*}(M\otimes \id)u.
\end{equation*}
Its spectral decomposition is simply $u\boxtimes u$, so that
\begin{equation*}
(H^{\otimes n}\otimes M_{N}(\C))^{\G_{N}(\CC')} \cong \left(H^{\otimes (n+2)}\right)^{\G_{N}(\CC')} \cong \spa\{T_{p}\mid p\in \CC'(0, n+2)\}.
\end{equation*}
Once again, all that we have to do is to check that the tensor product structure is the correct one. To do this, first observe that we are using an identification $H^{\otimes n+2}\simeq H^{\otimes n}\otimes M_{N}(\C)$ which is given on the basis vectors by
\begin{equation*}
e_{i_{1}}\otimes \cdots \otimes e_{i_{k}}\otimes e_{i_{k+1}}\otimes e_{i_{k+2}} \mapsto e_{i_{1}}\otimes \cdots \otimes e_{i_{k}}\otimes \left(e_{i_{k+1}}e_{i_{k+2}}^{*}\right).
\end{equation*}
Thus, using the usual notations, the invariant vectors are of the form
\begin{align*}
T_{p} & = \sum_{i_{1}, \dots, i_{k+2}}\delta_{p}(\mathbf{i})e_{i_{1}}\otimes \cdots \otimes e_{i_{k}}\otimes \left(e_{i_{k+1}}e_{i_{k+2}}^{*}\right) \\
& = \sum_{i_{1}, \dots, i_{k}}e_{i_{1}}\otimes \cdots \otimes e_{i_{k}}\otimes \left(\sum_{i_{k+1}, i_{k+2}}\delta_{p}(\mathbf{i})e_{i_{k+1}}e_{i_{k+2}}^{*}\right)
\end{align*}
so that
\begin{equation*}
b_{i_{1}\cdots i_{k}} = \sum_{i_{k+1}, i_{k+2}}\delta_{p}(\mathbf{i})e_{i_{k+1}}e_{i_{k+2}}^{*}.
\end{equation*}
The tensor product being given by the product of these elements, we get
\small
\begin{align*}
\iota_{k, l}(T_{p}\otimes T_{q}) & = \sum_{\mathbf{i}\mathbf{j}}e_{i_{1}}\otimes \cdots\otimes e_{i_{k}}\otimes e_{j_{1}}\otimes \cdots\otimes e_{j_{l}} \otimes \left(\sum_{i_{k+1}, i_{k+2}, j_{l+1}, j_{l+2}}\delta_{p}(\mathbf{i})\delta_{q}(\mathbf{j})e_{i_{k+1}}e_{i_{k+2}}^{*}e_{j_{l+1}}e_{j_{l+2}}^{*}\right) \\
& = \sum_{\mathbf{i}\mathbf{j}}e_{i_{1}}\otimes \cdots\otimes e_{i_{k}}\otimes e_{j_{1}}\otimes \cdots\otimes e_{j_{l}} \otimes\left(\sum_{i_{k+1}, i_{k+2}, j_{l+1}, j_{l+2}}\delta_{i_{k+2}, j_{l+1}}\delta_{p}(\mathbf{i})\delta_{q}(\mathbf{j})e_{i_{k+1}}e_{j_{l+2}}^{*}\right) \\
& = \sum_{\mathbf{i}\mathbf{j}}\delta_{p}(\mathbf{i})\delta_{q}(\mathbf{j})\delta_{i_{k+2}, j_{l+1}}e_{i_{1}}\otimes \cdots\otimes e_{i_{k}}\otimes e_{j_{1}}\otimes \cdots\otimes e_{j_{l}}\otimes e_{i_{k+1}}e_{j_{l+2}}^{*} \\
& = T_{p\odot_{2}q}. \qedhere
\end{align*}
\normalsize
\end{proof}

In the special case $\CC = \CC'$, the corresponding action is simply the action $\G_{N}(\CC')$ on $M_{N}(\C)$. In general, the action is induced from that one, see Section \ref{sec:induction}.

\subsection{Yetter-Drinfeld structure}

The combinatorial approach developed here can also be used to elucidate the existence of Yetter-Drinfeld structures for the ergodic actions that we constructed. For instance, it is well-known that the braided commutative Yetter-Drinfeld structure on $C(\G)$ given in Example \ref{ex:c(G)} restricts to a braided commutative Yetter-Drinfeld structure on any quotient $C(\HH\backslash \G)$. In the case of easy quantum groups, this can be easily described combinatorially. Indeed, let us denote by $\alpha_{k} : \CC'(0, k)\to \CC'(0, k+2)$ the map sending a partition $p$ to the same partition nested in a pairing. We now set
\begin{equation*}
\mathrm{a}_{k} : \eta\mapsto (\id\otimes \eta\otimes \id)\circ T_{\sqcap},
\end{equation*}
which is a linear map from $H^{\otimes k}$ to $H^{\otimes (k+2)}$. Observing that $\mathrm{a}_{k}(\eta_{p}) = \eta_{\alpha_{k}(p)}$, it is straightforward to check that these satisfy all the assumptions of Theorem \ref{theo:yd_reconstruction_na} and Corollary \ref{cor:bc}, hence the result.

But the constructions of the previous subsections do not always admit a Yetter-Drinfeld structure and we will now prove some results in the direction of non-existence. We start with the free real sphere, which comes from the module of projective partitions $\mathcal{P} = \Proj_{NC_{2}}$ over $NC_{2}$. Before stating and proving our result, let us introduce some notations. We set $d_{1} = \sqcap\circ\sqcap^{*}$ and
\[
d_{n+1} = (\vert^{\odot n}\odot d_{1}\odot\vert^{\odot n})\circ d_{n}\circ (\vert^{\odot n}\odot d_{1}\odot\vert^{\odot n}).
\]
In other words, $d_{n}$ consists in $n$ nested pairings in each row.

\begin{propo}
The action of $O_{N}^{+}$ on the free real sphere does not admit a Yetter-Drinfeld structure.
\end{propo}

\begin{proof}
Assume by contradiction that a Yetter-Drinfeld structure exists and consider the associated maps $\{\mathrm{a}_{k}\}_{k\in \N}$. We can write
\begin{equation*}
\mathrm{a}_{2}(\eta_{\vert\vert}) = \alpha\eta_{d_{1}\odot\vert} + \alpha'\eta_{\vert\odot d_{1}} + \alpha''\eta_{\vert\odot d_{1}\odot\vert} + \beta\eta_{\vert\vert\vert\vert} + \gamma\eta_{d_{2}}
\end{equation*}
and this yields
\begin{align*}
\varphi(\vert\odot d_{1}\odot\vert)\circ \mathrm{a}_{2}(\eta_{\vert\vert})
& = (\alpha + \alpha' + N\alpha'' + \beta)\eta_{\vert\odot d_{1}\odot\vert} + N\gamma\eta_{d_{2}}.
\end{align*}
Recall that it was observed in Example \ref{ex:orthogonal} that the vectors associated with non-crossing pair partitions are linearly independent. Therefore, from the equality
\begin{equation*}
\mathrm{a}_{2}\circ\varphi(d_{1})(\eta_{\vert\vert}) = \mathrm{a}_{2}(\eta_{d_{1}}) = \mathrm{a}_{2}\circ\varphi(R_{1}) = \varphi(R_{2}) = \eta_{d_{2}},
\end{equation*}
we deduce that $N\gamma = 1$, so that in particular $\gamma\neq 0$. Let us now consider
\begin{equation*}
\mathrm{a}_{1}(\eta_{\vert}) = \lambda\eta_{\vert\odot d_{1}} + \mu\eta_{d_{1}\odot\vert} + \nu\eta_{\vert\vert}.
\end{equation*}
A straightforward calculation yields
\begin{equation*}
\varphi(\id\otimes \bar{R}_{1}^{*}\otimes\id)\circ (\mathrm{a}_{1}\otimes \mathrm{a}_{1})(\eta_{\vert}\otimes \eta_{\vert}) = \lambda\eta_{\vert\vert\odot d_{1}} + \mu\eta_{d_{1}\odot\vert\vert} + \nu\eta_{\vert\vert\vert\vert}.
\end{equation*}
This should be equal to $\mathrm{a}_{1}\circ\iota_{1, 1}(\eta_{\vert}\odot\eta_{\vert}) = \mathrm{a}_{2}(\eta_{\vert\vert})$, contradicting linear independence.
\end{proof}

\begin{rem}
This result is, to our knowledge, new. Moreover, as a corollary it shows that the free real sphere is not the quotient $O_{N}^{+}/ O_{N-1}^{+}$, nor in fact any quotient of $O_{N}^{+}$. This was already proven (in a much more elementary way) in \cite[Prop 7.1]{DFW19}.
\end{rem}

The same idea works for the action of $S_{N}^{+}$ on $\C^{N}$ but the proof requires extra notations: $s_{2}\in P(1, 1)$ denoting the double singleton partition and $p_{3}\in P(2, 1)$ denoting the one-block partition on three points.

\begin{propo}
The action of $S_{N}^{+}$ on $\C^{N}$ does not admit a Yetter-Drinfeld structure.
\end{propo}

\begin{proof}
As in the orthogonal case, we will gather several computations involving the conditions of Theorem \ref{thm:tkergodic_yetter} to derive a contradiction. We start with the observation that $\mathrm{a}_{1}\circ\varphi(s_{2})(\eta_{\vert}) = \mathrm{a}_{1}(\eta_{s_{2}})$. Meanwhile, we can write
\begin{equation*}
\mathrm{a}_{1}(\eta_{\vert}) = \alpha\eta_{\vert\odot d_{1}} + \alpha'\eta_{\vert\odot s_{2}\odot s_{2}} + \beta\eta_{d_{1}\odot \vert} + \beta'\eta_{s_{2}\odot s_{2}\odot \vert} + \gamma\eta_{s_{2}\odot \vert\odot s_{2}} + \gamma'\eta_{s_{2}\odot s_{2}\odot s_{2}} + \delta\eta_{\vert\vert\vert}
\end{equation*}
so that
\begin{equation*}
\varphi(\vert\odot s_{2}\odot\vert)\circ \mathrm{a}_{1}(\eta_{\vert}) = (\alpha + \alpha')\eta_{\vert\odot s_{2}\odot s_{2}} + (\beta + \beta')\eta_{s_{2}\odot s_{2}\odot \vert} + (\gamma + \gamma')\eta_{s_{2}\odot s_{2}\odot s_{2}} + \delta\eta_{\vert\odot s_{2}\odot \vert}.
\end{equation*}
We will combine in the end this with another similar computation. Namely, we have $\mathrm{a}_{1}\circ\varphi(p_{3})(\eta_{d_{1}}) = \mathrm{a}_{1}(\eta_{s_{2}})$ while
\begin{equation*}
\varphi(\vert\odot p_{3}\odot \vert)\circ \mathrm{a}_{2}(\eta_{d_{1}}) = \varphi(\vert\odot p_{3}\odot \vert)(\eta_{d_{2}}) = \eta_{q}
\end{equation*}
where $q\in NC(3, 3)$ is the same as $d_{1}$ except that a singleton is nested inside the pairing on each row. Summing up, there exists $A, B, C, D\in \C$ such that
\begin{equation*}
\eta_{q} = A\eta_{\vert\odot s_{2}\odot s_{2}} + B\eta_{s_{2}\odot s_{2}\odot \vert} + C\eta_{s_{2}\odot s_{2}\odot s_{2}} + D\eta_{\vert\odot s_{2}\odot \vert}.
\end{equation*}
We cannot conclude directly because, as observed in Example \ref{ex:permutation}, the vectors above need not be linearly independent. But taking the inner product of the right-hand side with $e_{i_{1}}\otimes e_{i_{2}}\otimes e_{i_{3}}$ yields $C$ as soon as $i_{1}, i_{3}\neq 1$, while
\begin{equation*}
\langle \eta_{q}, e_{i_{1}}\otimes e_{i_{2}}\otimes e_{i_{3}}\rangle = \left\{\begin{array}{ccc}
1 & \text{ if } & i_{1} = i_{3} \\
0 & \text{ if } & i_{1} \neq i_{3}
\end{array}\right..
\end{equation*}
This is a contradiction concluding the proof.
\end{proof}

\section{Induction}\label{sec:induction}

Many constructions of Section \ref{sec:examples} can be interpreted in terms of induction of the corresponding actions. That induction procedure is well-known in the classical case and has been studied in detail in the (locally) compact quantum group setting, for instance in \cite{vaes2005new} or more recently in \cite{kitamura2022induced}. However, in these works the setting of compact quantum groups is recovered as a special case while it is possible to give much simpler proofs. Therefore, even though some of the results are well-known to experts, we decided to include them for completeness. As before, we will always consider the \emph{universal} versions of the quantum groups and the corresponding actions.

\subsection{Induced action and Frobenius reciprocity}

Let $\G$ be a compact quantum group and let $\HH$ be a compact quantum subgroup of $\G$, meaning that there is a surjective $*$-homomorphism $\pi$ from $C(\G)$ to $C(\HH)$ such that $\com\circ\pi = (\pi\otimes \pi)\circ\com$, and let $\alpha : B\to B\otimes C(\HH)$ be an action. Consider the following vector subspace
\begin{equation*}
\Ind_{\HH}^{\G}(B) = \left\{b\in B\otimes C(\G) \mid (\alpha\otimes \id)(b) = (\id\otimes \pi\otimes \id)\circ(\id\otimes \Delta_{\G})(b)\right\}.
\end{equation*}
It is clearly a C*-subalgebra of $B\otimes C(\G)$ and we claim that it is acted upon by $\G$ in a natural way.

\begin{lem}
The $*$-homomorphism $(\id\otimes \Delta_{\G})$ defines an action of $\G$ on $\Ind_{\HH}^{\G}(B)$.
\end{lem}

\begin{proof}
Using the definition of an action, we have
\begin{align*}
(\alpha\otimes \id\otimes \id)\circ(\id\otimes \Delta_{\G}) & = (\id\otimes \id\otimes \Delta_{\G})\circ(\alpha\otimes \id) \\
& = (\id\otimes\id\otimes \Delta_{\G})\circ(\id\otimes \pi\otimes \id)\circ(\id\otimes \Delta_{\G}) \\
& = (\id\otimes \pi\otimes \id\otimes \id)\circ(\id\otimes\id\otimes \Delta_{\G})\circ(\id\otimes \Delta_{\G}) \\
& = (\id\otimes \pi\otimes \id\otimes \id)\circ(\id\otimes \Delta_{\G}\otimes \id)\circ(\id\otimes \Delta_{\G}) \\
& = \left([(\id\otimes \pi\otimes \id)\circ(\id\otimes \Delta_{\G})]\otimes \id\right)\circ(\id\otimes \Delta_{\G})
\end{align*}
so that
\begin{equation*}
(\id\otimes \Delta_{\G})(\Ind_{\HH}^{\G}(B))\subseteq \Ind_{\HH}^{\G}(B)\otimes C(\G)
\end{equation*}
and the result follows.
\end{proof}

\begin{defi}
The action $(\id\otimes \Delta_{\G})$ on $\Ind_{\HH}^{\G}(B)$ is called the \emph{induced action} of $(B, \alpha)$ from $\HH$ to $\G$ and is denoted by $\Ind_{\HH}^{\G}(\alpha)$.
\end{defi}

\begin{rem}
The previous construction is called the \emph{cotensor product} in the algebra literature and is known to have good adjunction properties with respect to restriction. We will prove such an adjunction property in our setting below.  
\end{rem}

Here are a few elementary facts.

\begin{propo}\label{prop:induction_cases}
The following hold
\begin{enumerate}[label=$(\arabic*)$]
\item If $G$ and $H$ are classical compact groups and if $X$ is a classical compact $H$-space, then
\begin{equation*}
\Ind_{\HH}^{\G}(C(X)) = C(X\times_{H}G)
\end{equation*}
with the usual induced action.
\item The induced action is functorial in $B$, i.e. given two actions of $\HH$, $(A, \alpha)$ and $(B, \beta)$, and an $\HH$-equivariant map $T : A\to B$, then there is a unique $\G$-equivariant map denoted by $\Ind_{\HH}^{\G}(T): \Ind_{\HH}^{\G}(A) \to \Ind_{\HH}^{\G}(B)$.
\item If $\HH = \G$ and $\pi$ is the identity, then the induced action is equivariantly isomorphic to the original action.
\item If $B = \C$ and $\alpha$ is the trivial action, then $\Ind_{\HH}^{\G}(B) = C(\HH\backslash\G)$ with the translation action.
\end{enumerate}
\end{propo}

\begin{proof}
\begin{enumerate}
\item Classically, one has to consider the quotient of the direct product $X\times G$ by the equivalence relation $(x\cdot h, g) = (x, h\cdot g)$ for all $s\in H$. In functional terms, a function on $X\times G$ defines a function on the quotient if and only if
\begin{equation*}
f\circ(\alpha\times \id) = f\circ(\id\times \lambda)
\end{equation*}
as functions on $X\times H\times G$, where $\lambda$ denotes the left action of $H$ on $G$. These functions are exactly the elements of $\Ind_{H}^{G}(C(X))$, hence the result.
\item Let $(A, \alpha)$ and $(B, \beta)$ be two actions of $\HH$ and let $T : A\to B$ be $\HH$-equivariant. Setting $\widehat{T} = (T\otimes \id)$ yields a map
\begin{equation*}
\widehat{T} : \Ind_{\HH}^{\G}(A)\to B\otimes C(\G)
\end{equation*}
and we will prove that the range of $\widehat{T}$ is contained in $\Ind_{\HH}^{\G}(B)$. Indeed,
\begin{align*}
(\beta\otimes \id)\circ\widehat{T} & = (T\otimes \id\otimes \id)\circ(\alpha\otimes \id) \\
& = (T\otimes \id\otimes \id)\circ(\id\otimes \pi\otimes \id)\circ(\id\otimes \Delta_{\G}) \\
& = (\id\otimes \pi\otimes \id)\circ(\id\otimes \Delta_{\G})\circ\widehat{T}.
\end{align*}
Moreover, $\widehat{T}\circ(\id\otimes \Delta_{\G}) = (\id\otimes \Delta_{\G})\circ\widehat{T}$ by construction, hence this is a $\G$-equivariant map.
\item Let us set $B' := \alpha(B) \subset B\otimes C\G)$, which is a C*-algebra isomorphic to $B$. We claim that $B' = \Ind_{\HH}^{\G}(B)$. Indeed, if $b\in \Ind_{\HH}^{\G}(b)$, then applying $(\id\otimes \id\otimes \varepsilon_{\G})$ to the equality
\begin{align*}
(\id\otimes \Delta_{\G})(b) = (\alpha\otimes \id)(b)
\end{align*}
yields
\begin{equation*}
b = (\alpha\otimes \varepsilon_{\G})(b) = \alpha((\id\otimes\varepsilon_{\G})(b))\in B'.
\end{equation*}
Since moreover the $*$-isomorphism $\alpha : B\to B'$ is equivariant, the result is proven.
\item Simply observe that
\begin{equation*}
\Ind_{\HH}^{\G}(\C) = \{b\in C(\G) \mid (\pi\otimes \id)\com(b) = 1\otimes b\}
\end{equation*}
and that the right-hand side is precisely the definition of the quotient space in the statement. Moreover, the translation action on a quotient space is given by the restriction of the coproduct, hence the result.
\end{enumerate}
\end{proof}

We will now establish a Frobenius type reciprocity for actions which will enable us to understand some of the examples of Section \ref{sec:examples}. As before, let  $\HH$ be a compact quantum subgroup of $\G$ and let $u\in \B(H_{u})\otimes C(\G)$ be a finite-dimensional representation of $\G$. Then $(\id\otimes \pi)(u)$ is a representation of $\HH$, and this gives a restriction functor $\Res^{\G}_{\HH}: \rep(\G) \to \rep(\HH)$. Note that in that setting, we have the equality of coactions
\begin{equation*}
\delta_{\Res_{\HH}^{\G}(u)} = (\id\otimes \pi)\circ\delta_{u}.
\end{equation*}

\begin{propo}[Frobenius type reciprocity]\label{prop:frob_type}
For any $u \in \rep(\G)$ and any action $(B, \alpha)$ of $\HH$ on a C*-albegra $B$,
\begin{equation*}
\mor_{\HH}(\Res^{\G}_{\HH}(u), \alpha)\cong \mor_{\G}(u, \Ind_{\HH}^{\G}(\alpha)).
\end{equation*}
\end{propo}

\begin{proof}
Given $T\in \mor_{\HH}(\Res^{\G}_{\HH}(u), \alpha)$, consider the linear map
\begin{equation*}
\psi_{u}(T) := (T\otimes\id)\circ\delta_{u} : H_{u}\to B\otimes C(\G).
\end{equation*}
We claim that $\psi_{u}(T)\in\mor_{\G}(u, \Ind_{\HH}^{\G}(\alpha))$. Indeed, the equality
\begin{align*}
(\alpha \otimes \id)\circ\psi_{u}(T) & = (\alpha\circ T \otimes \id)\circ\delta_{u} = (T \otimes \id \otimes \id)\circ(\id \otimes \pi \otimes \id)\circ(\delta_{u} \otimes \id)\circ\delta_{u} \\
& = (T \otimes \pi \otimes \id)\circ(\id \otimes \com_{\G})\circ\delta_{u} \\
& = (\id \otimes \pi \otimes \id)\circ(\id \otimes \com_{\G})\circ(T \otimes \id)\circ\delta_{u} \\
& = (\id \otimes \pi \otimes \id)\circ(\id \otimes \com_{\G})\circ\psi_{u}(T)
\end{align*}
implies that $\psi_{u}(T)(H_{u})\subset\Ind_{\HH}^{\G}(\alpha)$, and the assertion therefore follows from
\begin{align*}
\Ind_{\HH}^{\G}(\alpha)\circ\psi_{u}(T) & = (\id\otimes \Delta_{\G})\circ(T\otimes \id)\circ\delta_{u} \\
& = (T\otimes \Delta_{\G})\circ\delta_{u} \\
& = ((T\otimes \id)\circ\delta_{u} \otimes \id)\circ\delta_{u} \\
& = (\psi_{u}(T)\otimes \id)\circ\delta_{u}.
\end{align*}
Conversely, given $T'\in\mor_{\G}(u, \Ind_{\HH}^{\G}(\alpha))$, consider the linear map
\begin{equation*}
\widehat{\psi}_{u}(T') = (\id\otimes\varepsilon_{\G})\circ T': H_{u} \to B.
\end{equation*}
We claim that $\widehat{\psi}_{u}(T')\in\mor_{\HH}(\Res^{\G}_{\HH}(u), \alpha)$. Indeed,
\begin{align*}
\alpha\circ\widehat{\psi}_{u}(T') & = (\id\otimes \id\otimes \varepsilon_{\G})\circ(\alpha\otimes \id)\circ T' \\
& = (\id\otimes \id\otimes\varepsilon_{\G})\circ(\id\otimes \pi\otimes \id)\circ(\id\otimes \com_{\G})\circ T' \\
& = (\id\otimes \pi\otimes \varepsilon_{\G})\circ(\id\otimes \com_{\G})\circ T' \\
& = (\id\otimes \varepsilon_{\G}\otimes \pi)\circ\Ind_{\HH}^{\G}(\alpha)\circ T' \\
& = (\id\otimes \varepsilon_{\G}\otimes \pi)\circ(T'\otimes \id)\circ\delta_{u} \\
& = ((\id\otimes \varepsilon_{\G})\circ T'\otimes \id)\circ(\id \otimes \pi)\circ\delta_{u} \\
& = (\widehat{\psi}_{u}(T')\otimes \id)\circ\delta_{\Res^{\G}_{\HH}(u)}.
\end{align*}
Finally, the result follows because for any $T\in \mor_{\HH}(\Res^{\G}_{\HH}(u), \alpha)$ and $T'\in \mor_{\G}(u, \Ind_{\HH}^{\G}(\alpha))$, we have on the one hand
\begin{align*}
\widehat{\psi}_{u}(\psi_{u}(T)) & = (\id \otimes \varepsilon_{\G})\circ(T\otimes \id)\circ\delta_{u} = T\circ(\id\otimes \varepsilon_{\G})\circ\delta_{u} = T
\end{align*}
and on the other hand
\begin{align*}
\psi_{u}(\widehat{\psi}_{u}(T')) & = ((\id\otimes \varepsilon_{\G})\circ T' \otimes \id)\circ\delta_{u} = (\id\otimes \varepsilon_{\G}\otimes \id)\circ(\id\otimes \Delta_{\G})\circ T' = T'. \qedhere
\end{align*} 
\end{proof}

Our aim now is to upgrade the previous statement, which was in terms of morphism spaces, into a statement at the level of weak unitary tensor functors. Observe that the functor $\F_{\alpha}\circ\Res^{\G}_{\HH}$ is always a weak unitary tensor functor on $\rep(\G)$, so that it make sense to try to compare it to $\F_{\Ind_{\HH}^{\G}(\alpha)}$. As one may expect, these two functors turn out to be the same.

\begin{theo}\label{thm:induction}
Let $\HH$ be a compact quantum subgroup of $\G$ and let $\alpha$ be an action of $\HH$. Then, there is a natural unitary monoidal isomorphism between the functors $\F_{\alpha}\circ\Res^{\G}_{\HH}$ and $\F_{\Ind_{\HH}^{\G}(\alpha)}$.
\end{theo}

\begin{proof}
Using the description of spectral subspaces involving morphism spaces, Proposition \ref{prop:frob_type} tells us that for all $u\in \rep(\G)$,
\begin{equation*}
\xymatrix{
\F_{\alpha} (\Res^{\G}_{\HH}(u)) = \mor_{\HH}({\Res^{\G}_{\HH}(u)}, \alpha)\ar[r]^-{\psi_{u}}_-{\cong} &\mor_{\G}(u, \Ind_{\HH}^{\G}(\alpha)) = \F_{\Ind_{\HH}^{\G}(\alpha)}(u).
}
\end{equation*}
Moreover, we have explicit isomorphisms $\psi_{u} : T\mapsto (T\otimes \id)\circ\delta_{u}$ and it suffices to check that they behave correctly with respect to morphisms. But if $S\in \mor_{\G}(u, v)$ and $T \in \mor_{\HH}({\Res^{\G}_{\HH}(u)}, \alpha)$, then
\begin{align*}
(\F_{\Ind_{\HH}^{\G}(\alpha)}(S)\circ\psi_{u})(T) & = (T\otimes \id)\circ\delta_{u}\circ S^{*} \\
& =  (T\circ S^{*}\otimes \id)\circ\delta_{v} \\
& = \psi_{v}(T\circ S^{*}) \\
& = (\psi_{v}\circ\F_{\alpha}(\Res^{\G}_{\HH}(S)))(T)
\end{align*}
as desired. Thus, $\psi = \{\psi_{u}\}_{u \in \rep(\G)} : \F_{\alpha}\circ\Res^{\G}_{\HH} \to \F_{\Ind_{\HH}^{\G}(\alpha)}$ yields a natural unitary monoidal isomorphism. 
\end{proof}

We can use this to describe several examples from Section \ref{sec:examples}. This is because any inclusion $\CC\subset \CC'$ of categories of partitions induces a surjection $\pi : C(\G_{N}(\CC))\to C(\G_{N}(\CC'))$, simply because there are more relations in the right-hand side than in the left-hand side. Moreover, $\pi$ intertwines the coproducts so that we have a compact quantum subgroup. We can therefore apply Theorem \ref{thm:induction} in that situation. More precisely, we have inclusions
\begin{equation*}
\mor_{\G_{N}(\CC)}(k, l) = \spa\{T_{p} \mid p\in \CC(k, l)\} \subset \spa\{T_{p} \mid p\in \CC'(k, l)\} = \mor_{\G_{N}(\CC')}(k, l)
\end{equation*}
so that if $\PP$ is a module of projective partitions over $\CC'$, then it is also a module of projective partitions over $\CC$ and
\begin{equation*}
\F_{N}^{\CC, \PP} = \F_{N}^{\CC', \PP}\circ\Res_{\G_{N}(\CC)}^{\G_{N}(\CC')}
\end{equation*}
because they coincide by definition on the spaces $\mor_{\G_{N}(\CC)}(k, k)$ for all $k\in \N$. Let us state this formally as a corollary.

\begin{coro}
Let $\CC\subset \CC'$ be categories of partitions and let $\PP$ be a module of projective partitions over $\CC'$. Then, the action corresponding to $\F_{\Proj}^{\CC, \PP}$ is equivariantly isomorphic to the induction of the action coming from $\F_{\Proj}^{\CC', \PP}$.
\end{coro}

The same of course works for actions coming from shifted line partitions. This shows for instance that the action corresponding to $\F_{L}^{NC, P(0, \bullet+1)}$ is the induction to $S_{N}^{+}$ of the permutation action of $S_{N}$ on $N$ points.

\subsection{Induced Yetter-Drinfeld C*-algebras}

It is natural in view of the preceding subsection to wonder whether the induction procedure preserves (braided commutative) Yetter-Drinfeld structures. We show here that the answer is yes and that some important examples can be obtain in that way. This can be done in two ways, one using an explicit computation of the (braided commutative) Yetter-Drinfeld structure and the other using the categorical picture of a (braided) compatible collection for weak unitary tensor functors.

The preservation of the Yetter-Drinfeld structure in an induction procedure was first given by R. Nest and C. Voigt in \cite{NV09} in the general framework of locally compact quantum groups. Our next proposition shows that the braided commutative structure is also conserved in the compact quantum group case.

\begin{propo}\label{prop:induced_yd}
Let $\HH$ be a quantum subgroup of a compact quantum group $\G$, with surjective Hopf $*$-algebra morphism $\pi: \pol(\G) \to \pol(\HH)$, and $(B,\lhd,\alpha)$ be a Yetter-Drinfeld $\HH$-C*-algebra. There exist a Hopf $*$-algebra action $\lhd^{\G,\HH}: \Ind^{\G}_{\HH}(\mathcal{B}) \otimes \pol(\G) \to \Ind^{\G}_{\HH}(\mathcal{B})$, defined by
\[
(b \otimes a') \lhd^{\G,\HH} a = (b \lhd \pi(a_{(2)})) \otimes S_{\G}(a_{(1)})a'a_{(3)}
\]
for all $b \otimes a' \in \Ind^{\G}_{\HH}(\mathcal{B})$ and $a \in \pol(\G)$, such that the tuple $\Ind^{\G}_{\HH}(B,\lhd,\alpha) := (\Ind^{\G}_{\HH}(B),\lhd^{\G,\HH},\Ind^{\G}_{\HH}(\alpha))$ yields a Yetter-Drinfeld $\G$-C*-algebra. Moreover, if the Yetter-Drinfeld $\HH$-C*-algebra $(B,\lhd,\alpha)$ is braided commutative, then the Yetter-Drinfeld $\G$-C*-algebra $\Ind^{\G}_{\HH}(B,\lhd,\alpha)$ is braided commutative.
\end{propo}
\begin{proof}
First, note that the linear map
\[
\blhd: (\mathcal{B} \otimes \pol(\G)) \otimes \pol(\G) \to \mathcal{B} \otimes \pol(\G), (b \otimes a') \blhd a = (b \lhd \pi(a_{(2)})) \otimes S_{\G}(a_{(1)})a'a_{(3)}
\]
yields a Hopf $*$-algebra action of $(\pol(\G),\com)$ on $\mathcal{B} \otimes \pol(\G)$. We will prove that $\Ind^{\G}_{\HH}(\mathcal{B})$ is an invariant $*$-subalgebra of $\mathcal{B} \otimes \pol(\G)$ by the action $\blhd$. Take $b = \sum_{i} b_{i} \otimes a'_{i} \in \Ind^{\G}_{\HH}(\mathcal{B})$ and $a \in \pol(\G)$. By Yetter-Drinfeld condition, we have
\begin{equation}\label{eq:equality2}
\alpha(b_{i} \lhd \pi(a_{(2)})) = ({b_{i}}_{[0]} \lhd \pi(a_{(3)})) \otimes S_{\HH}(\pi(a_{(2)})){b_{i}}_{[1]}\pi(a_{(4)}),
\end{equation}
and by definition of $\Ind^{\G}_{\HH}(\mathcal{B})$, it follows that
\begin{equation}\label{eq:equality1}
\sum_{i} {b_{i}}_{[0]} \otimes {b_{i}}_{[1]} \otimes a'_{i} = \sum_{i} \alpha(b_{i}) \otimes a'_{i} = (\alpha \otimes \id)(b) = (\id \otimes \pi \otimes \id)(\id \otimes \com_{\G})(b) = \sum_{i}b_{i} \otimes \pi({a'_{i}}_{(1)}) \otimes {a'_{i}}_{(2)}.
\end{equation}
Hence, the equality
\begin{align*}
(\alpha \otimes \id)(b \blhd a) & \,\;=\;\, \sum_{i} \alpha(b_{i} \lhd \pi(a_{(2)}) \otimes S_{\G}(a_{(1)})a'_{i}a_{(3)} \\
& \overset{(\ref{eq:equality2})}{=} \sum_{i} ({b_{i}}_{[0]} \lhd \pi(a_{(3)})) \otimes S_{\HH}(\pi(a_{(2)})){b_{i}}_{[1]}\pi(a_{(4)}) \otimes S_{\G}(a_{(1)})a'_{i}a_{(5)} \\
& \overset{(\ref{eq:equality1})}{=} \sum_{i} (b_{i} \lhd \pi(a_{(3)})) \otimes S_{\HH}(\pi(a_{(2)}))\pi({a'_{i}}_{(1)})\pi(a_{(4)}) \otimes S_{\G}(a_{(1)}){a'_{i}}_{(2)}a_{(5)} \\
& \,\;=\;\, \sum_{i} (b_{i} \lhd \pi(a_{(3)})) \otimes \pi(S_{\G}(a_{(2)}){a'_{i}}_{(1)}a_{(4)}) \otimes S_{\G}(a_{(1)}){a'_{i}}_{(2)}a_{(5)} \\
& \,\;=\;\, \sum_{i} (\id \otimes \pi \otimes \id)(\id \otimes \com_{\G})((b_{i} \lhd \pi(a_{(2)})) \otimes S(a_{(1)})a'_{i}a_{(3)}) \\
& \,\;=\;\, (\id \otimes \pi \otimes \id)(\id \otimes \com_{\G})(b \blhd a)
\end{align*}
shows that $b \blhd a \in \Ind^{\G}_{\HH}(\mathcal{B})$. This means that $\lhd^{\G,\HH} := \blhd|_{\Ind^{\G}_{\HH}(\mathcal{B})}$ is a Hopf $*$-algebra action of $\pol(\G)$ on $\Ind^{\G}_{\HH}(\mathcal{B})$. We claim that the tuple $(\Ind^{\G}_{\HH}(B),\lhd^{\G,\HH},\Ind^{\G}_{\HH}(\alpha))$ yields a Yetter-Drinfeld $\G$-C*-algebra. For that, note that
\[
b_{[0]} \otimes b_{[1]} = \Ind^{\G}_{\HH}(\alpha)(b) = \sum_{i} b_{i} \otimes {a'_{i}}_{(1)} \otimes {a'_{i}}_{(2)},
\]
thus
\begin{align*}
\Ind^{\G}_{\HH}(\alpha)(b \lhd^{\G,\HH} a) & = \sum_{i}(b_{i} \lhd \pi(a_{(2)})) \otimes \com_{\G}(S(a_{(1)})a'_{i}a_{(3)}) \\
& = \sum_{i} ((b_{i} \otimes {a'_{i}}_{(1)}) \lhd^{\G,\HH} \pi(a_{(2)})) \otimes S_{\G}(a_{(1)}){a'_{i}}_{(2)}a_{(3)} \\
& = (b_{[0]} \lhd^{\G,\HH} \pi(a_{(2)})) \otimes S_{\G}(a_{(1)})b_{[1]}a_{(3)}
\end{align*}
for any $b \in \Ind^{\G}_{\HH}(\mathcal{B})$ and $a \in \pol(\G)$. It remains to show that the Yetter-Drinfeld $\G$-C*-algebra $(\Ind^{\G}_{\HH}(B),\lhd^{\G,\HH},\Ind^{\G}_{\HH}(\alpha))$ is braided commutative when $(B,\lhd,\alpha)$ is braided commutative. This follows from the computation
\begin{align*}
b_{[0]}(b' \lhd^{\G,\HH} b_{[1]}) & \;= \sum_{j} b_{[0]}\left((b'_{j} \lhd \pi({b_{[1]}}_{(2)})) \otimes S_{\G}({b_{[1]}}_{(1)})a''_{j}b_{[1]_{(3)}}\right) \\
& \;= \sum_{i,j}b_{i}(b'_{j} \lhd \pi({a'_{i}}_{(1)})) \otimes a''_{j}{a'_{i}}_{(2)}
\end{align*}
which, by \eqref{eq:equality1} and braided commutative of $(B,\lhd,\alpha)$, implies
\[
b_{[0]}(b' \lhd^{\G,\HH} b_{[1]}) = b'b
\]
for all $b = \sum_{i} b_{i} \otimes a'_{i}, b' = \sum_{j} b'_{j} \otimes a''_{j} \in \Ind^{\G}_{\HH}(\mathcal{B})$.
\end{proof}

\begin{exe}\label{ex:quotient_type_yd_algebra}
Let $\HH$ be a quantum subgroup of a compact quantum group $\G$. By Example~\ref{ex:trivial_yd_algebra}, $(\C,\lhd_{\textrm{triv}},\alpha_{\textrm{triv}})$ is a trivial braided commutative Yetter-Drinfeld $\HH$-C*-algebra. By Proposition~\ref{prop:induction_cases}, it follows that $\Ind^{\G}_{\HH}(\C) = C(\HH\backslash\G)$ and $\Ind^{\G}_{\HH}(\alpha_{\textrm{triv}}) = \com|_{C(\HH\backslash\G)}$. By Proposition~\ref{prop:induced_yd}, we have
\[
b \lhd_{\textrm{triv}}^{\G,\HH} a = (1_{\C} \lhd \pi(a_{(2)})) \otimes S_{\G}(a_{(1)})ba_{(3)} = S_{\G}(a_{(1)})ba_{(2)}
\]
for all $b \in \pol(\HH\backslash\G)$ and $x \in \pol(\G)$, i.e. $\Ind^{\G}_{\HH}(\lhd_{\textrm{triv}})$ is the adjoint action of $(\pol(\G),\com)$ restricted to $\pol(\HH\backslash\G)$. Then
\[
\Ind^{\G}_{\HH}(\C,\lhd_{\textrm{triv}},\alpha_{\textrm{triv}}) = (C(\HH\backslash\G),\lhd_{\textrm{ad}}|_{\pol(\HH\backslash\G)},\com|_{C(\HH\backslash\G)})
\]
yields a braided commutative Yetter-Drinfeld $\G$-C*-algebra called {\em the quotient type coideal arising from the quantum subgroup $\HH$ of $\G$}. In particular, if $\HH = \{*\}$, we have $\Ind^{\G}_{\HH}(\C,\lhd_{\textrm{triv}},\alpha_{\textrm{triv}}) = (C(\G),\lhd_{\textrm{ad}},\com)$.
\end{exe}

Here is now another version of the induction result involving weak unitary tensor functors.

\begin{propo}\label{prop:induced_yd_equi}
Let $\HH$ be a quantum subgroup of a compact quantum group $\G$, with surjective Hopf $*$-algebra morphism $\pi: \pol(\G) \to \pol(\HH)$, and let $(B,\lhd,\alpha)$ be a Yetter-Drinfeld $\HH$-C*-algebra. If $\psi: \varphi_{\alpha}\circ\Res^{\G}_{\HH} \iso \varphi_{\Ind(\alpha)}$ denotes the natural unitary monoidal isomorphism given by Theorem~\ref{thm:induction}, then the following diagram
\[
\xymatrix@C=5pc{
\varphi_{\alpha}(\Res^{\G}_{\HH}(v)) = \mor_{\HH}(\Res^{\G}_{\HH}(v),\alpha)\ar[r]^-{{}^{\alpha}\mathcal{A}^{\Res^{\G}_{\HH}(u)}_{\Res^{\G}_{\HH}(v)}}\ar[d]_-{\psi_{v}}^-{\iso} & \varphi_{\alpha}(\Res^{\G}_{\HH}(\bar{u} \boxtimes v \boxtimes u)) = \mor_{\HH}(\Res^{\G}_{\HH}(\bar{u} \boxtimes v \boxtimes u),\alpha)\ar[d]\ar[d]^-{\psi_{\bar{u} \boxtimes v \boxtimes u}}_-{\iso} \\
\varphi_{\Ind^{\G}_{\HH}(\alpha)}(v) = \mor_{\G}(v,\Ind^{\G}_{\HH}(\alpha))\ar[r]_{{}^{\Ind^{\G}_{\HH}(\alpha)}\mathcal{A}^{u}_{v}} & \varphi_{\Ind^{\G}_{\HH}(\alpha)}(\bar{u} \boxtimes v \boxtimes u) = \mor_{\G}(\bar{u} \boxtimes v \boxtimes u,\Ind^{\G}_{\HH}(\alpha))
}
\]
commutes for any $u, v \in \rep(\G)$.
\end{propo}

\begin{proof}
Fix $u,v \in \rep(\G)$. Consider two orthonormal basis $(\xi_{i})$ and $(\eta_{j})$ in $H_{u}$ and $H_{v}$, respectively. By Proposition~\ref{prop:induced_yd}, there is a Hopf $*$-algebra action $\lhd^{\G,\HH}: \Ind^{\G}_{\HH}(\mathcal{B}) \otimes \pol(\G) \to \Ind^{\G}_{\HH}(\mathcal{B})$, defined by
\[
(b \otimes a') \lhd^{\G,\HH} a = (b \lhd \pi(a_{(2)})) \otimes S_{\G}(a_{(1)})a'a_{(3)}
\]
for all $b \otimes a' \in \Ind^{\G}_{\HH}(\mathcal{B})$, $a \in \pol(\G)$. Thus, given $T \in \varphi_{\alpha}(\Res^{\G}_{\HH}(v))$, we have
\begin{align*}
\left({}^{\Ind^{\G}_{\HH}(\alpha)}\mathcal{A}^{u}_{v}(\psi_{v}(T))\right)(\bar{\xi}_{i} \otimes \eta_{j} \otimes \xi_{k}) & = \psi_{v}(T)(\eta_{j}) \lhd^{\G,\HH} u_{ik} = \sum_{j'} (T(\eta_{j'}) \otimes v_{j'j}) \lhd^{\G,\HH} u_{ik} \\
& = \sum_{i',j',k'} (T(\eta_{j'}) \lhd \pi(u_{i'k'})) \otimes S(u_{ii'})v_{j'j}u_{k'k} \\
& = \sum_{i',j',k'} (T(\eta_{j'}) \lhd \Res^{\G}_{\HH}(u)_{i'k'}) \otimes \bar{u}_{i'i}v_{j'j}u_{k'k} \\
& = \sum_{i',j',k'} {}^{\alpha}\mathcal{A}^{\Res^{\G}_{\HH}(u)}_{\Res^{\G}_{\HH}(v)}(T)(\bar{\xi}_{i'} \otimes \eta_{j'} \otimes \xi_{k'}) \otimes \bar{u}_{i'i}v_{j'j}u_{k'k} \\
& = \left({}^{\alpha}\mathcal{A}^{\Res^{\G}_{\HH}(u)}_{\Res^{\G}_{\HH}(v)}(T) \otimes \id\right)\delta_{\bar{u} \boxtimes v \boxtimes u}(\bar{\xi}_{i} \otimes \eta_{j} \otimes \xi_{k}) \\
& = \psi_{\bar{u} \boxtimes v \boxtimes u}({}^{\alpha}\mathcal{A}^{\Res^{\G}_{\HH}(u)}_{\Res^{\G}_{\HH}(v)}(T))(\bar{\xi}_{i} \otimes \eta_{j} \otimes \xi_{k}),
\end{align*}
for any $\xi_{i},\xi_{k} \in H_{u}$ and $\eta_{j} \in H_{v}$. The result follows directly from this. 
\end{proof}

\begin{theo}[Categorical induction for Yetter-Drinfeld C*-algebras]
Let $\HH$ be a quantum subgroup of a compact quantum group $\G$, with surjective Hopf $*$-algebra morphism $\pi: \pol(\G) \to \pol(\HH)$, and $\hat\varphi: \rep(\HH) \to \hil_{f}$ be a weak unitary tensor functor. If $\{\hat{\mathcal{A}}^{u'}_{v'}\}_{u',v' \in \rep(\HH)}$ is a compatible (resp. braided compatible) collection for $\hat\varphi$, then the linear maps
\[
\{\mathcal{A}^{u}_{v} := \hat{\mathcal{A}}^{\Res^{\G}_{\HH}(u)}_{\Res^{\G}_{\HH}(v)}: \hat\varphi(\Res^{\G}_{\HH}(v)) \to \hat\varphi(\Res^{\G}_{\HH}(\bar{u} \boxtimes v \boxtimes u))\}_{u,v \in \rep(\G)},
\]
forms a compatible (resp. braided compatible) collection for the weak unitary tensor functor
\begin{equation*}
\varphi:=\hat\varphi\circ \Res^{\G}_{\HH}: \rep(\G) \to \hil_{f}.
\end{equation*}
Moreover, there is a $\G$-equivariant $*$-isomorphism between $B_{\varphi}$ and $\Ind^{\G}_{\HH}(B_{\hat\varphi})$, which restricts to the algebraic cores, commutes the Hopf $*$-algebra actions of $(\pol(\G),\com)$, i.e. we have a Yetter-Drinfeld $\G$-C*-algebras isomorphism between $(B_{\varphi},\lhd_{\varphi},\alpha_{\varphi})$ and $\Ind^{\G}_{\HH}(B_{\hat\varphi},\lhd_{\hat\varphi},\alpha_{\hat\varphi})$.
\end{theo}
\begin{proof}
Let $\{\hat{\mathcal{A}}^{u'}_{v'}\}_{u',v' \in \rep(\HH)}$ be a compatible (resp. braided compatible)  collection for the weak unitary tensor functor $\hat\varphi$. Due to fact that $\Res^{\G}_{\HH}$ is a tensor functor, it is straightforward see that the collection $\{\mathcal{A}^{u}_{v}\}_{u,v \in \rep(\G)}$ yields a compatible (resp. braided compatible)  collection for the weak unitary tensor functor $\varphi = \hat\varphi\circ \Res^{\G}_{\HH}$.

Denote by $\hat\phi: \hat\varphi \iso \varphi_{\alpha_{\hat\varphi}}$ the natural unitary monoidal isomorphism given by Tannaka-Krein reconstruction applied to $\hat\varphi$. By Theorem~\ref{thm:induction} applied to the action $\alpha_{\hat\varphi}: B_{\hat\varphi} \to B_{\hat\varphi} \otimes C(\HH)$, there is a natural unitary monoidal isomorphism $\psi: \varphi_{\alpha_{\hat\varphi}}\circ\Res^{\G}_{\HH} \iso \varphi_{\Ind^{\G}_{\HH}(\alpha_{\hat\varphi})}$. Thus, by Theorem~\ref{theo:yd_reconstruction_na} and Proposition~\ref{prop:induced_yd}, the natural unitary monoidal isomorphism $\psi\circ\hat\phi_{\Res^{\G}_{\HH}}:\varphi \iso \varphi_{\Ind^{\G}_{\HH}(\alpha_{\hat\varphi})}$ yields a $\G$-equivariant $*$-isomorphism $\uppsi:\mathcal{B}_{\varphi} \to \Ind^{\G}_{\HH}(\mathcal{B}_{\hat\varphi})$ defined by
\[
\uppsi(\pi_{\varphi}(\bar{X} \otimes \eta_{j})) = (\hat{\phi}_{\Res^{\G}_{\HH}(v)}(X) \otimes \id)\delta_{v}(\eta_{j}) 
\]
for all $X \in \varphi(v)$ and a orthonormal basis $(\eta_{j})$ in $H_{v}$. Hence, if $u,v \in \rep(\G)$ and $(\xi_{i})$, $(\eta_{j})$ are orthonormal bases in $H_{u}$ and $H_{v}$ respectively, then
\begin{align*}
\uppsi(\pi_{\varphi}(\bar{X} \otimes \eta_{j}) \lhd_{\varphi} u_{ik}) & = \uppsi(\pi_{\varphi}(\overline{\mathcal{A}^{v}_{u}(X)} \otimes \bar{\xi}_{i} \otimes \eta_{j} \otimes \xi_{k}))\\
& = (\hat{\phi}_{\Res^{\G}_{\HH}(\bar{u} \boxtimes v \boxtimes u)}(\mathcal{A}^{u}_{v}(X)) \otimes \id)\delta_{\bar{u} \boxtimes v \boxtimes u}(\bar{\xi}_{i} \otimes \eta_{j} \otimes \xi_{k}) \\
& = ((\psi_{\bar{u} \boxtimes v \boxtimes u} \circ \hat{\phi}_{\Res^{\G}_{\HH}(\bar{u} \boxtimes v \boxtimes u)} \circ \mathcal{A}^{u}_{v})(X))(\bar{\xi}_{i} \otimes \eta_{j} \otimes \xi_{k}) \\
& = \left(\left({}^{\Ind^{\G}_{\HH}(\alpha_{\hat\varphi})}\mathcal{A}^{u}_{v} \circ \psi_{v} \circ \hat{\phi}_{\Res^{\G}_{\HH}(v)}\right)(X)\right)(\bar{\xi}_{i} \otimes \eta_{j} \otimes \xi_{k}) \\
& = {}^{\Ind^{\G}_{\HH}(\alpha_{\hat\varphi})}\mathcal{A}^{u}_{v}((\hat{\phi}_{\Res^{\G}_{\HH}(v)}(X) \otimes \id)\circ\delta_{v})(\bar{\xi}_{i} \otimes \eta_{j} \otimes \xi_{k}) \\
& = ((\hat{\phi}_{\Res^{\G}_{\HH}(v)}(X) \otimes \id)\delta_{v}(\eta_{j})) \lhd^{\G,\HH}_{\hat{\varphi}} u_{ik} \\
& = \uppsi(\pi_{\varphi}(\bar{X} \otimes \eta_{j})) \lhd^{\G,\HH}_{\hat{\varphi}} u_{ik}
\end{align*}
for all $X \in \varphi(v)$, $\eta_{j} \in H_{v}$ and $u_{ik} \in \pol(\G)$. This implies that $\uppsi$ commutes the actions of $(\pol(\G),\com)$.
\end{proof}

\section{Ergodic actions of quantum permutations}\label{sec:rigidity}

In this section, we will investigate actions of the quantum permutation group $S_{N}^{+} = \G_{N}(NC)$ on classical compact spaces. More precisely, we will address the question whether such an action can be ergodic when the space is connected. The idea is that the functor corresponding to an action on a classical compact space, that is to say on a commutative unital C*-algebra, exhibits specific features which entail restrictions on the irreducible representations of $S_{N}^{+}$ which can appear in its spectral decomposition. We will make this clearer hereafter, but we first need a few facts on the representation theory of $S_{N}^{+}$.

It was proven by T. Banica in \cite{banica1999symmetries} that the irreducible representations of $S_{N}^{+}$ for $N\geq 4$ can be labeled by the integers in a natural way. Denoting by $u^{k}$ the $k$-th one and by $H_{k}$ the finite-dimensional Hilbert space it acts upon, the tensor product is given by the formula
\begin{equation*}
u^{k}\boxtimes u^{k'} = u^{\vert k-k'\vert}\oplus u^{\vert k-k'\vert + 1}\oplus \cdots \oplus u^{k+k'-1}\oplus u^{k+k'}.
\end{equation*}
Based on this, we denote by $P_{n}^{k, k'}$ the orthogonal projection onto the subspace of $H_{k}\otimes H_{k'}$ corresponding to the subrepresentation $u^{n}$. The next lemma will be the key to our result, and makes precise the fact that some property of the spectral decomposition translates into some property of the representation theory of $S_{N}^{+}$. Hereafter, $\sigma_{k}$ will denote the flip operator on $H_{k}\otimes H_{k}$. Note that $\sigma_{k}$ has exactly two eigenvalues, namely $\pm 1$.

\begin{lem}\label{lem:sigmak}
Let $\alpha$ be an ergodic action of $S_{N}^{+}$ on a commutative unital C*-algebra $B$ and let $\F$ be the corresponding weak unitary tensor functor. If there exists $k\geqslant n/2$ such that $\F(P_{n}^{k, k})\circ\iota_{k, k}\neq 0$, then $H_{n}$ is contained in an eigenspace of $\sigma_{k}$.
\end{lem}

\begin{proof}
Recall that there is an isomorphism of linear spaces
\begin{equation*}
B \cong \bigoplus_{n\in \N}\overline{\F(u^{n})}\otimes H_{n}
\end{equation*}
and that for $\xi_{1}, \xi_{2}\in \F(u^{k})$ and $\eta_{1}, \eta_{2}\in H_{k}$, the product reads 
\begin{equation*}
(\bar{\xi}_{1}\otimes \eta_{1})(\bar{\xi}_{2}\otimes \eta_{2}) = \sum_{n\in \N}\overline{\F(P_{n}^{k, k})\circ\iota_{k, k}(\xi_{1} \otimes \xi_{2})}\otimes P_{n}^{k, k}(\eta_{1}\otimes \eta_{2})
\end{equation*}
 in the Tannaka-Krein reconstruction procedure of Theorem \ref{theo:tannakakreinpinzarirobert} and Theorem \ref{thm:tkergodic} (see \cite{N14,PR07}).  Note that we used here the fact that any irreducible subrepresentation $u^{n}\subset u^{k}\otimes u^{k}$ has multiplicity one. Because $B$ is assumed to be commutative, we have for any choices of $\xi_{1}, \xi_{2}, \eta_{1}, \eta_{2}$,
\begin{equation*}
\sum_{n\in \N}\F(P_{n}^{k, k})\circ\iota_{k, k}(\xi_{1}\otimes \xi_{2})\otimes P_{n}^{k, k}(\eta_{1}\otimes \eta_{2}) = \sum_{n\in \N}\F(P_{n}^{k, k})\circ\iota_{k, k}(\xi_{2} \otimes \xi_{1})\otimes P_{n}^{k, k}(\eta_{2}\otimes \eta_{1}).
\end{equation*}

Now, the projections $(P_{n}^{k, k})_{0\leqslant n\leqslant 2k}$ have pairwise orthogonal ranges, hence the corresponding terms in the sums are linearly independent. As a consequence, we have
\begin{equation*}
\F(P_{n}^{k, k})\circ\iota_{k, k}(\xi_{1}\otimes \xi_{2})\otimes P_{n}^{k, k}(\eta_{1}\otimes \eta_{2}) = \F(P_{n}^{k, k})\circ\iota_{k, k}(\xi_{2}\otimes \xi_{1})\otimes P_{n}^{k, k}(\eta_{2}\otimes \eta_{1})
\end{equation*}
for all $0\leqslant n\leqslant 2k$. Assume that there exists $\xi_{1}, \xi_{2}$ such that $\F(P_{n}^{k, k})\circ\iota(\xi_{1}\otimes \xi_{2})\neq 0$. Then, for any $\eta_{1}, \eta_{2}$ there exists $\lambda \in \C^{*}$ such that
\begin{align*}
\F(P_{n}^{k, k})\circ\iota_{k, k}(\xi_{1}\otimes \xi_{2}) & = \lambda^{-1} \F(P_{n}^{k, k})\circ\iota_{k, k}(\xi_{2} \otimes \xi_{1}) \\
P_{n}^{k, k}(\eta_{1}\otimes \eta_{2}) & = \lambda P_{n}^{k, k}(\eta_{2}\otimes \eta_{1})
\end{align*}
Observe that $\lambda$ does not depend on $\eta_{1}$ and $\eta_{2}$ by the first equation nor on $\xi_{1}$ and $\xi_{2}$ by the second one. Moreover, applying the property twice yields $\lambda^{2} = 1$, hence $\lambda\in \{-1, 1\}$.

If $\lambda = 1$, then we have $P_{n}^{k, k}\sigma_{k} = P_{n}^{k, k}$ and taking adjoints yields $\sigma_{k} P_{n}^{k, k} = P_{n}^{k, k}$. Using the fact that $P_{1} = (\id + \sigma_{k})/2$ is the projection onto the eigenspace of $\sigma_{k}$ corresponding to the eigenvalue $1$, we see that $P_{1} P_{n}^{k, k} = P_{n}^{k, k}$ and the claim is proven. If $\lambda = -1$, the same argument shows that $P_{n}^{k, k}$ is dominated by the projection $P_{-1}$ onto the the eigenspace of $\sigma_{k}$ for the eigenvalue $-1$.
\end{proof}

The next step is to prove that the previous condition is never satisfied as soon as $n\geqslant 2$. 
To this end, we first need to make the operator $P_{\vert^{\odot k}}$, which was already used in the proof of Theorem \ref{thm:projectivepartitionaction}, more concrete. Recall that
\begin{equation*}
P_{\vert^{\odot k}} = Q_{\vert^{\odot k}} - R_{\vert^{\odot k}} = \mathrm{id} - R_{\vert^{\odot k}},
\end{equation*}
where $R_{\vert^{\odot k}}$ is the supremum of all the projections $Q_{q}$ for $q\prec \vert^{\odot k}$.

\begin{lem}\label{lem:elementsinhighestrep}
The subspace $\mathrm{Ran}(R_{\vert^{\odot k}})\subset(\mathbb{C}^{N})^{\otimes k}$ is the span of the vectors
\begin{enumerate}[label=$(\arabic*)$]
\item $e_{i_{1}}\otimes\cdots\otimes e_{i_{k}}$, where $i_{t} = i_{t+1}$ for at least one $1\leqslant t\leqslant k-1$;
\item $\xi_{1}\otimes\cdots\otimes\xi_{k}$, where $\xi_{t}\in\mathbb{C}^{N}$ for each $1\leqslant t\leqslant k$ and there exist at least one $t$ such that
\begin{equation*}
\xi_{t} = \sum_{i=1}^{N}e_{i}.
\end{equation*}
\end{enumerate}
\end{lem}

\begin{proof}
It follows from the definition that all the vectors above belong to the range of $Q_{q}$ for some projective partition $q\prec \vert^{\odot n}$. So consider now conversely such a projective partition. We must prove that its range is contained in the span of the vectors in the statement. Because $q$ is not $\vert^{\odot n}$, it has a block which is not $\vert$. If that block is a singleton, then the range of $Q_{q}$ is spanned by vectors of the form $2)$. Otherwise, there are at least two points in the same row which are connected, and in fact in both rows since $p = p^{*}$. Let us denote by $k_{1} < k_{2}$ two such points. Because $p$ is non-crossing, all the points in between those cannot be connected to an upper point, so that we may reduce until we have either a singleton or $k_{2} = k_{1} + 1$. The first case was already treated, while in the second one the range of $Q_{q}$ will be contained in the span of the vectors of the form $1)$.
\end{proof}

We will also need more precise information on the range of $P_{n}^{k,k}$. Let $2\leqslant n\leqslant 2k-1$. We have to distinguish whether $n$ is even or odd and we will only treat the odd case, the even one being done in the same way but with simplifications (see the comments at the end of the proof of Proposition \ref{prop:notsymetric}). We therefore write $n = 2k-2l+1$. Recall from \cite[Sec 5.2]{FW16} that $h_{\square}^{l}$ denotes the partition in $NC(2l, 2l)$ where the $i$-th point in each row is connected to th $2l-i$-th one. Similarly, $h_{\boxvert}^{l}$ is the same partition except that the extreme points of each row are also connected. In a pictorial expression,
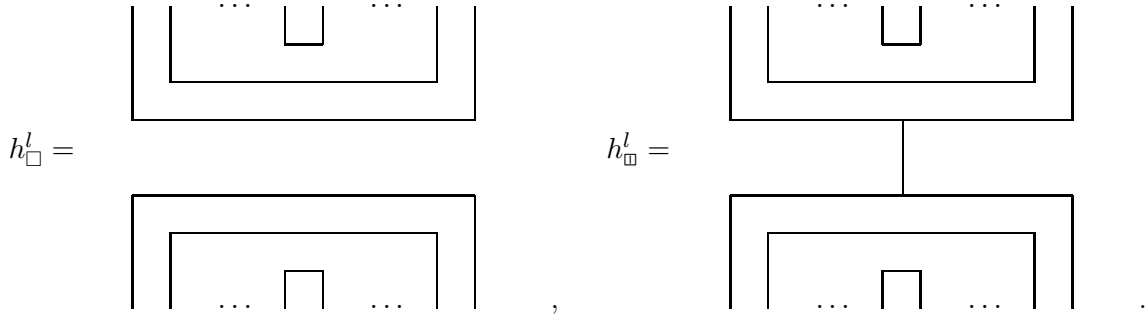
\setlength{\unitlength}{0.5cm}
\begin{center}
	\begin{picture}(15,9)
	\put(1,4){$h_{\square}^{l} =$}
	\put(3,9){\partii{3}{1}{10}}
	\put(3,9){\partii{2}{2}{9}}
	\put(3,9){\partii{1}{5}{6}}
	\put(6.5,8){$\ldots$}
	\put(10.5,8){$\ldots$}
	\put(6.5,0){$\ldots$}
	\put(10.5,0){$\ldots$}
	\put(3,0){\uppartii{3}{1}{10}}
	\put(3,0){\uppartii{2}{2}{9}}
	\put(3,0){\uppartii{1}{5}{6}}
	\end{picture}
	, 
	\begin{picture}(15,9)
	\put(1,4){$h_{\boxvert}^{l} =$}
	\put(3,9){\partii{3}{1}{10}}
	\put(3,9){\partii{2}{2}{9}}
	\put(3,9){\partii{1}{5}{6}}
	\put(6.5,8){$\ldots$}
	\put(10.5,8){$\ldots$}
	\put(6.5,0){$\ldots$}
	\put(10.5,0){$\ldots$}
	\put(3,0){\uppartii{3}{1}{10}}
	\put(3,0){\uppartii{2}{2}{9}}
	\put(3,0){\uppartii{1}{5}{6}}
	\put(8.8,3){\line(0,1){2}}
	\end{picture}.
\smallskip	
\end{center}
Using these notations, it follows from the proof of \cite[Thm 4.27]{FW16} that there exist $0 < \lambda \leqslant 1$
such that 
\begin{equation*}
P_{n}^{k,k} = \lambda\left(P_{\vert^{\odot k}}\otimes P_{\vert^{\odot k}}\right)P_{\widetilde{h}_{\boxvert}^{k, l}}\left(P_{\vert^{\odot k}}\otimes P_{\vert^{\odot k}}\right)
\end{equation*}
where $\widetilde{h}_{\boxvert}^{k,l} = \vert^{\odot(k-l)}\odot h_{\boxvert}^{l}\odot \vert^{\odot(k-l)}$. Moreover, $\left(P_{\vert^{\odot k}}\otimes P_{\vert^{\odot k}}\right)P_{\widetilde{h}_{\boxvert}^{k, l}}$ is a multiple of partial isometry. 

\begin{lem}\label{lem:rangep} We have
\begin{equation*}
\mathrm{Ran}(P_{n}^{k,k}) = \mathrm{Ran}\left((P_{\vert^{\odot k}}\otimes P_{\vert^{\odot k}})(Q_{\widetilde{h}_{\boxvert}^{k, l}} - \widetilde{R}_{\widetilde{h}_{\boxvert}^{k, l}})\right),
\end{equation*}
where $\widetilde{R}_{\widetilde{h}_{\boxvert}^{k, l}}$ is the supremum of all the projections $Q_{q}$ such that $q=\widetilde{h}_{\square}^{k, l}$ or $q = \widetilde{h}_{\boxvert}^{k, l'}$ or $\tilde{h}_{\square}^{k, l'}$ for some $l' > l$.
\end{lem}

\begin{proof}
Note that
\begin{equation*}
\mathrm{Ran}(P_{n}^{k,k}) = \mathrm{Ran}\left((P_{\vert^{\odot k}}\otimes P_{\vert^{\odot k}})P_{\widetilde{h}_{\boxvert}^{k, l}}\right) \text{ with } P_{\widetilde{h}_{\boxvert}^{k, l}} = Q_{\widetilde{h}_{\boxvert}^{k, l}} - \bigvee_{q\prec\widetilde{h}_{\boxvert}^{k, l}}Q_{q}.
\end{equation*}
It suffices to analyze the maps $(P_{\vert^{\odot k}}\otimes P_{\vert^{\odot k}})Q_{q}$ for $q\prec\widetilde{h}_{\boxvert}^{k, l}$. Note that $(P_{\vert^{\odot k}}\otimes P_{\vert^{\odot k}})Q_{q} =0$ if $q\preceq r\odot |^{\odot k}$ or $q\preceq |^{\odot k}\odot r$ for some projective partition $r\prec |^{\odot k}$. Thus the result follows from \cite[Lemma 5.7]{FW16}.
\end{proof}

Now we are going to build explicit elements in $H_{n}\subset H_{k}\otimes H_{k}$ which are not contained in an eigenspace of $\sigma_{k}$ and this will involve many indices and tensors. In order to lighten the notations and to keep the proof readable, we will first introduce some shorthand notations. We will denote as usual by $(e_{i})_{1\leqslant i\leqslant N}$ the canonical basis of $\C^{N}$.
Moreover, for any $k$-tuple $\mathbf{j} = (j_{1}, \ldots, j_{k})$ and integer $1 < l < k$, we write
\begin{equation*}
\mathbf{j}_{1} = (j_{1}, \dots, j_{k-l}); \quad \mathbf{j}_{2} = (j_{k-l+1}, \dots, j_{k}); \quad \mathbf{j}_{3} = (j_{1}, \dots, j_{l}); \quad \mathbf{j}_{4} = (j_{l+1}, \dots, j_{k}).
\end{equation*}

\begin{propo}\label{prop:notsymetric}
Let $N\geqslant 4$ and $2\leqslant n\leqslant 2k$. Then, $H_{n}$ contains vectors of $H_{k}\otimes H_{k}$ which are not symmetric nor antisymmetric.
\end{propo}
 
\begin{proof}
Let $\xi = \sum_{i=1}^{N}e_{i}$ and let $v_1 ,v_2 , v_3$ be an orthonormal basis of $\mathrm{span} \{e_i :1\leq i\leq 4 \} \cap \xi ^\bot $ given by
\[
v_1 = \frac{1}{\sqrt 2} e_1 - \frac{1}{\sqrt 2}e_2,\quad
v_2 = \frac{1}{\sqrt 6} e_1 + \frac{1}{\sqrt 6} e_2 - \frac{2}{\sqrt 6} e_3, \quad
v_3 = \frac{1}{2\sqrt 3} e_1 + \frac{1}{2\sqrt 3} e_2 +\frac{1}{2\sqrt 3} e_3 - \frac{3}{2\sqrt 3} e_4.
\]
Then
\begin{equation*}
A_{\mathbf{i}} = v_{i_{1}}\otimes \cdots\otimes v_{i_{k}}
\end{equation*}
is in the range of $P_{k}$ as soon as $i_{t}\neq i_{t+1}$ for all $1\leqslant t\leqslant k-1$ by Lemma \ref{lem:elementsinhighestrep}. For an integer $1 \leqslant l < k$ we will consider the $k$-tuples $\mathbf{i} = (1, 2, 1, 2, \cdots)$ and
\begin{equation*}
\mathbf{i}' = (i_{k}, i_{k-1}, \dots, i_{k-l+1}, 3, \widetilde{i}_{k-l-1}, 3, \widetilde{i}_{k-l-3}, \dots)
\end{equation*}
where $\widetilde{i}_{t} = 2$ if $i_{t} = 1$ and $\widetilde{i}_{t} = 1$ if $i_{t} = 2$.

We will focus on the case $n$ odd since the case $n$ even is very similar but slightly simpler (see comments at the end of the proof). We can then write $n = 2k - 2l + 1$ and define
\begin{equation*}
\eta = v_{\mathbf{i}_{1}}\otimes \widetilde{\eta}_{l}\otimes v_{\mathbf{i}'_{4}}
\end{equation*}
where
\begin{equation*}
\widetilde{\eta}_{l} = e_1 \otimes \left(\sum_{\mathbf{r}\in\{1, \dots, N\}^{l-1}}e_{\mathbf{r}}\otimes e_{\mathbf{r}^{-1}}\right)\otimes e_1 - 
e_4 \otimes \left(\sum_{\mathbf{r}\in\{1, \dots, N\}^{l-1}}e_{\mathbf{r}}\otimes e_{\mathbf{r}^{-1}}\right)\otimes e_4.
\end{equation*}
It is straightforward to check by definition that $\eta$ is in the range of $Q_{\widetilde{h}_{\boxvert}^{k, l}}$ and is orthogonal to the range of $Q_{\widetilde{h}_{\boxvert}^{k, l'}}$ for any $l' > l$. In other words, $\eta$ is in the range of $P^{k, k}_{2k-2l+1}$. Hence
\begin{equation*}
\widehat{\eta} = (P_{\vert^{\odot k}}\otimes P_{\vert^{\odot k} })(\eta)\in H_{n}\subset H_{k}\otimes H_{k}
\end{equation*}
and by definition
\begin{equation*}
\langle \widehat{\eta}, A_{\mathbf{i}}\otimes A_{\mathbf{i}'}\rangle = \langle \eta, A_{\mathbf{i}}\otimes A_{\mathbf{i}'}\rangle \quad \text{ and } \quad \langle \widehat{\eta}, A_{\mathbf{i}'}\otimes A_{\mathbf{i}}\rangle = \langle \eta, A_{\mathbf{i}'}\otimes A_{\mathbf{i}}\rangle.
\end{equation*}

On the one hand, we have
\begin{equation*}
\langle \eta, A_{\mathbf{i}'}\otimes A_{\mathbf{i}}\rangle = \langle v_{\mathbf{i_{1}}}, v_{\mathbf{i}'_{1}}\rangle\langle \widetilde{\xi}_{l}, v_{\mathbf{i}'_{2}}\otimes v_{\mathbf{i}_{3}}\rangle\langle v_{\mathbf{i}'_{4}}, v_{\mathbf{i}_{4}}\rangle = 0.
\end{equation*}
This is because $v_{\mathbf{i}'_{4}}$ has at least one tensor equal to $v_{3}$ while $v_{\mathbf{i}_{4}}$ has all its tensors equal to either $v_{1}$ or $v_{2}$. On the other hand (using $\|v_{i}\|  = 1$),
\begin{align*}
\langle \eta, A_{\mathbf{i}}\otimes A_{\mathbf{i}'}\rangle & = \langle v_{\mathbf{i_{1}}}, v_{\mathbf{i}_{1}}\rangle\langle \widetilde{\eta}_{l}, v_{\mathbf{i}_{2}}\otimes v_{\mathbf{i}'_{3}}\rangle\langle v_{\mathbf{i}'_{4}}, v_{\mathbf{i}'_{4}}\rangle \\
& =  \langle \widetilde{\eta}_{l}, v_{\mathbf{i}_{2}}\otimes v_{\mathbf{i}'_{3}}\rangle.
\end{align*}
Given $\mathbf{r} = (r_{1}, \dots, r_{l-1})$ such that
\begin{equation*}
\langle e_i \otimes e_{\mathbf{r}}\otimes e_{\mathbf{r}^{-1}}\otimes e_i , v_{\mathbf{i}_{2}}\otimes v_{\mathbf{i}'_{3}}\rangle \neq 0,
\end{equation*}
this inner product must be positive since the minus signs $-$ appear bilaterally in each summand of $v_{\mathbf{i}_{2}}\otimes v_{\mathbf{i}'_{3}}$ that has a nonzero contribution. Moreover, because $v_{i_{k-l+1}}$ is in the span of $\{e_{1}, e_{2}, e_{3}\}$, the inner product above vanishes for $e_{i} = e_{4}$. Together with the definition of $v_{i_{k-l+1}}$ and the choice of $\widetilde{\eta}_{l}$ we see that $\langle \widetilde{\eta}_{l}, v_{\mathbf{i}_{2}}\otimes v_{\mathbf{i}'_{3}}\rangle \neq 0$. Hence
\begin{equation*}
\langle \eta, A_{\mathbf{i}}\otimes A_{\mathbf{i}'}\rangle \neq 0.
\end{equation*}

As a conclusion,
\begin{equation*}
\vert\langle \widehat{\eta}, A_{\mathbf{i}}\otimes A_{\mathbf{i}'}\rangle\vert \neq \vert \langle \widehat{\eta}, A_{\mathbf{i}'}\otimes A_{\mathbf{i}}\rangle\vert
\end{equation*}
so that $\widehat{\eta}$ does not belong to any eigenspace of $\sigma_{k}$.

For $l$ even, simply replace $\widetilde{\eta}_{l}$ by
\[
\sum_{\mathbf{r}\in \{1, \dots, N\}^{l}}e_{\mathbf{r}}\otimes e_{\mathbf{r}^{-1}}
\]
if $l\geqslant 1$. For $l = 0$, $A_{\mathbf{i}}\otimes A_{\mathbf{i'}}$ is in the range of $P_{2k}$ and is not symmetric, hence the result.
\end{proof}

We are now ready to prove our theorem. Recall that a commutative unital C*-algebra is always of the form $C(X)$ for some compact space $X$. Moreover, $X$ is connected if and only if $C(X)$ contains no projection apart from $0$ and $1$.

\begin{theo}\label{thm:ergodicpermutation}
Let $X$ be a connected compact space and assume that there exists an ergodic action of $S_{N}^{+}$ on $C(X)$, for some $N\geq 4$. Then, $X$ is a point.
\end{theo}

\begin{proof}
If $N \leqslant 3$, $S_{N}^{+}$ coincides with the classical permutation group $S_{N}$. Since it is finite, the result holds trivially in that case. We therefore assume $N\geqslant 4$.  Recall that $B = C(X)$ can be decomposed as $\oplus B_{k}$. If $B _k =\{0\}$ for all $k\geqslant 1$, then we are done. Otherwise, let us  denote by $k_{0}$ the smallest nonzero integer such that $B_{k_{0}}\neq \{0\}$ and set $B' = B_{0}\oplus B_{k_{0}}$. Note that the formula for the product implies that if $\F(P_{n}^{k, k})\circ \iota_{k, k} = 0$, then the product of two elements of $B_{k}$ has no component in $B_{n}$.  By Proposition \ref{prop:notsymetric} and Lemma \ref{lem:sigmak}, $B'\cdot B'\subset B_{0}\oplus B_{1}$. Since $B_{1} = \{0\}$ or $B_{1} = B_{k_{0}}$ according to the choice of $k_0$, we always have $B'\cdot B'\subset B'$, so that $B'$ is a finite-dimensional subalgebra of $B$. Moreover, $B'$ is self-adjoint because, each representation of $S_{N}^{+}$ being self-conjugate, $B_{k}^{*} = B_{k}$ for all $k$. Thus, $B$ contains a finite-dimensional -C*-subalgebra of dimension strictly greater than one. But such an algebra always contains a non-trivial projection, contradicting connectedness.
\end{proof}
\begin{rem}
It was originally asked by D. Goswami in \cite{goswami2011rigidity} whether $S_{N}^{+}$ can act faithfully on a non-trivial connected compact space, where faithfully means that the action does not factor through any quantum subgroup. H. Huang constructed in \cite{huang2013faithful} such an action but observed that it was not ergodic. The question therefore became whether $S_{N}^{+}$ can act ergodically on a non-trivial connected compact space, and Theorem \ref{thm:ergodicpermutation} gives a negative answer.
\end{rem}

It is natural to investigate analogues of Theorem \ref{thm:ergodicpermutation} for other free easy quantum groups. It turns out that it still holds and the proof is each time a slight modification of the case of $S_{N}^{+}$, hence we will simply sketch it.

\begin{theo}
Let $\G$ be one of the following compact quantum groups :
\begin{enumerate}[label=$(\arabic*)$]
\item $O_{N}^{+}$ for $N\geqslant 3$;
\item $H_{N}^{+}$ for $N\geqslant 4$;
\item $B_{N}^{+}$ for $N\geqslant 4$;
\item $B_{N}^{\prime +}$ for $N\geqslant 4$;
\item $B_{N}^{\sharp +}$ for $N\geqslant 4$;
\item $S_{N}^{\prime +}$ for $N\geqslant 4$.
\end{enumerate}
If $\G$ acts ergodically on a compact connected space $X$, then $X$ is a point.
\end{theo}

\begin{proof}
Note that for the free easy quantum groups considered in the statement, the fusion rules are described in \cite[Sect. 5.2]{FW16}, and any irreducible subrepresentation has multiplicity at most one in the direct sum decomposition of tensor products of irreducible representations. So the analogues of Lemma \ref{lem:sigmak} still hold. For the remaining part of the proof, we give in each case the necessary modifications of the previous proofs.
\begin{enumerate}
\item For Proposition \ref{prop:notsymetric}, first note that the range of $P_{\vert^{\odot k}}$ is now simply spanned by the vectors of the form (1) in the statement of Lemma \ref{lem:elementsinhighestrep}. Moreover, any subrepresentation of $H_{k}\otimes H_{k}$ is of the form $H_{n}$ with $n = 2k - 2l$. Set 
\begin{equation*}
\widetilde{\eta}_{l} = \sum_{\mathbf{r}\in\{1, \dots, N\}^{l}}e_{\mathbf{r}}\otimes e_{\mathbf{r}^{-1}},
\end{equation*}
$v_i = e_i$ and keep the same expression for $A_{\mathbf{i}}$ and $A_{\mathbf{i}'}$. One then applies the same argument as in the proof of Proposition \ref{prop:notsymetric} and Theorem \ref{thm:ergodicpermutation}, which yields the result.

\item Recall that the irreducible representations of $H_{N}^{+}$ are indexed by the words in the free monoid over $\{0,1\}$. For such a word $w\coloneqq w_{1}w_{2}\cdots w_{k}$ with each $w_{i}\in\{0,1\}$, we write $e_{i, 0} = e_{i}$, $e_{i, 1} = e_{i}^{\otimes2}$ and define $v_{1, w_i} = e_{1, w _i}- e_{2,w_i}$ and $v_{2,w_i }, v_{3,w}$ similarly. Consider a projective partition $p\in NC_{\mathrm{even}}$ and its through-block decomposition $p =  p_{u}^{*}p_{u}$. Then, as in the proof of Lemma \ref{lem:rangep}, any $q\prec p$ in $NC_{\mathrm{even}}(k)$ is of the form $q = p_{u}^{*}rp_{u}$ for some non-crossing $r\neq|^{\odot t(p)}$. Repeating the proof of Lemma \ref{lem:elementsinhighestrep} for $r$, we may describe similarly the range of $P_{p}$ and deduce that 
\begin{equation*}
A_{\mathbf{i}, w(p)}\coloneqq v_{i_1, w(p)_1} \otimes \cdots \otimes  v_{i_m, w(p)_m} 
\end{equation*}
is an element of $H_{p}$ as soon as $1\leqslant i_{r}\leqslant N-1$ and $i_{r}\neq i_{r+1}$ for all $r$. Let us keep the indices $\mathbf{i}$ and $\mathbf{i}'$ chosen in the previous subsection for $S_{N}^{+}$ and define $\eta$ and $\widehat\eta$ similarly with $v_i$ replaced by $v_{i,w_r}$. Using this and repeating the argument in the proof of Proposition \ref{prop:notsymetric} and Theorem \ref{thm:ergodicpermutation}, we see that for a word $w$ on $\{0, 1\}$, if $B_{w}$ is the corresponding spectral subspace then $B_{w}\cdot B_{w}$ is spanned by the spectral subspaces $B_{w'}$ with $w'$ of length at most one. Denoting by $w$ a non-empty word of minimal length such that ${w}\neq \emptyset$, we conclude as in the proof of  Theorem \ref{thm:ergodicpermutation} that $B_{\emptyset}\oplus B_{w}$ is a C*-subalgebra, hence the result.

\item It is known that $B_{N}^{+}$ is isomorphic to $O_{N-1}^{+}$, hence the result follows from the first point.

\item First note that $B_{N}^{\prime +}$ is isomorphic to $B_{N}^{+}\times \Z_{2}$. As a consequence, its irreducible representations are given by pairs $(u, z^{\epsilon})$ where $u$ is an irreducible representation of $B_{N}^{+}$, $z$ is the generator of $C(\Z_{2})$ and $\epsilon\in \{0, 1\}$. As a consequence, when studying subrepresentations of $(u, \epsilon)\otimes (u, \epsilon)$ we work with fixed parity and the computation is exactly the same as in $B_{N}^{+}$.

\item The category of partitions of $B_{N}^{\sharp +}$ is the set $NC_{1,2}^{\sharp}$ of non-crossing partitions with blocks of size at most two and such that the points of the partitions can be labeled counterclockwise alternatively by $\oplus\ominus\oplus\ominus\cdots\oplus\ominus$, and each partition contains an arbitrary number of blocks of size two (pairs) connecting one $\oplus$ with one $\ominus$ and an even number of singleton blocks (see \cite{weber2012classification}).

It follows from the results of \cite{FW16} that any projective partition in $NC_{1, 2}^{\sharp}$ is a horizontal concatenation of identity partitions and double singletons. We index such partitions by the the words in the free monoid over $\{0,1\}$. More precisely, for a word $w = w_{1}w_{2}\cdots w_{r}$ over $\{0,1\}$ we set $p = p_{w_{1}}\odot p_{w_{2}}\odot\cdots\odot p_{w_{r}}$, where $p_{0} = |$ and $p_{1} = \brokenvert$. We will then denote by $k_{1}(p) < k_{2}(p) < \cdots < k_{s}(p)$ the indices $k$ with $w_{k} = 0$. We need to analyze the irreducible subrepresentations of $u_{p}\boxtimes u_{p}$. As before, with the through-block decomposition $p = p_{u}^{*}p_{u}$, these subrepresentations are indexed by the partitions 
\begin{equation*}
p\square^{l}p\coloneqq (p_{u}^{*}\odot p_{u}^{*})(\vert^{\odot(t(p)-l)}\odot h_{\square}^{l}\odot|^{\odot(t(p)-l)})(p_{u}\odot p_{u})
\end{equation*}
(the partitions with $h_{\boxvert}^{l}$ do not appear since they are not in $NC_{1, 2}^{\sharp}$). The subspace $H_{l}\subset H_{p}\otimes H_{p}$ corresponding to the subrepresentation associated with $p\square^{l}p$ is given by the range of the partial isometry $(P_{p}\otimes P_{p})P_{p\square^{l}p}$, where $P_{p\square^{l}p} = Q_{p\square^{l}p}-\vee_{q\prec p\square^{l}p}Q_{q}$. The same argument as in the proof of Lemma \ref{lem:rangep} implies that the range of $(P_{p}\otimes P_{p})P_{p\square^{l}p}$ equals that of
\[
(P_{p}\otimes P_{p})(Q_{p\square^{l}p}-\bigvee_{l'>l}Q_{p\square^{l'}p}).
\]
		
We can now easily adapt the constructions to obtain non symmetric nor antisymmetric vectors in $H_{l}\subset H_{p}\otimes H_{p}$. Set $k = t(p)$ and recall the definition of the vector $A_{\mathbf{i}}$ for a given multi-index $\mathbf{i} = (i_{1}, \dots, i_{k})$. We define a new vector $A_{\mathbf{i}}^{\sharp}\in(\mathbb{C}^{N})^{\otimes r}$, where $r$ denotes the length of the word $w(p)$, in the following way: in the $k_{s}(p)$-th component of the tensor we put the $s$-th component of $A_{\mathbf{i}}$, and in other components we put the vector $\sum_{i=1}^{N}e_{i}$. More formally, if $\sigma \in S_{r}$ denotes the permutation such that $\sigma(r-s+1) = k_{s}(p)$, then
\begin{equation*}
A_{\mathbf{i}}^{\sharp} = \sigma\left(A_{\mathbf{i}}\otimes \left(\sum_{i=1}^{N}e_{i}\right)^{\otimes r-t(p)}\right).
\end{equation*}
In the same way we define a new vector $\xi^{\sharp}\in(\mathbb{C}^{N})^{\otimes2r}$ from the vector $\xi$ appearing in the proof of Proposition \ref{prop:notsymetric}. By the arguments of the proof of Proposition \ref{prop:notsymetric}, we derive the following conclusions:
\begin{enumerate}
\item If $t(p)\geqslant 1$ and $l = 0$, then $A_{\mathbf{i}}^{\sharp}\otimes A_{\mathbf{i}'}^{\sharp}$ belongs to the range of $(P_{p}\otimes P_{p})P_{p\square^{l}p}$, and is neither symmetric nor antisymmetric;

\item If $t(p)\geqslant 2$ and $0\leqslant l\leqslant t(p)-1$, then $\xi^{\sharp}$ belongs to the range of $(P_{p}\otimes P_{p})P_{p\square^{l}p}$, and is neither symmetric nor antisymmetric.
\end{enumerate}	
Now we are ready to prove that any unital commutative C*-algebra $B$ under an ergodic action of $B_{N}^{\sharp +}$ must be finite-dimensional. First observe that there are only finitely many equivalence classes of irreducible representations $u_{p}$ with $t(p) = 1$. Indeed, if a projective partition contains two consecutive singletons in a row, then these can be removed by capping without changing the equivalence class. Assume now that there exists $p\in NC_{1, 2}^{\sharp}$ such that the component $B_{p}$ is nonzero. Consider the minimal number $t(p)$ satisfying that property and choose such a $p$ with $w(p)$ of minimal length. Then,
\[
B_{p} \cdot B_{p}\subset \spa\{B_{q}\mid t(q) \leqslant 1\}
\]
by the previous discussion and we can once again conclude.	

\item The argument is the same as for $B_{N}^{\prime +}$.
\end{enumerate}
\end{proof}

\appendix

\section{More details on Remark \ref{rem:cond}}\label{appendix:cond}
We will explain in this short appendix how to deduce the conditions in Remark  \ref{rem:cond} from \ref{cond:3}. This is well-known to experts but we do not find any explicit arguments in the literature. For the convenience of the reader we include the details here. 

We keep the notation in \ref{cond:3} and Remark  \ref{rem:cond}. Let $R_{u,v} = \iota_{u,v}(\varphi(u) \otimes \varphi(v))$ be the range of $P_{u,v}$ for $u,v \in \rep(\G)$. Recall that for $X\in \varphi (u)$, the map $S^{u,v}_{X}$ is defined by 
\[
S^{u,v}_{X}: \varphi(v) \to \varphi(u \boxtimes v),\quad
Y \mapsto \iota_{u,v}(X \otimes Y) .
\]
Its adjoint map is given by $(S^{u,v}_{X})^{*}: \varphi(u \boxtimes v) \to \varphi(v)$, where
\[
(S^{u,v}_{X})^{*}(a) = \left\{\begin{array}{ll} \displaystyle\sum_{k} \langle X_{k},X \rangle Y_{k} & ,\text{ if } a = \iota_{u,v}\left(\displaystyle\sum_{k} X_{k} \otimes Y_{k}\right) \in R_{u,v}  \\ 0 & ,\text{ if } a \in R_{u,v}^{\perp}\end{array}\right.
\]
Note that $P_{u,v} S^{u,v}_{X} = S^{u,v}_{X}$, and hence
\begin{equation}\label{eq:sandp}
(S^{u,v}_{X})^{*} P_{u,v} = (S^{u,v}_{X})^{*}.
\end{equation}
Moreover,  by \ref{cond:2} we have
\[
\varphi(\id\otimes f)\circ S^{u,v}_{X} = S^{u,v'}_{X}\circ\varphi(f),
\]
and therefore
\[
(S^{u,v}_{X})^{*}\circ\varphi(\id\otimes f^{*}) = \varphi(f^{*})\circ (S^{u,v'}_{X})^{*}
\]
for all $f\in\mor_{\G}(v,v')$.

\begin{lem}\label{prop:natural}
Let $u,v,w \in \rep(\G)$ and $X \in \varphi(u)$. It follows
\[
(S^{u,v\boxtimes w}_{X})^{*}\iota_{u \boxtimes v,w}(P_{u,v} \otimes \id) = \iota_{v,w}((S^{u,v}_{X})^{*} \otimes \id).
\]
\end{lem}
\begin{proof}
Fix $u,v,w \in \rep(\G)$ and $X \in \varphi(u)$. For $a = \sum_{k}\iota_{u,v}(X_{k} \otimes Y_{k})\in R_{u,v}$ and $Z \in \varphi(w)$, we have
\begin{align*}
(S^{u,v\boxtimes w}_{X})^{*}(\iota_{u \boxtimes v,w}(a \otimes Z)) & = \sum_{k}(S^{u,v\boxtimes w}_{X})^{*}(\iota_{u \boxtimes v,w}(\iota_{u,v}(X_{k} \otimes Y_{k}) \otimes Z)) \\
& = \sum_{k}(S^{u,v\boxtimes w}_{X})^{*}(\iota_{u, v\boxtimes w}(X_{k} \otimes \iota_{v,w}(Y_{k} \otimes Z)) \\
& = \sum_{k}\langle X_{k}, X\rangle \iota_{v,w}(Y_{k} \otimes Z) = \iota_{v,w}\left(\sum_{k}\langle X_{k}, X\rangle Y_{k} \otimes Z\right) \\
& = \iota_{v,w}((S^{u,v}_{X})^{*}(a) \otimes Z).
\end{align*}
Together with \eqref{eq:sandp} we obtain
\[
(S^{u,v\boxtimes w}_{X})^{*}\iota_{u \boxtimes v,w}(P_{u,v} \otimes \id) = \iota_{v,w}((S^{u,v}_{X})^{*}P_{u,v} \otimes \id) = \iota_{v,w}((S^{u,v}_{X})^{*} \otimes \id).
\]
\end{proof}

\begin{lem}\label{prop:PR}
Let $u,v,w \in \rep(\G)$ and $X \in \varphi(u)$. If
\begin{equation*}
P_{u\boxtimes v,w}P_{u,v\boxtimes w} \leq P_{u,v,w},
\end{equation*}
then
\[
(S^{u,v\boxtimes w}_{X})^{*}\iota_{u \boxtimes v,w} = (S^{u,v\boxtimes w}_{X})^{*}\iota_{u \boxtimes v,w}(P_{u,v}\otimes \id).
\]
\end{lem}
\begin{proof}
As explained in \cite[(3.12)]{PR07}, the assumption indeed yields that
\begin{equation*}
P_{u\boxtimes v,w}P_{u,v\boxtimes w} = P_{u,v,w}.
\end{equation*}
Take $a\in\varphi(u\boxtimes v)$, $Z\in\varphi(w)$, $X'\in\varphi(u)$ and $b\in\varphi(v\boxtimes w)$, we have
\begin{align*}
\langle \iota_{u \boxtimes v,w}(a \otimes Z),\iota_{u, v \boxtimes w}(X' \otimes b)\rangle & = \langle \iota_{u \boxtimes v,w}(a \otimes Z),\iota_{u\boxtimes v, w}\circ\iota^{*}_{u\boxtimes v,w}\circ\iota_{u, v \boxtimes w}(X' \otimes b)\rangle \\
& = \langle \iota_{u \boxtimes v,w}(a \otimes Z),P_{u\boxtimes v,w}\circ P_{u, v \boxtimes w}\circ\iota_{u, v \boxtimes w}(X' \otimes b)\rangle \\
& = \langle \iota_{u \boxtimes v,w}(a \otimes Z),P_{u,v,w}\circ\iota_{u, v \boxtimes w}(X' \otimes b)\rangle
\end{align*}
We write
\[
P_{u,v,w}\circ\iota_{u,v\boxtimes w}(X'\otimes b)=\sum_{X'' , Y' ,Z'}\iota_{u,v,w}(X''\otimes Y'\otimes Z').
\]
for some $X''\in\varphi(u)$, $Y'\in\varphi(v)$ and $Z'\in\varphi(w)$. Note that
\begin{align*}
  \langle \iota_{u \boxtimes v,w}(a \otimes Z),\iota_{u,v, w}(X'' \otimes Y' \otimes Z')\rangle
& = \langle \iota_{u \boxtimes v,w}(a \otimes Z),\iota_{u,
\boxtimes v, w}(\iota_{u,v}(X'' \otimes Y') \otimes Z')\rangle \\
& = \langle a \otimes Z,\iota_{u,v}(X'' \otimes Y') \otimes Z'\rangle \\
& = \langle a,\iota_{u,v}(X'' \otimes Y') \rangle\langle Z,Z'\rangle
\end{align*}
Then, for $a \in R_{u,v}^{\perp}$ and $Z \in \varphi(w)$,
\[
\iota_{u \boxtimes v,w}(a \otimes Z) \in R_{u,v\boxtimes w}^{\perp}.
\]
In other words,
\[
(S^{u,v\boxtimes w}_{X})^{*}\iota_{u \boxtimes v,w}((\id_{\varphi(u\boxtimes v)} - P_{u,v}) \otimes \id) = 0,
\]
as desired.
\end{proof}
Obviously the above two lemmas yields immediately the assertion in Remark \ref{rem:cond}.

\section{A Tannaka-Krein reconstruction theorem for Yetter-Drinfeld C*-algebras}\label{app:yd}

In this appendix, we will establish a general Tannaka-Krein reconstruction theorem for Yetter-Drinfeld C*-algebras in terms of weak unitary tensor functors. In the sequel \emph{we will always assume $\G$ is of Kac type}, which means that the modular properties of the Haar state are trivial, and in particular the solution $R_u$ of the conjugation equation for $u$ can be expressed by $R_u (1_{\mathbb C}) = \sum_i \bar \xi _i \otimes \xi_i$ for an orthonormal basis $(\xi_i)$ of $H_u$, as in the main body of the paper. All the statements and proofs in this appendix work for the general non-Kac setting up to standard adaptations, but this goes beyond the main scope of this paper, and hence we prefer to stick the Kac setting for simplicity. 

The following type of maps plays a central role in our reconstruction procedures.

\begin{defi}[Compatible collection]\label{def:compatible_collection}
Let $\G$ be a compact quantum group and $\varphi: \rep(\G) \to \hil_{f}$ be a weak unitary tensor functor. A collection of linear maps
\[
\{\mathcal{A}^{u}_{v}: \varphi(v) \to \varphi(\bar{u} \boxtimes v \boxtimes u)\}_{u,v \in \rep(\G)}
\]
is called {\em a compatible collection for $\varphi$} if it satisfies the following conditions for all $u,v,w \in \rep(\G)$:
\begin{enumerate}[label=\textup{(A\arabic*)},start=0]
\item\label{cond:A0} if $u$ admits a decomposition $u = \oplus_{i}u_{i}$ into a direct sum of irreducible representations with isometric intertwiners  $\omega_{i} \in \mor_{\G}(u_{i},u)$, then
\[
\sum_{i}\varphi(\bar{\omega}_{i} \xbox \id_{v} \xbox \omega_{i})\mathcal{A}^{u_{i}}_{v} = \mathcal{A}^{u}_{v};
\]

\item\label{cond:A1} $\mathcal{A}^{\triv}_{v} = \id_{\varphi(v)}$ and $\mathcal{A}^{u}_{\triv} = \varphi(R_{u})$;

\item\label{cond:A2} for any $f \in \mor_{\G}(v,w)$, the diagram
\[
\xymatrix@C=5pc{
\varphi(v)\ar[r]^-{\varphi(f)}\ar[d]_-{\mathcal{A}^{u}_{v}} & \varphi(w)\ar[d]^-{\mathcal{A}^{u}_{w}} \\
\varphi(\bar{u} \boxtimes v \boxtimes u)\ar[r]_-{\varphi(\id_{\bar{u}} \xbox f \xbox \id_{u})} & \varphi(\bar{u} \boxtimes w \boxtimes u)
}
\]
commutes;

\item\label{cond:A3} the diagram
\[
\xymatrix@C=5pc{
\varphi(v)\ar[r]^-{\mathcal{A}^{u}_{v}}\ar[d]_-{\mathcal{A}^{u \boxtimes w}_{v}} & \varphi(\bar{u} \boxtimes v \boxtimes u)\ar[d]^-{\mathcal{A}^{w}_{\bar{u} \boxtimes v \boxtimes u}} \\
\varphi(\overline{u \boxtimes w} \boxtimes v \boxtimes u \boxtimes w)\ar[r]^-{\iso}_-{\varphi(\sigma_{u,w} \xbox \id_{v,u,w})} & \varphi(\bar{w} \boxtimes \bar{u} \boxtimes v \boxtimes u \boxtimes w)
}
\]
commutes, where the intertwiner $\sigma_{u,w} \in \mor_{\G}(\overline{u \boxtimes w},\bar{w} \boxtimes \bar{u})$ denotes the canonical isomorphism ${\sigma_{u,w}(\overline{\xi \otimes \eta}) = \bar{\eta} \otimes \bar{\xi}}$ for $\xi\in H_u$ and $\eta\in H_w$;

\item\label{cond:A4} the diagram
\[
\xymatrix@C=5pc{
\varphi(v) \otimes \varphi(v') \ar[rr]^-{\iota}\ar[d]_-{\mathcal{A}^{u}_{v} \otimes \mathcal{A}^{u}_{v'}} && \varphi(v \boxtimes v')\ar[d]^-{\mathcal{A}^{u}_{v \boxtimes v'}} \\
\varphi(\bar{u} \boxtimes v \boxtimes u) \otimes \varphi(\bar{u} \boxtimes v' \boxtimes u)\ar[r]_-{\iota} & \varphi(\bar{u} \boxtimes v \boxtimes u \boxtimes \bar{u} \boxtimes v' \boxtimes u)\ar[r]_-{\varphi(\id \xbox \bar{R}^{*}_{u} \xbox \id)} & \varphi(\bar{u} \boxtimes v \boxtimes v' \boxtimes u)
}
\]
commutes;

\item\label{cond:A5} the diagram
\[
\xymatrix@C=5pc{
\varphi(v)\ar[rr]^-{J_{v}}\ar[d]_-{\mathcal{A}^{u}_{v}} && \varphi(\bar{v})\ar[d]^-{\mathcal{A}^{u}_{\bar{v}}} \\
\varphi(\bar{u} \boxtimes v \boxtimes u)\ar[r]_-{J_{\bar{u} \boxtimes v \boxtimes u}} & \varphi(\overline{\bar{u} \boxtimes v \boxtimes u})\ar[r]^-{\iso}_-{\varphi(\sigma_{\bar{u},v,u})} & \varphi(\bar{u} \boxtimes \bar{v} \boxtimes u)
}
\]
commutes, where the intertwiner $\sigma_{\bar{u},v,u} \in \mor_{\G}(\overline{\bar{u} \boxtimes v \boxtimes u},\bar{u} \boxtimes \bar{v} \boxtimes u)$ denotes the canonical isomorphism ${\sigma_{\bar{u},v,u}(\overline{\bar{\xi} \otimes \eta \otimes \chi}) = \bar{\chi} \otimes \bar{\eta} \otimes \xi}$ for $\xi,\chi\in H_u$ and $\eta\in H_v$.
\end{enumerate}
A compatible collection $\{\mathcal{A}^{u}_{v}\}_{u,v \in \rep(\G)}$ for $\varphi$ is called {\em braided compatible}, if it additionally satisfies:
\begin{enumerate}[resume*]
\item\label{cond:A6} for any $u,v \in \rep(\G)$, the diagram
\[
\xymatrix@C=5pc{
\varphi(u) \otimes \varphi(v) \ar[rr]^-{\Sigma}\ar[d]_-{\id \otimes \mathcal{A}^{u}_{v}} && \varphi(v) \otimes \varphi(u)\ar[d]^-{\iota} \\
\varphi(u) \otimes \varphi(\bar{u} \boxtimes v \boxtimes u)\ar[r]_-{\iota} & \varphi(u \boxtimes \bar{u} \boxtimes v \boxtimes u) \ar[r]_-{\varphi(\bar{R}^{*}_{u} \xbox \id_{v,u})} & \varphi(v \boxtimes u)
}
\]
commutes.
\end{enumerate}
\end{defi}


The above notion of braided compatible collection arises very naturally from the concrete setting. Indeed, let $(B,\lhd,\alpha)$ be a Yetter-Drinfeld $\G$-C*-algebra. Consider the associated weak unitary tensor functor $\varphi_{\alpha}: \rep(\G) \to \hil_{f}$ given by $\varphi_{\alpha}(u) =  \mor_{\G}(u,\alpha)$ and $\varphi_{\alpha}(f)(T) = T\circ f^{*}$ for each $u,u' \in \rep(\G)$ and $f \in \mor_{\G}(u,u')$. Recall that in this setting we have 
\[
\iota_{u,v}: \varphi_{\alpha}(u) \otimes \varphi_{\alpha}(v) \to \varphi_{\alpha}(u \boxtimes v), \quad \iota_{u,v}(T \otimes S)(\xi \otimes \xi') = T(\xi)S(\xi')
\]
and 
\[
J_{u}: \varphi_{\alpha}(u) \to \varphi_{\alpha}(\bar{u}), \quad J_{u}(T)(\bar{\xi}) = T(\xi)^{*}.
\]
For $u, v \in \rep(\G)$ and $T \in \varphi_{\alpha}(v)$, we define the linear map
\[
\tensor*[^\alpha]{\mathcal A}{_v^u} (T):\bar{H}_{u} \otimes H_{v} \otimes H_{u} \to B,\qquad
\bar\xi  \otimes \eta \otimes \xi'  \mapsto T(\eta ) \lhd ( \omega_{\xi',\xi} \otimes \id)(u).
\]
In other words if $(\xi_{i})$ is an orthonormal basis of $H_{u}$ and $(u_{ij})$ denote the corresponding matrix coefficients of $u$, then we may simply write
\[\tensor*[^\alpha]{\mathcal A}{_v^u} (T)  (\bar\xi _i  \otimes \eta \otimes \xi_j) =T(\eta)\lhd u_{ij}. \]
We have
\begin{align*}
(\alpha \circ \tensor*[^\alpha]{\mathcal A}{_v^u}(T))(\bar{\xi}_{i} \otimes \eta \otimes \xi_{k}) & = \alpha(T(\eta) \lhd u_{ik}) \\
& = (T(\eta)_{[0]} \lhd (u_{ik})_{(2)}) \otimes S((u_{ik})_{(1)})T(\eta)_{[1]}(u_{ik})_{(3)} &\text{(by \eqref{eq:yetter_drinfeld_condition})}\\
& = ((\tensor*[^\alpha]{\mathcal A}{_v^u}(T) \otimes \id) \circ \delta_{\bar{u} \boxtimes v \boxtimes u})(\bar{\xi}_{i} \otimes \eta \otimes \xi_{k}),
\end{align*}
for all $\xi_{i},\xi_{k} \in H_{u}$ and $\eta \in H_{v}$, which means that $\tensor*[^\alpha]{\mathcal A}{_v^u}(T) \in \varphi_{\alpha}(\bar{u} \boxtimes v \boxtimes u)$. Thus, we obtain a well-defined linear map
\[
\tensor*[^\alpha]{\mathcal A}{_v^u}: \varphi_{\alpha}(v) \to \varphi_{\alpha}(\bar{u} \boxtimes v \boxtimes u),\quad T\mapsto \tensor*[^\alpha]{\mathcal A}{_v^u}(T) .
\]
The justification of the definition at the beginning is the following result.

\begin{lem}\label{lem:action_collection}
The collection of linear maps $\{\tensor*[^\alpha]{\mathcal A}{_v^u}\}_{u,v \in \rep(\G)}$ yields a compatible collection for $\varphi_{\alpha}$ in the sense of definition~\ref{def:compatible_collection}. Moreover, if $(B,\lhd,\alpha)$ is braided commutative then $\{\tensor*[^\alpha]{\mathcal A}{_v^u}\}_{u,v \in \rep(\G)}$ is braided compatible.
\end{lem}

\begin{proof}[Proof for quantum groups $\G$ of Kac type]
Let us check the required conditions \ref{cond:A0}-\ref{cond:A6} in detail.
\begin{itemize}
\item[\ref{cond:A0}] Let $u = \oplus {x}$ be a decomposition into a direct sum of mutually orthogonal irreducible representations with isometric intertwiners $\omega_{x}\in\mor_\G ({x} , u)$.
For $\xi,\xi' \in H_u$ and $\eta \in H_v$, we have
\begin{align*}
\Bigg(\sum_{x}\varphi(\bar{\omega}_{x} \xbox \id_{v} \xbox \omega_{x})\, \tensor*[^\alpha]{\mathcal A}{_v^{x}}(T)\Bigg)(\bar{\xi}  \otimes \eta  \otimes \xi') & = \sum_{x}\tensor*[^\alpha]{\mathcal A}{_v^{x}}(T)(\overline{\omega^{*}_{x}(\xi )} \otimes \eta  \otimes \omega^{*}_{x}(\xi')) \\
& =  T(\eta ) \lhd \sum_{x}(\langle \;\cdot\;\omega^{*}_{x}(\xi'),\omega^{*}_{x}(\xi )\rangle \otimes \id)(u^{x}) \\
& = T(\eta ) \lhd (\omega_{\xi ',\xi } \otimes \id)(u) \\
& = \tensor*[^\alpha]{\mathcal A}{_v^u}(T)(\bar{\xi}  \otimes \eta  \otimes \xi')
\end{align*}
 
\item[\ref{cond:A1}] We have
\[
{}^{\alpha}\mathcal{A}^{u}_{\triv}(T)(\bar{\xi}_{i} \otimes \xi_{k}) = \cou(u_{ik})T(1_{\C}) = T(R^{*}_{u}(\bar{\xi}_{i} \otimes \xi_{k})) = \varphi_{\alpha}(R_{u})(T)(\bar{\xi}_{i} \otimes \xi_{k}).
\]
and
\[
{}^{\alpha}\mathcal{A}^{\triv}_{v}(T)(\eta ) = T(\eta ) = \id_{\varphi_{\alpha}(v)}(T)(\eta ).
\]

\item[\ref{cond:A2}] Let $u,v,v' \in \rep(\G)$ and $f \in \mor_{\G}(v,v')$. Then
\begin{align*}
\tensor*[^\alpha]{\mathcal A}{_{v'}^u}(\varphi_{\alpha}(f)(T))(\bar{\xi}_{i} \otimes \eta \otimes \xi_{k}) &  = \varphi_{\alpha}(f)(T)(\eta) \lhd u_{ik} = \tensor*[^\alpha]{\mathcal A}{_v^u}(T)(\bar{\xi}_{i} \otimes f^{*}(\eta) \otimes \xi_{k}) \\
& = (\varphi_{\alpha}(\id \xbox f \xbox \id)\,\tensor*[^\alpha]{\mathcal A}{_v^u}(T))(\bar{\xi}_{i} \otimes \eta \otimes \xi_{k})
\end{align*}

\item[\ref{cond:A3}] Let $u,v,w \in \rep(\G)$. If $(\chi_{i'})$ is an orthonormal basis of $H_{w}$ and $(w_{i' k'})$ denotes the corresponding matrix coefficients, then
\begin{align*}
 (\tensor*[^\alpha]{\mathcal A}{_{\bar{u} \boxtimes v \boxtimes u}^{w}}\tensor*[^\alpha]{\mathcal A}{_v^u}(T))(\bar{\chi}_{i'} \otimes \bar{\xi}_{i} \otimes \eta \otimes \xi_{k} \otimes \chi_{k'}) & = \tensor*[^\alpha]{\mathcal A}{_v^u}(T)(\bar{\xi}_{i} \otimes \eta \otimes \xi_{k}) \lhd w_{i'k'} \\
& = T(\eta) \lhd u_{ik}w_{i'k'} = T(\eta) \lhd (u \boxtimes w)_{ii'kk'},
\end{align*}
while
\begin{align*}
(\varphi_{\alpha}(\sigma_{u,w} \xbox \id_{v,u,w})\, \tensor*[^\alpha]{\mathcal A}{^{u \boxtimes w}_{v}}(T))(\bar{\chi}_{i'} \otimes \bar{\xi}_{i} \otimes \eta \otimes \xi_{k} \otimes \chi_{k'})
&= \tensor*[^\alpha]{\mathcal A}{^{u \boxtimes w}_{v}}(T)(\overline{\xi_{i} \otimes \chi_{i'}} \otimes \eta \otimes \xi_{k} \otimes \chi_{k'}) \\
& = T(\eta) \lhd (u \boxtimes w)_{ii'kk'}.
\end{align*}

\item[\ref{cond:A4}] Let $u,v,v' \in \rep(\G)$. We have
\begin{align*}
{}^{\alpha}\mathcal{A}^{u}_{v \boxtimes v'}(\iota(T \otimes T'))(\bar{\xi}_{i} \otimes \eta\otimes \eta' \otimes \xi_{k}) & = (\iota(T \otimes T'))(\eta\otimes \eta') \lhd u_{ik} = T(\eta_{j})T'(\eta') \lhd u_{ik} \\
& = \sum_{l}(T(\eta_{j}) \lhd u_{il})(T'(\eta') \lhd u_{lk}) \\
& = \sum_{l}\left(\tensor*[^\alpha]{\mathcal A}{_v^u}(T)(\bar{\xi}_{i} \otimes \eta\otimes \xi_{l})\right)\left(\tensor*[^\alpha]{\mathcal A}{_{v'}^u}(T)(\bar{\xi}_{l} \otimes \eta' \otimes \xi_{k})\right) \\
& = \iota(\tensor*[^\alpha]{\mathcal A}{_v^u}(T) \otimes \tensor*[^\alpha]{\mathcal A}{_{v'}^u}(T))\left(\bar{\xi}_{i} \otimes \eta\otimes \bar{R}_{u}(1_{\C}) \otimes \eta' \otimes \xi_{k}\right) \\
& = (\varphi_{\alpha}(\id \xbox \bar{R}^{*}_{u} \xbox \id)\iota(\tensor*[^\alpha]{\mathcal A}{_v^u}(T) \otimes \tensor*[^\alpha]{\mathcal A}{_{v'}^u}(T)))\left(\bar{\xi}_{i} \otimes \eta\otimes \eta' \otimes \xi_{k}\right).
\end{align*}

\item[\ref{cond:A5}] Let $u,v \in \rep(\G)$.
We have
\begin{align*}
(\varphi_{\alpha}(\sigma_{\bar{u},v,u})J_{\bar{u} \boxtimes v \boxtimes u}\tensor*[^\alpha]{\mathcal A}{_v^u}(T))(\bar{\xi}_{i} \otimes \bar{\eta}_{j} \otimes \xi_{k}) & = J_{\bar{u} \boxtimes v \boxtimes u}(\tensor*[^\alpha]{\mathcal A}{_v^u}(T))(\overline{\bar{\xi}_{k} \otimes \eta_{j} \otimes \xi_{i}}) \\
& = (\tensor*[^\alpha]{\mathcal A}{_v^u}(T)(\bar{\xi}_{k} \otimes \eta_{j} \otimes \xi_{i}))^{*} = (T(\eta_{j}) \lhd u_{ki})^{*} \\
& = T(\eta_{j})^{*} \lhd S(u_{ki})^{*} = J_{v}(T)(\bar{\eta}_{j}) \lhd u_{ik} \\
& = ({}^{\alpha}\mathcal{A}^{u}_{\bar{v}}J_{v}(T))(\bar{\xi}_{i} \otimes \bar{\eta}_{j} \otimes \xi_{k}).
\end{align*}

\item[\ref{cond:A6}] Let $u,v \in \rep(\G)$. We have
\begin{align*}
(\varphi_{\alpha}(\bar{R}^{*}_{u} \xbox \id_{v,u})\iota(\id \otimes \tensor*[^\alpha]{\mathcal A}{_v^u})(S \otimes T))(\eta \otimes \xi_{i}) & = (\iota(\id \otimes \tensor*[^\alpha]{\mathcal A}{_v^u})(S \otimes T))(\bar{R}_{u}(1_{\C}) \otimes \eta \otimes \xi_{i}) \\
& =\sum_{k}(\iota(\id \otimes \tensor*[^\alpha]{\mathcal A}{_v^u})(S \otimes T))(\xi_{k} \otimes \bar{\xi}_{k} \otimes \eta \otimes \xi_{i}) \\
& = \sum_{k}S(\xi_{k})(T(\eta) \lhd u_{ki}) = S(\xi_{i})_{[0]}(T(\eta) \lhd S(\xi_{i})_{[1]}) .
\end{align*}
If $(B,\lhd,\alpha)$ is braided commutative, then by \eqref{eq:braided_commutative_condition} the last line of the above equation equals $T(\eta)S(\xi_{i})$, or equivalently $(\iota_{v,u}\Sigma(S \otimes T))(\eta \otimes \xi_{i})$, whence the result. \qedhere 
\end{itemize}
\end{proof}

In the following we will show that in order to obtain a Yetter-Drinfeld structure compatible with the recovered ergodic action arising from a weak unitary tensor functor, it suffices to have a compatible collection for the weak unitary tensor functor. If this collection is also braided, then the Yetter-Drinfeld structure will be braided commutative. To this end, let us first recall briefly the reconstructed objects from Theorem \ref{theo:tannakakreinpinzarirobert} (we refer to \cite{PR07,N14} for more details). Let $\G$ be a compact quantum group and let $\varphi: \rep(\G) \to \hil_{f}$ be a weak unitary tensor functor. We have an ergodic action  $\alpha: B \to B \otimes C(\G)$ arising from this reconstruction procedure. We have already mentioned that the algebraic core $\mathcal{B}$ of $B$, as a vector space, is given by
\[
\mathcal{B} = \bigoplus_{x \in \irr(\G)} \overline{\varphi(x)} \otimes H_{x}.
\]
This algebra is indeed a compression of a larger algebra whose vector space structure is given by
\[\tilde{\mathcal{B}}  = \bigoplus_{u \in R(\G)} \overline{\varphi(u)} \otimes H_{u},
\]
where $R(\G)$ denotes some small subcategory (i.e. the objects in this subcategory form a set) containing the tensor subcategory generated by $\irr(\G)$ in $\rep(\G)$.
Moreover $\mathcal{B}$ is obtained from $\tilde{\mathcal{B}}$ via the surjective map
\[
\pi_{\varphi}: \tilde{\mathcal{B}}  \to \mathcal{B},\quad
\pi_{\varphi}(\bar{X} \otimes \xi) = \sum_{x } \overline{\varphi(\theta^*_{x})(X)} \otimes \theta^{*}_{x}(\xi)
\]
where $u \in R(\G)$, $X \in \varphi(u)$, $\xi \in H_{u}$ and $u = \oplus {x}$ is a decomposition into a direct sum of irreducible representations with intertwiners  $\theta_{x}\in\mor_{\G}( {x},u)$. The $*$-algebraic structure on $\mathcal{B} $ is given by
\begin{equation}\label{eq:starB}
\pi_\varphi(\bar X \otimes \zeta )\pi_\varphi(\bar X ' \otimes \zeta ') = \pi_\varphi (\overline{\iota(X \otimes X')} \otimes \zeta \otimes \zeta'),\quad \pi_\varphi(\bar X \otimes \zeta )^* = \pi_\varphi (\overline{J_u (X)} \otimes \bar \zeta),
\end{equation}
for all $X\in \varphi (u)$, $Y\in \varphi (v) $, $\zeta \in H_u$, $\zeta' \in H_v$ with $u,v \in R(\G)$. We will often use the following properties of $\pi_\varphi$:
\begin{enumerate}
\item  $\pi_{\varphi}\circ\pi_{\varphi} = \pi_{\varphi}$ when we regard $\mathcal{B}$ as a vector subspace of $\tilde{\mathcal{B}} $;

\item Given an intertwiner $\omega \in \mor_{\G}(u,v)$ which is a scalar multiple of an isometry, by the definition of $\pi_\varphi$ we have	 
\begin{equation}\label{eq:piw}
\pi_{\varphi}\circ(\overline{\varphi(\omega)} \otimes \id_{H_{v}})= \pi_{\varphi}\circ(\id_{\overline{\varphi(u)}} \otimes \omega^{*}), \qquad
\pi_{\varphi}\circ(\overline{\varphi(\omega^{*})} \otimes \id_{H_{u}})= \pi_{\varphi}\circ(\id_{\overline{\varphi(v)}} \otimes \omega);
\end{equation}

\item The ergodic action $\alpha : B \to B  \otimes C(\G)$ arising from $\varphi$ can be expressed by
\begin{equation}\label{eq:actionformula}
\alpha(\pi_{\varphi}(\bar{X} \otimes \xi_{j})) = \sum_{i} \pi_{\varphi}(\bar{X} \otimes \xi_{i}) \otimes u_{ij}
\end{equation}
for any $u \in R(\G)$, $X \in \varphi(u)$ and orthonormal basis $(\xi_{i})$ in $H_{u}$. 
\end{enumerate}
In the special case where $\varphi:u\mapsto H_u$ is the so-called forgetful tensor functor $\varphi_{\textrm{forg}}$, corresponding to the ergodic action $\Delta: C(\G) \to C(\G) \otimes C(\G)$, we have $\mathcal{B}  = \oplus_{x\in\irr(\G)}\bar{H}_{x} \otimes H_{x} \iso \pol(\G)$ and $\pi_{\varphi}$ is the map
\[
\pi_{\G}: \oplus_{u\in R(\G)}\bar{H}_{u} \otimes H_{u} \to \pol(\G),\quad
\pi_{\G}(\bar{\xi} \otimes \eta) = (\omega_{\xi\eta} \otimes \id)(u) ,\qquad
\forall u \in R(\G),\,\xi,\eta \in H_{u}.
\]
We may write $\pi_{\G}(\bar{\xi}_{i} \otimes \xi_{j}) = u_{ij}$ if $(\xi_{i})$ is an orthonormal basis of $H_{u}$ with corresponding matrix coefficients $(u_{ij})$ of $u$.

We are now ready for the main result of this appendix.

\begin{theo}\label{theo:action_reconstruction}
Let $\varphi: \rep(\G) \to \hil_{f}$ be a weak unitary tensor functor, $(B,\alpha)$ be the corresponding reconstructed ergodic action by Theorem \ref{theo:tannakakreinpinzarirobert}, and $\mathcal{B}$ be the algebraic core of $B$. Assume that $\{\mathcal{A}^{u}_{v}\}_{u,v \in \rep(\G)}$ is a compatible collection for $\varphi$. Define the linear map $\lhd: \mathcal{B} \otimes \pol(\G) \to \mathcal{B}$ by 
\[
(\bar{X} \otimes \chi) \lhd (\bar{\xi} \otimes \eta) = \pi_\varphi (\overline{\mathcal{A}^{u}_{v}(X)} \otimes \bar{\xi} \otimes \chi \otimes \eta)
\]
for all $\bar{X} \otimes \chi \in \overline{\varphi(v)} \otimes H_{v}$, $\bar{\xi} \otimes \eta \in \bar{H}_{u} \otimes H_{u}$ and $u,v \in \irr (\G)$. Then the tuple $(B, \lhd, \alpha)$ forms a Yetter-Drinfeld $\G$-C*-algebra. Moreover, if $\{\mathcal{A}^{u}_{v}\}_{u,v \in \rep(\G)}$ is braided, then $(B, \lhd, \alpha)$ is also braided commutative.\end{theo}
To prove the theorem we will need the following useful lemma.

\begin{lem}\label{lem:pivarphi}
With the previous notations, consider the linear map $\tilde{\lhd}:  \tilde{\mathcal{B}}  \otimes  \widetilde{\pol(\G)} \to  \tilde{\mathcal{B}} $ defined by
\[
(\bar{X} \otimes \chi) \tilde{\lhd} (\bar{\xi} \otimes \eta) = \overline{\mathcal{A}^{u}_{v}(X)} \otimes \bar{\xi} \otimes \chi \otimes \eta
\]
for all $\bar{X} \otimes \chi \in \overline{\varphi(v)} \otimes H_{v} $, $ \bar{\xi} \otimes \eta \in \bar{H}_{u} \otimes H_{u} $ with $u,v \in R(\G)$. Then
\[\pi_{\varphi}(b) \lhd \pi_{\G}(a) = \pi_{\varphi}(b \tilde{\lhd} a)\]
for all $b\in \tilde{\mathcal{B}}$ and $a\in \widetilde{\pol(\G)}$.
\end{lem}
\begin{proof}
Consider the decompositions $u = \oplus_i u_{i}$ and $v = \oplus_j u_{j}$ into direct sums of irreducible representations with isometric intertwiners $\omega_{i}\in\mor_\G ({u_{i}},{u})$ and $\theta_{j} \in\mor_\G ({v_{j}} , {v})$. Note that by \ref{cond:A2} and \ref{cond:A0} we may write
\begin{equation*}
{\varphi(\bar{\omega}^{*}_{i} \xbox \theta^{*}_{j} \xbox \omega^{*}_{k})\mathcal{A}^{u}_{v} }= \delta_{i,k} {\varphi(\bar{\omega}^{*}_{i} \xbox \id \xbox \omega^{*}_{k})\mathcal{A}^{u}_{v_{j}}\varphi(\theta^{*}_{j}) } = \delta_{i,k} {\mathcal{A}^{u_{i}}_{v_{j}} \varphi(\theta^{*}_{j}) }
\end{equation*}
Then for $a = \bar{\xi} \otimes \eta \in \bar{H}_{u} \otimes H_{u}$ and $b = \bar{X} \otimes \zeta \in \overline{\varphi(v)} \otimes H_{v}$, Equation \eqref{eq:piw} yields
\begin{align*}
\pi_{\varphi}((\bar{X} \otimes \zeta) \tilde{\lhd} (\bar{\xi} \otimes \eta)) & =  \pi_{\varphi}(\overline{\mathcal{A}^{u}_{v}(X)} \otimes \bar{\xi} \otimes \zeta \otimes \eta) \\
& = \pi_{\varphi}\left(\sum_{i,j,k}\overline{\varphi(\bar{\omega}^{*}_{i} \xbox \theta^{*}_{j} \xbox \omega^{*}_{k})\mathcal{A}^{u}_{v}(X)} \otimes \bar{\omega}^{*}_{i}\bar{\xi} \otimes \theta^{*}_{j}\zeta \otimes \omega^{*}_{k}\eta\right) 
 \\
& = \pi_{\varphi}\left(\sum_{i,j} (\overline{\varphi(\theta^{*}_{j})X} \otimes \theta^{*}_{j}\zeta) \tilde{\lhd} (\overline{\omega^{*}_{i}\xi} \otimes \omega^{*}_{i}\eta) \right) \\
& = \pi_{\varphi}(\pi_{\varphi}(b) \tilde{\lhd} \pi_{\G}(a)) = \pi_{\varphi}(b) \lhd \pi_{\G}(a). \qedhere
\end{align*}
\end{proof}

Now let us prove the desired theorem. As mentioned previously we will only discuss the Kac type case.

\begin{proof}[Proof of Theorem \ref{theo:action_reconstruction} for quantum groups $\G$ of Kac type]
We first check that  $\lhd: \mathcal{B} \otimes \pol(\G) \to \mathcal{B} $ yields an action of the Hopf $*$-algebra $\pol(\G)$ on $\mathcal B$. We need to check (1)-(3) in Definition \ref{def:hopfaction}. To this end, take $  \bar{\xi} \otimes \eta \in \bar{H}_{u} \otimes H_{u}$, $ \bar{\xi'} \otimes \eta' \in \bar{H}_{w} \otimes H_{w}$ and $ \bar{X} \otimes \zeta \in \overline{\varphi(v)} \otimes H_{v}$ for $u,v,w \in\irr(\G)$. We have
\begin{align*}
 ((\bar{X} \otimes \zeta)  {\lhd} (\bar{\xi} \otimes \eta))  {\lhd} (\bar{\xi'} \otimes \eta')
&= \pi_{\varphi}((\overline{\mathcal{A}^{u}_{v}(X)} \otimes \bar{\xi} \otimes \zeta \otimes \eta) \tilde{\lhd} (\bar{\xi'} \otimes \eta')) \ \qquad\qquad\qquad \text{(by Lemma \ref{lem:pivarphi})}\\
&= \pi_{\varphi}(\overline{\mathcal{A}^{w}_{\bar{u} \boxtimes v \boxtimes u}\mathcal{A}^{u}_{v}(X)} \otimes \bar{\xi'} \otimes \bar{\xi} \otimes \zeta \otimes \eta \otimes \eta') \\
& = \pi_{\varphi}(\overline{\varphi(\sigma_{u,w} \xbox \id)\mathcal{A}^{u \boxtimes w}_{v}(X)} \otimes \sigma_{u,w}(\overline{\xi \otimes \xi'}) \otimes \zeta \otimes \eta \otimes \eta') \quad \text{(by \ref{cond:A3})}\\
& = \pi_{\varphi}((\overline{\mathcal{A}^{u \boxtimes w}_{v}(X)} \otimes \overline{\xi \otimes \xi'} \otimes \zeta \otimes \eta \otimes \eta') 
 \qquad \qquad \text{(by Equation \eqref{eq:piw})} \\
& = \pi_{\varphi}((\bar{X} \otimes \zeta) \tilde{\lhd} (\overline{\xi \otimes \xi'} \otimes \eta \otimes \eta')) \\
& = (\bar{X} \otimes \zeta)  {\lhd}( (\bar{\xi} \otimes \eta)    (\bar{\xi'} \otimes \eta')),  \qquad\qquad\qquad \qquad \ \quad \text{(by Lemma \ref{lem:pivarphi})}
\end{align*}
and
\[
(\bar{X} \otimes \zeta) \lhd 1_{\pol(\G)}  = \pi_{\varphi}(\overline{\mathcal{A}^{\triv}_{v}(X)} \otimes \bar 1 \otimes \zeta \otimes 1) \overset{\ref{cond:A1}}{=} \bar{X} \otimes \zeta.
\]
In other words, we have
\begin{equation}\label{eq:i}\tag{HA1}
b \lhd (aa') = (b \lhd a) \lhd  a', \quad \forall a,a' \in \pol (\G), b\in \mathcal B
\end{equation}
and
\begin{equation}\label{eq:i2}\tag{HA2}
b \lhd 1 = b, \quad \forall b\in \mathcal B.
\end{equation}
Now take an orthonormal basis $(\xi_{i})$ of $H_{u}$ and consider $ \bar{\xi}_{i} \otimes \xi_{j} \in \bar{H}_{u} \otimes H_{u}$, $  \bar{X} \otimes \zeta \in \overline{\varphi(v)} \otimes H_{v}$ and $ \bar{X'} \otimes \zeta' \in \overline{\varphi(v')} \otimes H_{v'}$ for $u,v,v' \in\irr (\G)$. 
We have
\begin{align*}
& \ \quad \pi_{\varphi}((\overline{\iota(X \otimes X')} \otimes \zeta \otimes \zeta') \tilde{\lhd} (\bar{\xi}_{i} \otimes \xi_{j} )) &  \text{(by Lemma \ref{lem:pivarphi})} \\
&  =  \pi_{\varphi}(\overline{\mathcal{A}^{u}_{v \boxtimes v'}\iota(X \otimes X')} \otimes \bar{\xi}_{i} \otimes \zeta \otimes \zeta' \otimes \xi_{j}) \\
& = \pi_{\varphi}(\overline{\varphi(\id \xbox \bar{R}^{*}_{u} \xbox \id)\iota(\mathcal{A}^{u}_{v}(X) \otimes \mathcal{A}^{u}_{v'}(X'))} \otimes \bar{\xi}_{i} \otimes \zeta \otimes \zeta' \otimes \xi_{j}) 
&  \text{(by \ref{cond:A4})}\\
& = \pi_{\varphi}(\overline{\iota(\mathcal{A}^{u}_{v}(X) \otimes \mathcal{A}^{u}_{v'}(X'))} \otimes \bar{\xi}_{i} \otimes \zeta \otimes \bar{R}_{u}(1) \otimes \zeta' \otimes \xi_{j}) & \text{(by Equation \eqref{eq:piw})} \\
& = \sum_{k}\pi_{\varphi}(\overline{\mathcal{A}^{u}_{v}(X)} \otimes \bar{\xi}_{i} \otimes \zeta \otimes \xi_{k})\pi_{\varphi}(\overline{\mathcal{A}^{u}_{v'}(X')} \otimes \bar{\xi}_{k} \otimes \zeta' \otimes \xi_{j}) \\
& = \sum_{k}\pi_{\varphi}((\bar{X} \otimes \zeta) \tilde{\lhd} (\bar{\xi}_{i} \otimes \xi_{k}))\pi_{\varphi}((\bar{X'} \otimes \zeta') \tilde{\lhd} (\bar{\xi}_{k} \otimes \xi_{j})) .
\end{align*}
Recall that we view $a\coloneqq  \bar{\xi}_{i} \otimes \xi_{j}$ as the element of matrix coefficient $u_{ij}$ in the Hopf algebra $\pol(\G)$ and we use Sweedler's notations $a _{(1)} \otimes  a _{(2)} = \com( a ) = \sum_{k}  (\bar{\xi}_{i} \otimes \xi_{k}) \otimes  (\bar{\xi}_{k} \otimes \xi_{j}).$ We therefore have
\begin{equation}
\label{eq:ii}\tag{HA3}
(bb')\lhd a =  ( b  \lhd  a _{(1)})( b' \lhd  a _{(2)}),\qquad \forall a\in\pol(\G), b,b' \in \mathcal B.
\end{equation}
On the other hand,
\begin{align*}
1_{\mathcal{B}} \lhd u_{ij} 
& = \pi_{\varphi}(\overline{\mathcal{A}^{u}_{\triv}(1)} \otimes \bar{\xi}_{i} \otimes 1 \otimes \xi_{j}) \\
& =  \pi_{\varphi}(\overline{\varphi(R_{u})(1)} \otimes \bar{\xi}_{i} \otimes \xi_{j}) &\text{(by \ref{cond:A1})}\\
& = \pi_{\varphi}(\bar 1 \otimes R^{*}_{u}(\bar{\xi}_{i} \otimes \xi_{j})) &\text{(by \eqref{eq:piw})} \\
& = \delta_{i,j}\pi_{\varphi}(\bar 1 \otimes 1) = \cou(u_{ij})1_{\mathcal{B}}
\end{align*}
for every $1 \leq i,j \leq \dim(u)$.
In other words,
\begin{equation}\label{eq:ii2}\tag{HA4}
1_{\mathcal{B}} \lhd a 
= \cou(a)1_{\mathcal{B}}, \qquad
\forall a\in\pol(\G).
\end{equation}
Now we discuss the involution. With the identification $ {\pol(\G)} = \oplus_x \bar H _x \otimes H_x$, we may interpret the antipode as $S(\bar{\xi} \otimes \eta )^* = \bar{\eta} \otimes \xi$. Take $  \bar{\xi} \otimes \eta \in \bar{H}_{u} \otimes H_{u}$ and $ \bar{X} \otimes \zeta \in \overline{\varphi(v)} \otimes H_{v}$ for $u,v \in \irr (\G)$. Then,
\begin{align*}
( (\bar{X} \otimes \zeta) \lhd S(\bar{\xi} \otimes \eta)^{*})^{*} & = ( (\bar{X} \otimes \zeta) \lhd ( \bar{\eta} \otimes \xi))^{*} = \pi_{\varphi}(\overline{\mathcal{A}^{u}_{v}(X)} \otimes \bar{\eta} \otimes \zeta \otimes \xi)^{*} & \text{(by Equation \eqref{eq:starB})}\\
& = \pi_{\varphi}(\overline{J_{\bar{u} \boxtimes v \boxtimes u}(\mathcal{A}^{u}_{v}(X))} \otimes \overline{\bar{\eta} \otimes \zeta \otimes \xi}) \qquad\\
& = \pi_{\varphi}(\overline{\varphi(\sigma^{*}_{\bar{u},v,u})\mathcal{A}^{u}_{\bar{v}}(J_{v}(X))} \otimes \sigma^{*}_{\bar{u},v,u}(\bar{\xi} \otimes \bar{\zeta} \otimes \eta)) & \qquad\  \text{(by \ref{cond:A5})}\\
& = \pi_{\varphi}(\overline{\mathcal{A}^{u}_{\bar{v}}(J_{v}(X))} \otimes \bar{\xi} \otimes \bar{\zeta} \otimes \eta) \hskip0.11\textwidth & \text{(by Equation \eqref{eq:piw}}\\
& =  (\overline{J_{v}(X)} \otimes \bar{\zeta})  {\lhd} (\bar{\xi} \otimes \eta) .
\end{align*}
As a consequence, 
\begin{equation}
\label{eq:iii}\tag{HA5}
b^{*} \lhd a = (b \lhd S(a)^{*})^{*},\qquad \forall  a\in\pol(\G), b \in \mathcal B.
\end{equation}
From Equations \eqref{eq:i}-\eqref{eq:iii}, we deduce that $\lhd: \mathcal{B} \otimes \pol(\G) \to \mathcal{B} $ is an action of the Hopf $*$-algebra $\pol(\G)$ on $\mathcal B$ according to Definition \ref{def:hopfaction}.

In the following we need to check the Yetter-Drinfeld condition \eqref{eq:yetter_drinfeld_condition}. Consider two orthonormal bases $(\xi_{i})$ and $(\zeta_{k})$ of $H_{u}$ and $H_{v}$ respectively for $u,v \in \irr(\G)$, and denote by $(u_{ij})$ and $(v_{kl})$ the corresponding matrix coefficients. Let $a = \bar{\xi}_{i} \otimes \xi_{j} \in \bar{H}_{u} \otimes H_{u}$ and $b = \bar{X} \otimes \zeta_{k} \in \overline{\varphi(v)} \otimes H_{v}$. Recall the notations
\[
 b _{[0]} \otimes  b _{[1]} := \alpha( b ) = \alpha( \bar{X} \otimes \zeta_{k} ) = \sum_{k'}  (\bar{X} \otimes \zeta_{k'}) \otimes v_{k'k},
\]
and
\[
 a _{(1)} \otimes  a _{(2)} \otimes  a _{(3)} = \sum_{i',j'}  (\bar{\xi}_{i} \otimes \xi_{i'}) \otimes  (\bar{\xi}_{i'} \otimes \xi_{j'}) \otimes  (\bar{\xi}_{j'} \otimes \xi_{j}).
\]
We have
\begin{align*}
\alpha( b \lhd  a)  & =  \alpha(\pi_{\varphi}(\overline{\mathcal{A}^{u}_{v}(X)} \otimes \bar{\xi}_{i} \otimes \zeta_{k} \otimes \xi_{j}))  \\
& = \sum_{i',k',j'}\pi_{\varphi}(\overline{\mathcal{A}^{u}_{v}(X)} \otimes \bar{\xi}_{i'} \otimes \zeta_{k'} \otimes \xi_{j'}) \otimes \bar{u}_{i'i}v_{k'k}u_{j'j} \qquad \text{(by Equation \eqref{eq:actionformula})}\\
& = \sum_{i',j',k'} ((\bar{X} \otimes \zeta_{k'}) \lhd (\bar{\xi}_{i'} \otimes \xi_{j'})) \otimes S(u_{ii'})v_{k'k}u_{j'j} \\
& = ( b _{[0]} \lhd  a _{(2)}) \otimes S( a _{(1)}) b _{[1]} a_{(3)},
\end{align*}
which yields \eqref{eq:yetter_drinfeld_condition}.

Eventually we assume that $\{\mathcal{A}^{u}_{v}\}_{u,v \in \rep(\G)}$ is braided compatible and prove \eqref{eq:braided_commutative_condition}. Let $(\xi_{i})$ be an orthonormal basis of $H_{u}$ for $u\in\irr (\G)$. We take $b = \bar{X} \otimes \xi_{i} \in \overline{\varphi(u)} \otimes H_{u}$ and ${b' = \bar{Y} \otimes \zeta \in \overline{\varphi(v)} \otimes H_{v} }$ for $v\in\irr (\G)$. We have
\begin{align*}
 b _{[0]}( b' \lhd  b_{[1]}) & = \sum_{j}  (\bar{X} \otimes \xi_{j})( (\bar{Y} \otimes \zeta) \lhd  (\bar{\xi}_{j} \otimes \xi_{i})) \\
&= \sum_{j} \pi_{\varphi}(\overline{\iota(X \otimes \mathcal{A}^{u}_{v}(Y)} \otimes \xi_{j} \otimes \bar{\xi}_{j} \otimes \zeta \otimes \xi_{i}) \\
& = \pi_{\varphi}(\overline{\iota(X \otimes \mathcal{A}^{u}_{v}(Y)} \otimes R_{u}(1) \otimes \zeta \otimes \xi_{i})  \\
& = \pi_{\varphi}(\overline{\varphi(\bar{R}^{*}_{u} \xbox \id)\iota(X \otimes \mathcal{A}^{u}_{v}(Y))} \otimes \zeta \otimes \xi_{i}) 
& \text{(by Equation \eqref{eq:piw})} \\
& = \pi_{\varphi}(\overline{\iota(Y \otimes X)} \otimes \zeta \otimes \xi_{i}) & \text{(by \ref{cond:A6})} \\
&= b' b,
\end{align*}
which proves \eqref{eq:braided_commutative_condition}, as desired.
\end{proof}

Now, we summarize the previous results in the following reconstruction theorem for (braided commutative) Yetter-Drinfeld C*-algebras.

\begin{theo}\label{theo:yd_reconstruction}
Let $\varphi: \rep(\G) \to \hil_{f}$ be a weak unitary tensor functor and $(B ,\alpha )$ be the corresponding action of $\G$ arising from the Tannaka-Krein reconstruction for $\varphi$. The following  are equivalent:
\begin{enumerate}[label=\textup{(\roman*)}]
\item there exists a compatible collection $\{\mathcal{A}^{u}_{v}\}_{u,v \in \rep(\G)}$ for $\varphi$;
\item there is an action of the Hopf $*$-algebra $(\pol(\G),\com)$ on $\mathcal{B} $, $\lhd : \mathcal{B}  \otimes \pol(\G) \to \mathcal{B} $, such that the tuple $(B ,\lhd ,\alpha )$ yields a Yetter-Drinfeld $\G$-C*-algebra.
\end{enumerate}
In that case, if $\phi: \varphi \iso \varphi_{\alpha}$ denotes the natural unitary monoidal isomorphism between the weak unitary tensor functors $\varphi$ and $\varphi_{\alpha }$, then the diagram
\[
\xymatrix@C=5pc{
\varphi(v)\ar[r]^-{\mathcal{A}^{u}_{v}}\ar[d]_-{\phi_{v}}^-{\iso} & \varphi(\bar{u} \boxtimes v \boxtimes u)\ar[d]\ar[d]^-{\phi_{\bar{u} \boxtimes v \boxtimes u}}_-{\iso} \\
\mor_{\G}(v,\alpha)\ar[r]_{{}^{\alpha}\mathcal{A}^{u}_{v}} & \mor_{\G}(\bar{u} \boxtimes v \boxtimes u,\alpha)
}
\]
commutes for every $u,v \in \rep(\G)$. Moreover, $\{\mathcal{A}^{u}_{v}\}_{u,v \in \rep(\G)}$ is braided compatible if and only if $(B ,\lhd ,\alpha)$ is braided commutative.
\end{theo}
\begin{proof}
The implication (i)$\Rightarrow $(ii) follows from Theorem~\ref{theo:action_reconstruction}. On the other hand, the implication (ii)$\Rightarrow $(i) follows directly from Lemma~\ref{lem:action_collection}. 
For $v \in \rep(\G)$ and an orthonormal basis $(\eta_{j})$ in $H_{v}$, the natural isomorphism $\phi_{v} : \varphi(v) \to \varphi_{\alpha}(v)$ is given by
\[
\phi_{v}(X)(\eta_{j}) = \pi_{\varphi}(\bar{X} \otimes \eta_{j})
\]
for all $X \in \varphi(v)$ and $\eta_{j} \in H_{v}$. Then, if $(\xi_{i})$ is an orthonormal basis in $H_{u}$, we have
\begin{align*}
(({}^{\alpha}\mathcal{A}^{u}_{v})\phi_{v}(X))(\bar{\xi}_{i} \otimes \eta_{j} \otimes \xi_{k}) & = \phi_{v}(X)(\eta_{j}) \lhd u_{ik} \\
& = \pi_{\varphi}(\bar{X} \otimes \eta_{j}) \lhd \pi_{\G}(\bar{\xi}_{i} \otimes \xi_{k}) \\
& = \pi_{\varphi}(\overline{\mathcal{A}^{u}_{v}(X)} \otimes \bar{\xi}_{i} \otimes \eta_{j} \otimes \xi_{k}) \\
& = \phi_{\bar{u} \boxtimes v \boxtimes u}(\mathcal{A}^{u}_{v}(X))(\bar{\xi}_{i} \otimes \eta_{j} \otimes \xi_{k})
\end{align*}
for all $\xi_{i}, \xi_{k} \in H_{u}$ and $\eta_{j} \in H_{v}$. This implies the last part of the statement.
\end{proof}

We conclude this appendix with some examples.

\begin{exe}\label{ex:trivial_yd_algebra}
Let $\G$ be a compact quantum group. Consider the functor $\varphi_{\textrm{triv}}: \rep(\G) \to \hil_{f}$, $u \mapsto H^{\G}_{u}$, $T\in\mor_{\G}(u,v) \mapsto T|_{H^{\G}_{u}}$, where $H^{\G}_{u} := \{ \xi \in H_{u}: \delta_{u}(\xi) = \xi \otimes 1_{C(\G)}\}$ denotes the fixed point space for $u \in \rep(\G)$. Note that $H^{\G}_{u} \iso \C$ if $\triv \prec u$ and $H^{\G}_{u} =0$ otherwise, and that $H^{\G}_{u} \otimes H^{\G}_{v} \subset H^{\G}_{u \boxtimes v}$ for all $u ,v \in \rep(\G)$. This shows that $\varphi_{\textrm{triv}}$ is a weak unitary tensor functor. By the Tannaka-Krein reconstruction the corresponding algebraic core is given by $\mathcal{B}_{\varphi_{\textrm{triv}}} = \oplus_{x \in \irr(\G)}\overline{H^{\G}_{x}} \otimes H_{x} = \C$ and the corresponding action of $\G$ is given by $\alpha_{\varphi_{\textrm{triv}}}: \C \to \C \otimes C(\G)$, $\lambda \mapsto \lambda \otimes 1_{C(\G)}$. Take now the linear maps $\{\mathcal{A}^{u}_{v}\}_{u,v \in \rep(\G)}$ given by $\mathcal{A}^{u}_{v}= R_{u}$ if $\triv \prec v$ and $\mathcal{A}^{u}_{v}= 0$ otherwise. It is straightforward to see that this trivial collection $\{\mathcal{A}^{u}_{v}\}_{u,v \in \rep (\G )}$ is a braided compatible collection for the functor $\varphi_{\textrm{triv}}$. Thus the resulting tuple $(\C,\lhd_{\varphi_{\textrm{triv}}},\alpha_{\varphi_\textrm{triv}})$ yields a braided commutative Yetter-Drinfeld $\G$-$*$-algebra, where the Hopf $*$-algebra action is given by
\begin{align*}
1 \lhd_{\varphi_{\textrm{triv}}} u_{ij} & = \pi_{\varphi_{\textrm{triv}}}(\overline{\mathcal{A}^{u}_{\triv}(1)} \otimes \bar{\xi}_{i} \otimes 1 \otimes \xi_{j}) = \pi_{\varphi_{\textrm{triv}}}(\overline{R_{u}(1)} \otimes \bar{\xi}_{i} \otimes \xi_{j}) = R^{*}_{u}(\bar{\xi}_{i} \otimes \xi_{j}) = \cou(u_{ij})
\end{align*}
for every $u_{ij} \in \pol(\G)$. In other words, $\lhd_{\varphi_{\textrm{triv}}}$ is the trivial action. 
 
\end{exe}

\begin{exe}\label{ex:trivial_quotient_coideal_yd_algebra}
Let $\G$ be a compact quantum group. Consider the canonical forgetful tensor functor $\varphi_{\textrm{forg}}: \rep(\G) \to \hil_{f}$, $u \mapsto H_{u}$. By Tannaka-Krein reconstruction we have $(\mathcal{B}_{\varphi},\alpha_{\varphi}) = (\pol(\G),\com)$. Take $u,v \in \rep(\G)$ and chose $(\xi_{i})$ an orthonormal basis in $H_{u}$, then define $\mathcal{A}^{u}_{v}: \varphi(v) \to \varphi(\bar{u} \boxtimes v \boxtimes u)$ by $\eta \mapsto \sum_{i} \bar{\xi}_{i} \otimes \eta \otimes \xi_{i}$ for all $\eta \in \varphi(v)$. It is straightforward see that $\{\mathcal{A}^{u}_{v}\}_{u,v \in \rep (\G )}$ is a braided compatible collection for $\varphi$ and it is easy to see that the corresponding tuple $(C(\G),\lhd,\com)$ yields a braided commutative Yetter-Drinfeld $\G$-$*$-algebra where the Hopf $*$-algebra action $\lhd$ is  the adjoint action $b \lhd_{\textrm{ad}} a = S(a_{(1)})ba_{(2)}$ for every $a, b \in \pol(\G)$. This example is essentially an extension of Example \ref{ex:c(G)}.
\end{exe}
%

\begin{exe}
Let $\phi: \rep(\G) \to \rep(\G')$ be a monoidal equivalence between the compact quantum groups $\G$ and $\G'$. Consider the tensor functor $\varphi = \varphi_{\textrm{forg}}\circ\phi: \rep(\G) \to \hil_{f}$. By Tannaka-Krein reconstruction applied to $\varphi$, we obtain a C*-algebra $B$ with algebraic core given by $\mathcal{B} := \oplus_{x \in \irr(\G)} \bar{H}_{\phi(x)} \otimes H_{x}$ and action $\alpha: B \to B \otimes C(G)$ given by
\[
\alpha(\pi_{\varphi}(\bar{X}_{i'} \otimes \xi_{i})) = \sum_{j} \pi_{\varphi}(\bar{X}_{i'} \otimes \xi_{j}) \otimes u_{ji} = \sum_{j} \pi_{\varphi}(\bar{X}_{i'} \otimes \xi_{j}) \otimes \pi_{\G}(\bar{\xi}_{j} \otimes \xi_{i}),
\]
for each $u \in \rep(\G)$ and orthonormal basis $(X_{i'})$, $(\xi_{i})$ in $\varphi(u)$ and $H_{u}$ respectively. The C*-algebra $B$ also carries an action $\alpha': B \to C(\G') \otimes B$ defined by
\[
\alpha'(\pi_{\varphi}(\bar{X}_{i'} \otimes \xi_{i})) = \sum_{j'}S(\phi(u)_{j'i'}) \otimes \pi_{\varphi}(\bar{X}_{j'} \otimes \xi_{i}),
\]
for each $u \in \rep(\G)$ and orthonormal basis $(X_{i'})$, $(\xi_{i})$ in $\varphi(u)$ and $H_{u}$ respectively. It is easy to check that these two actions commute, i.e. $(\alpha' \otimes \id)\alpha = (\id \otimes \alpha)\alpha'$. The C*-algebra $B$ is known in the literature as the {\em linking algebra associated with the monoidal equivalence $\phi$}, see for example \cite{BDRV06}. Moreover, one can easily prove that the action $(\mathcal{B},\alpha)$ is actually {\em a Hopf-Galois extension of $\C$ over $\pol(\G)$}, i.e. an ergodic action such that the {\em canonical map} $\Gamma: \mathcal{B} \otimes \mathcal{B} \to \mathcal{B} \otimes \pol(\G)$, $b \otimes b' \mapsto (b \otimes 1)\alpha(b')$ is a bijective map.

We now want to apply our reconstruction theorem for Yetter-Drinfeld C*-algebras to the weak unitary tensor functor $\varphi$. For the simplicity of exposition we assume again that $\G$ and $\G '$ are of Kac type. Consider $u \in \rep(\G)$ and orthonormal bases $(X_{i'})$ and $(\xi_{'})$ in $H_{\phi(u)}$ and $H_{u}$ respectively. Then,
\begin{align*}
\Gamma\left(\sum_{l}\pi_{\varphi}(X_{l} \otimes \bar{\xi}_{i}) \otimes \pi_{\varphi}(\bar{X}_{l} \otimes \xi_{k})\right) 
& = \sum_{l,k'}\pi_{\varphi}(\overline{\overline{X_{l}} \otimes X_{l}} \otimes \bar{\xi}_{i} \otimes \xi_{k'}) \otimes \pi_{\G}(\bar{\xi}_{k'} \otimes \xi_{k}) \\
& = \sum_{k'}\pi_{\varphi}(\overline{\varphi(R_{u})(1)} \otimes \bar{\xi}_{i} \otimes \xi_{k'}) \otimes \pi_{\G}(\bar{\xi}_{k'} \otimes \xi_{k}) \\
& = \sum_{k'}\pi_{\varphi}(\bar 1 \otimes R^{*}_{u}(\bar{\xi}_{i} \otimes \xi_{k'})) \otimes \pi_{\G}(\bar{\xi}_{k'} \otimes \xi_{k}) \qquad\text{(by \eqref{eq:piw})} \\
& = 1_{\mathcal{B}} \otimes u_{ik}.
\end{align*}
Thus, we can set
\begin{equation}\label{eq:inverse_Gamma}
\Gamma^{-1}(1_{\mathcal{B}} \otimes u_{ik})^{(1)} \otimes \Gamma^{-1}(1_{\mathcal{B}} \otimes u_{ik})^{(2)} := \Gamma^{-1}(1_{\mathcal{B}} \otimes u_{ik}) = \sum_{l}\pi_{\varphi}(X_{l} \otimes \bar{\xi}_{i}) \otimes \pi_{\varphi}(\bar{X}_{l} \otimes \xi_{k}).
\end{equation}
for each $u \in \rep(\G)$ and orthonormal basis $(X_{i'})$ and $(\xi_{i})$ in $\varphi(u)$ and $H_{u}$, respectively. Now, since $\{\tensor*[^{\com_{\G'}}]{\mathcal A}{_{v'}^{u'}}\}_{u',v' \in \rep(\G')}$ is the canonical braided compatible collection for the functor $\varphi_{\textrm{forg}}$, by monoidal equivalence it is straightforward to see that the collection $\{\tensor*[^{\com_{\G'}}]{\mathcal A}{_{\phi(v)}^{\phi(u)}}\}_{u,v \in \rep(\G)}$ is a braided compatible collection for the functor $\varphi$. Hence, we have a Hopf $*$-algebra action $\lhd$ such that $(B,\lhd,\alpha)$ is a braided commutative Yetter-Drinfeld $\G$-C*-algebra, and $\lhd$ is given by
\begin{align*}
\pi_{\varphi}(\bar{Y}_{j'} \otimes \eta_{j}) \lhd u_{ik} & = \pi_{\varphi}(\bar{Y}_{j'} \otimes \eta_{j}) \lhd \pi_{\G}(\bar{\xi}_{i} \otimes \xi_{k}) \\
& = \pi_{\varphi}(\overline{\tensor*[^{\com_{\G'}}]{\mathcal A}{_{\phi(v)}^{\phi(u)}}(Y_{j'})} \otimes \bar{\xi}_{i} \otimes \eta_{j} \otimes \xi_{k}) \\
& = \sum_{l}\pi_{\varphi}(\overline{\overline{X_{l}} \otimes Y_{j'} \otimes X_{l}} \otimes \bar{\xi}_{i} \otimes \eta_{j} \otimes \xi_{k}) \\
& = \sum_{l}\pi_{\varphi}(X_{l} \otimes \bar{\xi}_{i})\pi_{\varphi}(\bar{Y}_{j'} \otimes \eta_{j})\pi_{\varphi}(\bar{X}_{l} \otimes \xi_{k}) \\
& = \Gamma^{-1}(1_{\mathcal{B}} \otimes u_{ik})^{(1)}\pi_{\varphi}(\bar{Y}_{j'} \otimes \eta_{j})\Gamma^{-1}(1_{\mathcal{B}} \otimes u_{ik})^{(2)}
\end{align*}
for each $u,v \in \rep(\G)$, and orthonormal basis $(Y_{j'})$, $(\eta_{j})$ and $(\xi_{i})$ in $H_{\phi(u)}$, $H_{v}$ and $H_{u}$, respectively. This action is the so-called {\em Miyashita-Ulbrich action}.
\end{exe}

\end{document}